\documentclass[10pt]{amsart}

\usepackage{setspace}

\usepackage{amssymb}   

\usepackage{amsmath}

\usepackage{mathrsfs}

\usepackage{stmaryrd}

\usepackage{enumitem}

\usepackage[all]{xy}

\usepackage{textcomp}

\usepackage{amsbsy}
\usepackage{verbatim}
\usepackage{amsthm}

\usepackage
[pdfauthor={Carl Wang-Erickson},
 pdftitle={Algebraic Families of Galois Representations},
 bookmarks=false]
{hyperref}

\usepackage{datetime}
\renewcommand{\dateseparator}{-}
\renewcommand{\today}{\the\year \dateseparator \twodigit\month
\dateseparator \twodigit\day}

\usepackage{color}
\usepackage{comment}
\usepackage[numbers, square]{natbib}

\usepackage{xcolor}
\numberwithin{equation}{section}
\newtheorem{thm}{Theorem}[section]
\newtheorem*{thm*}{Theorem}
\newtheorem*{thmA}{Theorem A}
\newtheorem*{thmB}{Theorem B}
\newtheorem*{thmC}{Theorem C}

\newtheorem{thm-defn}[thm]{Theorem/Definition}
\newtheorem{lem}[thm]{Lemma}
\newtheorem{prop}[thm]{Proposition}
\newtheorem{cor}[thm]{Corollary}

\theoremstyle{definition}
\newtheorem{defn}[thm]{Definition}
\newtheorem{defn-lem}[thm]{Definition/Lemma}

\newtheorem{eg}[thm]{Example}
\newtheorem{fact}[thm]{Fact}

\theoremstyle{remark}
\newtheorem{rem}[thm]{Remark}

\DeclareMathOperator{\ob}{ob}

\DeclareMathOperator{\Isom}{Isom}

\DeclareMathOperator{\Proj}{Proj}
\DeclareMathOperator{\Fil}{Fil}

\newcommand{\Rep}{\mathrm{Rep}}
\newcommand{\RepB}{{\overline{\mathrm{Rep}}}}
\newcommand{\CRep}{\mathcal{R}\mathrm{ep}}
\newcommand{\CRepB}{\overline{\mathcal{R}\mathrm{ep}}}

\newcommand{\CPsR}{\mathcal{P}\mathrm{sR}}

\newcommand{\PsR}{\mathrm{PsR}}

\newcommand{\tupi}{{[\underline{\pi}]}}
\newcommand{\tuep}{{[\underline{\varep}]}}

\newcommand{\ur}{\mathrm{ur}}
\newcommand{\cris}{\mathrm{cris}}
\newcommand{\st}{\mathrm{st}}
\newcommand{\pst}{\mathrm{pst}}

\newcommand{\Ac}{{A_{\text{cris}}}}
\newcommand{\Bpc}{B^+_{\text{cris}}}
\newcommand{\Bc}{B_{\mathrm{cris}}}
\newcommand{\Bpst}{B^+_{\text{st}}}
\newcommand{\Bst}{B_{\mathrm{st}}}
\newcommand{\BpdR}{B^+_{\mathrm{dR}}}
\newcommand{\BdR}{B_{\mathrm{dR}}}

\newcommand{\Acac}{A_{\text{cris}, A^\circ}}
\newcommand{\Bpcb}{B^+_{\mathrm{cris}, B}}
\newcommand{\Bpca}{B^+_{\mathrm{cris}, A}}

\newcommand{\Bpstb}{B^+_{\mathrm{st}, B}}
\newcommand{\Bpsta}{B^+_{\mathrm{st}, A}}
\newcommand{\Bsta}{B_{\st,A}}

\newcommand{\bv}{{\bold v}}

\newcommand{\Gr}{\mathrm{Gr}}
\newcommand{\gr}{\mathrm{gr}}
\DeclareMathOperator{\Fr}{Fr}

\newcommand{\Int}{\mathrm{Int}}

\newcommand{\Adm}{\mathcal{A}\mathrm{dm}}

\newcommand{\FS}{\mathcal{F}\mathcal{S}}

\newcommand{\Db}{{\bar D}}

\newcommand{\lb}{{[\![}}
\newcommand{\rb}{{]\!]}}
\newcommand{\lp}{{(\!(}}
\newcommand{\rp}{{)\!)}}

\newcommand{\ab}{\mathrm{ab}}
\newcommand{\CH}{\mathrm{CH}}

\DeclareMathOperator{\Tr}{\mathrm{Tr}}
\newcommand{\Ad}{\mathrm{Ad}}

\DeclareMathOperator{\uM}{\underline{\frM}}

\DeclareMathOperator{\Hom}{\mathrm{Hom}}
\DeclareMathOperator{\End}{\mathrm{End}}

\DeclareMathOperator{\sgn}{sgn}

\DeclareMathOperator{\Mod}{Mod}

\DeclareMathOperator{\Gal}{{Gal}}
\DeclareMathOperator{\Aut}{Aut}

\DeclareMathOperator{\Ext}{Ext}

\DeclareMathOperator{\ord}{ord}
\DeclareMathOperator{\Res}{Res}

\DeclareMathOperator{\Lie}{Lie}

\DeclareMathOperator{\rk}{rk}
\DeclareMathOperator{\Sym}{Sym}

\newcommand{\GL}{\mathrm{GL}}
\newcommand{\PGL}{\mathrm{PGL}}

\DeclareMathOperator{\Spec}{Spec}
\DeclareMathOperator{\MaxSpec}{MaxSpec}
\DeclareMathOperator{\Spf}{Spf}

\DeclareMathOperator{\ch}{\mathrm{char}}

\DeclareMathOperator{\ad}{\mathrm{ad}}

\newcommand{\cA}{{\mathcal A}}
\newcommand{\cB}{{\mathcal B}}

\newcommand{\cD}{{\mathcal D}}
\newcommand{\cE}{{\mathcal E}}
\newcommand{\cF}{{\mathcal F}}
\newcommand{\cG}{{\mathcal G}}

\newcommand{\cK}{{\mathcal K}}
\newcommand{\cL}{{\mathcal L}}
\newcommand{\cM}{{\mathcal M}}

\newcommand{\cO}{{\mathcal O}}
\newcommand{\cP}{{\mathcal P}}

\newcommand{\cR}{{\mathcal R}}

\newcommand{\cX}{{\mathcal X}}

\newcommand{\frd}{{\mathfrak d}}

\newcommand{\frmm}{{\mathfrak m}}

\newcommand{\fro}{{\mathfrak o}}
\newcommand{\frp}{{\mathfrak p}}
\newcommand{\frq}{{\mathfrak q}}

\newcommand{\bA}{{\mathbb A}}

\newcommand{\bC}{{\mathbb C}}

\newcommand{\bF}{{\mathbb F}}
\newcommand{\bG}{{\mathbb G}}

\newcommand{\bP}{{\mathbb P}}
\newcommand{\bQ}{{\mathbb Q}}

\newcommand{\bZ}{{\mathbb Z}}

\newcommand{\frA}{{\mathfrak A}}

\newcommand{\frM}{{\mathfrak M}}
\newcommand{\frN}{{\mathfrak N}}

\newcommand{\frR}{{\mathfrak R}}
\newcommand{\frS}{{\mathfrak S}}

\newcommand{\frX}{{\mathfrak X}}

\newcommand{\varep}{\varepsilon}

\newcommand{\ra}{\rightarrow}

\newcommand{\lra}{\longrightarrow}
\newcommand{\lrisom}{\buildrel\sim\over\lra}
\newcommand{\risom}{\buildrel\sim\over\ra}
\newcommand{\rinj}{\hookrightarrow}

\newcommand{\rsurj}{\twoheadrightarrow}

\newcommand{\floor}[1]{\ensuremath{\lfloor #1 \rfloor}}

\newcommand{\Zp}{{\bZ_p}}
\newcommand{\alg}{\mathrm{alg}}
\newcommand{\Fp}{{\bF_p}}
\newcommand{\phz}{\varphi}
\newcommand{\mDb}{{\frmm_\Db}}

\newcommand{\AcR}{{A_{\cris, \cR}}}

\newcommand{\fMod}{{\frM\fro\frd}}

\newcommand{\br}{{\bar\rho}}

\title[Algebraic Families of Galois Representations]{Algebraic Families of Galois Representations and Potentially Semi-Stable Pseudodeformation Rings} 
\author{Carl Wang-Erickson}
\address{Mathematics Department, Brandeis University \\
415 South Street, MS 050 \\
Waltham, MA 02453, USA}
\email{c.wang-erickson@imperial.ac.uk}
\thanks{This work was partially supported by an NSF graduate research fellowship.}

\begin{document}

\begin{abstract}
We construct and study the moduli of continuous representations of a profinite group with integral $p$-adic coefficients. We present this moduli space over the moduli space of continuous pseudorepresentations and show  that this morphism is algebraizable. When this profinite group is the absolute Galois group of a $p$-adic local field, we show that these moduli spaces admit Zariski-closed loci cutting out Galois representations that are potentially semi-stable with bounded Hodge-Tate weights and a given Hodge and Galois type. As a consequence, we show that these loci descend to the universal deformation ring of the corresponding pseudorepresentation. 
\end{abstract}

\maketitle

\tableofcontents

%\newpage

\section{Introduction}

\subsection{Overview} Mazur \cite{mazur1989} initiated the systematic study of the moduli of representations of a Galois group $G$ in terms of complete local deformation rings. For  a fixed residual representation $\bar \rho$ with coefficients in the finite residue field $\bF,$ which admits 
a universal deformation ring $R_{\bar \rho},$ the resulting moduli space $\Spf R_{\bar \rho}$ is ``purely formal'' in the sense that the underlying algebraic scheme $\Spec \bF$ is $0$-dimensional. These deformation rings have been studied extensively in recent years, playing a significant role in automorphy lifting theorems. 

In contrast, the moduli of Galois representations that are not purely formal, i.e.\ positive-dimensional algebraic families of residual representations, have been somewhat neglected. They do appear implicitly in the work of Skinner-Wiles \cite{SW1999} and Bella\"iche-Chenevier \cite{BC2009}. The space $\Ext^1_G(\bar\rho_2, \bar\rho_1)$ of extensions of two distinct irreducible residual representations $\bar\rho_1, \bar\rho_2$ of $G$ 
\begin{equation}
\label{eq:ext_fam}
%\[
\begin{pmatrix}
\bar\rho_1 & * \\
0 & \bar\rho_2
\end{pmatrix}
%\]
\end{equation} 
is the most basic example of such a residual family. The goal of this paper is to set up a general theory of families of Galois representations and to show that conditions from $p$-adic Hodge theory may be sensibly imposed on them. The ad hoc use of these ideas in \cite{SW1999} and \cite{BC2009} suggests that these spaces should have applications to modularity lifting theorems, and the study of Selmer groups.

To state our first main result, recall that a {\em pseudorepresentation} of $G$ is the data of a polynomial for each element of $G$, satisfying coherence and continuity conditions one expects from characteristic polynomials of a representation. Chenevier \cite{chen2014} has shown that a residual pseudorepresentation $\Db$ of $G$ admits only formal deformations, and that these are parameterized by a universal deformation ring $(R_\Db, \mDb)$.  Let $\CRep_\Db$ denote the groupoid which attaches to any quotient $B$ of $\bZ_p\lb t_1, \dotsc, t_n\rb \langle z_1, \dotsc, z_m\rangle$, the category of locally free $B$-modules $V_B$ equipped with a continuous linear action $\rho_B: G \ra \Aut_B(V_B)$ having residual pseudorepresentation $\Db.$

Write $\hat\psi(\rho_B)$ for the pseudorepresentation induced by a representation $(V_B, \rho_B)$.
Denote by $\rho^{ss}_\Db$ the unique semi-simple representation such that $\hat\psi(\rho^{ss}_\Db) = \Db.$ 
We say that $\Db$ is multiplicity-free if $\rho^{ss}_\Db$ has no multiplicity among its simple factors. 

\begin{thmA}[Theorem \ref{thm:PsR_background}] 
If $G$ satisfies Mazur's finiteness condition $\Phi_p,$ then $R_\Db$ is Noetherian and 
\[
\hat\psi: \CRep_\Db \lra \Spf R_\Db
\]
is a formally finite type $\Spf R_\Db$-formal algebraic stack. Moreover,
\begin{enumerate}
\item $\CRep_\Db$ arises as the $\mDb$-adic completion of a finite type $\Spec R_\Db$-algebraic stack 
\[
\psi: \Rep_\Db \lra \Spec R_\Db.
\]
\item The defect $\nu: \Spec \psi_*(\cO_{\Rep_\Db}) \ra \Spec R_\Db$ between the GIT quotient and the pseudodeformation space is a finite universal homeomorphism which is an isomorphism in characteristic zero. In particular, $\psi$ is universally closed. 
\item If $\Db$ is multiplicity-free, then $R_\Db$ is precisely the GIT quotient ring.
\end{enumerate}
\end{thmA}

The theorem should be compared with the result that $\hat\psi$ is an isomorphism when $\rho^{ss}_\Db$ is absolutely irreducible. The fact that $\hat\psi$ is algebraizable may be thought of as an interpolation of the algebraicity of each fiber of $\hat\psi$ over a pseudorepresentation $D$. This fiber consists of the representations with semi-simplification isomorphic to $\rho^{ss}_D$, and is naturally algebraic as in \eqref{eq:ext_fam} above. The proof of Theorem A uses the existence of a ``universal Cayley-Hamilton algebra'' whose representations naturally factor the continuous representations of $\CRep_\Db$. The theorem then follows from results on the moduli of representations of finitely generated algebras. 

Having constructed and algebraized these families, we prove that it is possible to impose conditions from $p$-adic Hodge theory on them, namely that they are potentially semi-stable with a given Hodge type $\bv$ and Galois type $\tau$
in the sense of \cite{pssdr}. Let $K$ be a finite extension of $\mathbb Q_p,$ with Galois group $G = G_K.$ 
We will refer to $\CRep^\square_\Db$, which is the framed version of $\CRep_\Db$ (see \S\ref{sec:fam_gal}).

\begin{thmB}[{\S\ref{subsec:universal_C}}] Let $\tau$ and $\bold v$ be fixed Galois and Hodge types. There exists a closed formal substack $\CRep^{\tau, \bv}_\Db \rinj \CRep_\Db$ such that for any finite $\bQ_p$-algebra $B$ and point $\zeta: \Spec B \ra \CRep_\Db$, $\zeta$ factors through $\CRep^{\tau, \bold v}_\Db$ if and only if the corresponding representation $V_B$ of $G_K$ is potentially semi-stable of Galois and Hodge type $(\tau, \bold v).$ Moreover, 
\begin{enumerate}
\item $\CRep^{\square, \tau, \bv}_\Db[1/p]$ is reduced, locally complete intersection, equi-dimensional, and generically formally smooth over $\bQ_p$. If we replace ``semi-stable'' with ``crystalline,'' it is everywhere formally smooth over $\bQ_p$.

\item If $\Db$ is multiplicity-free, then $\CRep^{\tau, \bold v}_\Db$ is algebraizable, i.e.\ $\CRep^{\tau, \bold v}_\Db$ is the completion of a closed substack $\Rep^{\tau, \bold v}_\Db$ of $\Rep_\Db.$ The geometric properties of (1) also apply to $\Rep^{\tau, \bv}_\Db[1/p]$, except equi-dimensionality, which applies to its framed version $\Rep^{\square, \tau,\bv}_\Db[1/p]$.
\end{enumerate}
\end{thmB}

One may also produce the $R_\Db$-algebraic closed substack $\Rep^{\tau, \bold v}_\Db$ without any condition on $\Db$ assuming an algebraization conjecture for $\psi$ (see \S\ref{subsec:FGAMS}). 

We emphasize that the methods to cut out these loci of representations are due to Kisin \cite[\S\S1-2]{pssdr} in the case of complete local coefficient rings, and that we adapt his arguments to hold over more general coefficient schemes. The geometric properties of the loci follow from results on the ring-theoretic properties of equi-characteristic zero potentially semi-stable deformation rings, principally  \cite{bellovin2014}.

The fact that $\psi$ is algebraic of finite type and universally closed can be used to produce a potentially semi-stable \emph{pseudodeformation} ring. A pseudorepresentation $D$ valued in a $p$-adic field $E$ will be said to satisfy a condition applying to representations when the associated semi-simple representation $\rho^{ss}_D$ satisfies this condition (see Definition \ref{defn:pseudo_property}).
\begin{thmC}[{\S\ref{subsec:psspsr}}]
Assuming that the algebraization $\Rep^{\tau, \bold v}_\Db$ of $\CRep^{\tau, \bold v}_\Db$ exists, the scheme-theoretic image of $\Rep^{\tau, \bold v}_\Db$ under $\psi$ corresponds to a quotient of $R_{\Db}^{\tau, \bold v}$ of $R_\Db$ which satisfies the following property: for any finite extension $E/\bQ_p$ and any point $z: \Spec E \ra \Spec R_\Db$, the corresponding pseudorepresentation $D_z: G_K \ra E$ is potentially semi-stable of Galois and Hodge type $(\tau,\bold v)$ if and only if $z$ factors through $R^{\tau, \bold v}_\Db$. Moreover, 
\begin{enumerate}
\item $R^{\tau, \bold v} := R_\Db^{\tau, \bold v}[1/p]$ is reduced for any $(\tau,\bv)$ and does not depend on the choice of $R_\Db^{\tau, \bold v}$.
\item When we replace ``semi-stable'' by ``crystalline,'' $R^{\tau, \bold v}$ is pseudo-rational (see Definition \ref{defn:pseudo-rational}); in particular, it is normal and Cohen-Macaulay. 
\end{enumerate}
\end{thmC}

We remark that the ring-theoretic properties of the potentially semi-stable pseudodeformation rings in Theorem C are deduced from the geometric properties of the families of potentially semi-stable representations in Theorem B using invariant theory: Theorem A tells us that $R^{\tau, \bold v}$ is a GIT quotient ring. The conventional techniques used to study ring-theoretic properties of Galois deformation rings in terms of Galois cohomology have not been directly applicable to study pseudodeformation rings 
$R^{\tau, \bold v}_{\bar D}$ or $R^{\tau, \bold v}$. The author intends to report on this in future work. 

Having shown that conditions from $p$-adic Hodge theory cut out a Zariski-closed condition on algebraic families of local Galois representations, we end the paper with a discussion of the corresponding constructions for families of global Galois representations, and pseudorepresentations. We remark that the correct notion of ``a global pseudorepresentation that is locally potentially semi-stable'' is more restrictive than ``a global pseudorepresentation such that its restriction to each decomposition group over $p$ is potentially semi-stable.'' This is well-illustrated through the explicit example of a 2-dimensional global ordinary pseudodeformation ring, which we discuss in \S\ref{subsec:ordinary_psspsr}. These ordinary pseudodeformation rings are compared to Hecke algebras in \cite{WWE2015, WWE2015a}. 

As a final point, we emphasize that Theorem A is based on a study of the moduli of representations of a finitely generated associative algebra over a Noetherian ring in \S2. Theorem A is deduced from this study by ``removing the topology'' from the representation theory of profinite groups. The conclusions of Theorem A may be viewed as generalizations, allowing for the profinite topology and non-zero characteristic, of parts of the investigations of Le Bruyn \cite{lebruyn2008, lebruyn2012} (building on \cite{procesi1987}) in non-commutative algebraic geometry. 

\subsection{Summary Outline}  

Section 2 discusses the geometry of the moduli spaces of $d$-dimensional representations $\Rep^d_R$ and pseudorepresentations $\PsR^d_R$ of an associative algebra $R$, especially with reference to the natural map $\psi: \Rep^d_R \ra \PsR^d_R$ associating a representation to its induced pseudorepresentation. The main idea pursued is that the adjoint action of $\GL_d$ on the scheme of framed representations $\Rep^{\square,d}_R$, whose associated quotient stack is $\Rep^d_R$, has GIT quotient nearly equal to $\PsR^d_R$. In order to establish this, we draw geometric and algebraic consequences of Chenevier's work on pseudorepresentations \cite{chen2014}. First, we establish that the GIT quotient and $\PsR^d_R$ naturally have identical geometric points because each set of geometric points naturally corresponds to isomorphism classes of semi-simple representations. We then introduce the notion of universal Cayley-Hamilton quotient, which factors the representations of $R$. Using the theory of polynomial identity rings to derive finiteness properties of the Cayley-Hamilton quotient, we show that the discrepancy between the GIT quotient and $\PsR^d_R$ is finite. 

To conclude Section 2, we augment the theory of generalized matrix algebras of \cite[\S1.3]{BC2009} so that it functions well in arbitrarily small characteristic, attaching a canonical pseudorepresentation to a generalized matrix algebra. We also discover that when $R$ is a generalized matrix algebra, the associated space of pseudorepresentations is precisely the GIT quotient. 

In Section 3 we study the map $\psi$ in the setting of continuous representation theory of a profinite group $G$, so that we take coefficients in formal schemes over $\bZ_p$. The key result is Proposition \ref{prop:E(G)}, namely that the universal Cayley-Hamilton quotient $E(G)_\Db$, which factors the representations of $G$ with residual pseudorepresentation $\Db$, is finite as a module over $R_\Db$ and that its adic topology as an $R_\Db$-module is equivalent to the topology induced by $G$. Consequently, the moduli space of representations of $E(G)_\Db$ over $\Db$ is a finite-type $R_\Db$-algebraic model $\Rep_\Db$ for the formal moduli space $\CRep_\Db$ of representations of $G$ with residual pseudorepresentation $\Db$. Adding the results of Section 2, we get Theorem A. We then discuss how Theorem A implies that formal GAGA holds for $\psi$ in certain cases. 

In Section 4 we begin our study of potentially semi-stable representations of $G = G_K$. We adapt the methods of Kisin \cite[\S1]{pssdr} to cut out a locus of representations with $E$-height $\leq h$ within the universal families $\CRep^\square_\Db$. The point of the generalization is that coefficients must now be allowed to be quotients of $R_\Db\langle z_1, \dotsc, z_a\rangle$, while the coefficients rings were taken to be local in \emph{loc.\ cit.}  Along the way, we expand the allowable coefficients in the theory of Fontaine \cite[\S1.2]{fontaine1990}, drawing an equivalence between continuous representations of $G_{K_\infty}$ with arbitrary discrete coefficients, and \'{e}tale $\phz$-modules. The work of Emerton and Gee \cite{EG2015} will interpolate these families, as there exist families of \'{e}tale $\phz$-modules larger than those that admit a Galois representation. We then construct a projective subscheme of an affine Grassmannian parameterizing lattices of $E$-height $\leq h$ (i.e.\ Kisin modules) in the \'{e}tale $\phz$-module, and produce a characteristic zero period map relating a family of $G_{K_\infty}$-representations to an family of Kisin modules. 

Section 5 continues with the next part of Kisin's method \cite[\S2]{pssdr}, descending, in families, the comparison of a Kisin module to a $G_{K_\infty}$-representation down to a comparison of a $(\phz,N)$-module to a $G_K$-representation. This comparison is valid over a certain locus, and Section 6 begins with the conclusion that this locus consists of exactly those $G_K$-representations that are semi-stable with Hodge-Tate weights in $[0,h]$. Then we cut out connected components corresponding to a given $p$-adic Hodge type or potential semi-stability with a certain Galois type. Theorem B follows from applying these constructions to a universal family of representations and algebraizing these closed subschemes using formal GAGA for $\psi$. Finally, geometric properties of these spaces in characteristic 0 are then deduced from existing results on their local rings at closed points. 

In Section 7, we apply Theorems A and B to cut out potentially semi-stable pseudodeformation rings as the scheme-theoretic image of the potentially semi-stable locus in $\Rep_\Db$ under $\psi$, proving Theorem C. Works of Alper \cite{alper2013, alper2014} and Schoutens \cite{schoutens2008} allow us to deduce the ring-theoretic properties of potentially crystalline pseudodeformation rings using invariant theory. We discuss representations of the Galois group $G_{F,S}$ of a number field and cut out loci of representations and pseudorepresentations which are potentially semi-stable at decomposition groups over $p$. There are subtleties in this definition, which we illustrate through an example of ordinary pseudorepresentations. 

\subsection{Acknowledgements} 

We wish to thank Mark Kisin and recognize his influence in two capacities, firstly as the originator of the $p$-adic Hodge theoretic methods and ideas in this paper, and secondly for his suggestion to examine the geometric relationship between moduli spaces of Galois representations and moduli spaces of pseudorepresentations. This was begun in the author's Ph.D.\ thesis under his supervision. The influence of Ga\"etan Chenevier will also be clear to the reader. It is also a pleasure to thank Brian Conrad, Barry Mazur, Ga\"etan Chenevier, Jo\"el Bella\"iche, Preston Wake, Rebecca Bellovin, David Zureick-Brown, and Brandon Levin for helpful discussions related to this work. We also thank the anonymous referees for their thorough readings and helpful comments. Part of this work was completed with support from the National Science Foundation in the form of a graduate research fellowship. They have our thanks. Finally, we are grateful for the support and hospitality of the mathematics departments at Harvard University and Brandeis University. 

\section{Moduli of Representations of a Finitely Generated Group or Algebra}
\label{sec:Rep}

Let $A$ be a commutative Noetherian ring, let $R$ be an associative but not necessarily commutative $A$-algebra, and let $d \geq 1$ be an integer. We will often assume that $R$ is finitely generated over $A$ (cf.\ \cite[\S1.6.2]{MR2001}), e.g.\ the main result Theorem \ref{thm:psi_main}. For example, we may have $R = A[G]$ for some finitely generated group $G$. We will study the moduli of $d$-dimensional representations of $R$ relative to the space of $d$-dimensional pseudorepresentations of $R$, ultimately showing in Theorem \ref{thm:psi_main} that they almost form an adequate moduli space when $R$ is finitely generated. Later, in \S\ref{sec:fam_gal}, we will apply this study to continuous representations of a profinite group. 

\subsection{Moduli Spaces of Representations and Pseudorepresentations}

With $A$, $R$, and $d$ as above and $S = \Spec A$, here are the moduli groupoids we will consider. 
\begin{defn}
\label{defn:rep}
\begin{enumerate}
\item Define the functor on $S$-schemes $\Rep^{\square,d}_R$ by setting
\[
X \mapsto \{\cO_X\text{-algebra homomorphisms } R \otimes_{\cO_S} \cO_X \lra M_d(X).
\]
\item Define the $S$-groupoid $\Rep_R^d$ by setting
\[
\begin{aligned}
\ob \Rep^d_R(X) &= \{V / X \text{ rank } d \text{ vector bundle,} \\
& \cO_X\text{-algebra homomorphism } R \otimes_{\cO_S} \cO_X \lra \End_{\cO_X}(V) \},
\end{aligned}
\]
with the natural $\cO_X$-linear, $R$-equivariant isomorphisms of such objects.
\item Define the $S$-groupoid $\RepB_R^d$ by setting
\[
\begin{aligned}
\ob \RepB^d_R(X) &= \{\cE \text{ a rank } d^2\ \cO_X\text{-Azumaya algebra,} \\
& \cO_X\text{-algebra homomorphism } R \otimes_{\cO_S} \cO_X \lra \cE \},
\end{aligned}
\]
with the natural $\cO_X$-linear, $R$-equivariant isomorphisms of such objects.
\end{enumerate}
\end{defn}

One can check that $\Rep^{\square, d}_R$ is representable by an affine scheme which is finite type over $S$ if $R$ is finitely generated over $A$. It has been studied extensively, especially when $A$ is an algebraically closed field of characteristic zero (see e.g.\ \cite{LM1985, procesi1987, lebruyn2008, lebruyn2012}). Also, $\Rep^d_R$ is equivalent to the algebraic quotient stack $[\Rep^{\square,d}_R/\GL_d]$ by the adjoint action of $\GL_d$ just as we also have $\RepB^d_R \cong [\Rep^{\square,d}_R/\PGL_d]$. While $\RepB^d_R$ has the advantage that it often has representable loci, we will focus on $\Rep^d_R$ because it carries a universal vector bundle. 

We will be interested in the geometry of $\Rep^d_R$ relative to the moduli space of $d$-dimensional pseudorepresentations of $R$. We will use the notion of pseudorepresentation due to Chenevier \cite{chen2014}, following previous notions due to Wiles \cite{wiles1988} and Taylor \cite{taylor1991}. He uses the notion of a multiplicative polynomial law due to Roby \cite{roby1, roby2}.

\begin{defn}[{\cite[\S1.5]{chen2014}}]
\label{defn:psrep_initial}
A $d$-dimensional \emph{pseudorepresentation} $D$ of $R$ over $A$ is a homogenous multiplicative polynomial law $D : R \ra A$, i.e.\ an association of each commutative $A$-algebra $B$ to a function
\[
D_B: R \otimes_A B \lra B
\]
satisfying the following conditions:
\begin{enumerate}
\item $D_B$ is multiplicative and unit-preserving (but not necessarily additive),
\item $D_B$ is homogenous of degree $d$, i.e.
\[\forall \ b \in B, \forall\ x \in R \otimes_A B, \quad D_B(bx) = b^d D_B(x), \text{ and}\]
\item $D$ is functorial on $A$-algebras, i.e.\ for any commutative $A$-algebras $B \ra B'$, the diagram
\[\xymatrix{
R \otimes_A B \ar[r]^(.6){D_B} \ar[d] & B \ar[d] \\
R \otimes_A B' \ar[r]^(.6){D_{B'}} & B'
}
\]
commutes.
\end{enumerate}
We define $\PsR^d_R(B)$ to be the set of $d$-dimensional pseudorepresentations $D: R \otimes_A B \ra B$. It is evident that $\PsR^d_R$ is a functor on $A$-algebras.
\end{defn}

A pseudorepresentation may be thought of as an ensemble of characteristic polynomials, one for each element of $R$, satisfying compatibility properties as if they came from a representation of $R$. For $r \in R$, its characteristic polynomial $\chi(r,t) \in A[t]$ is given by $D_{A[t]}(t-r)$ and is written
\begin{equation}
\label{eq:chi_poly}
\chi(r,t) = t^d - \Lambda_{d-1}^D(r)t^{d-1} + \dotsm + (-1)^d \Lambda_d^D(r) = t^d + \sum_{i=1}^d (-1)^i\Lambda_i^D(r) t^{d-i}.
\end{equation}
Indeed, the $\chi(r,t)$ for $r \in R$ characterize the pseudorepresentation \cite[Lem.\ 1.12(ii)]{chen2014}. 

Any $B$-valued representation $(V,\rho) \in \Rep^d_R(B)$ of $R$ induces a pseudorepresentation, denoted $\psi(V)$, given by composition of $\rho: R \otimes_A B \ra \End_B(V)$ with the determinant map $\det: \End_B(V) \ra B$. This is easily checked to be functorial in $A$-algebras and therefore defines a morphism 
\[
\psi: \Rep^d_R \lra \PsR^d_R.
\]
There also exist analogous maps to $\PsR^d_R$ from $\Rep^{\square, d}_R$ and $\RepB^d_R$. The usual notion of characteristic polynomial of a representation coincides with the characteristic polynomial of the representation's induced pseudorepresentation. 

Base changes of $\psi$ have a natural interpretation as follows. 
\begin{defn}
\label{defn:D_compatible}
With $R$ and $A$ as above and any $B$-valued $d$-dimensional pseudorepresentation $D: R \otimes_A B \ra B$ of $R$, we say that a $d$-dimensional representation $(\rho,V)$ of $R$ with coefficients in a $B$-algebra $C$ is \emph{compatible with $D$} if $\det \circ \rho = D \otimes_B C$. We denote the resulting $\Spec B$-groupoid by $\Rep_{R,D}$ and note that $\Rep_{R,D} = \Rep^d_R \times_{\PsR^d_R} \Spec B$, where $\Spec B \ra \PsR^d_R$ corresponds to the pseudorepresentation $D$.
\end{defn}

Chenevier, following the work of Roby \cite{roby1, roby2}, proved that the functor of $d$-dimensional pseudorepresentations of $R$
\[
\PsR^d_R : B \mapsto \{d\text{-dimensional pseudorepresentations } D: R \otimes_A B \lra B\}
\]
is representable by an affine scheme $\PsR^d_R = \Spec \Gamma^d_A(R)^\ab$ (see \cite[Prop.\ 1.6]{chen2014}), so that there exists a universal $d$-dimensional pseudorepresentation of $R$ 
\[
D^u: R \otimes_A \Gamma^d_A(R)^\ab \lra \Gamma^d_A(R)^\ab.
\]
When $R$ is finitely generated over $A$, $\Gamma^d_A(R)^\ab$ is also finitely generated over $A$ \cite[Prop.\ 2.38]{chen2014}. 

The notion of a kernel of a pseudorepresentation provides a first step toward our goal of understanding $\psi$. 
\begin{defn}
\label{defn:kernel}
The \emph{kernel} $\ker(D)$ of a pseudorepresentation $D: R \ra A$ is a two-sided ideal of elements $r \in R$ such that for all $A$-algebras $B$ and all $r' \in R \otimes_A B$, the characteristic polynomial $\chi(rr', t) \in B[t]$ is equal to $t^d$. Call $D$ \emph{faithful} if $\ker(D) = 0$.
\end{defn}
See \cite[\S1.17]{chen2014} for further properties of the kernel, among them being the fact that the quotient algebra $R/\ker(D)$ is the minimal quotient through which $D$ factors. Moreover, in the case that $A$ is an algebraically closed field, the surjection $R \ra R/\ker(D)$ realizes the representation $\rho^{ss}_D$ of the following  
\begin{thm}[{\cite[Thm.\ A]{chen2014}}]
\label{thm:chenevier_A}
Let $\bar k$ be an algebraically closed field and let $D: R \ra \bar k$ be a $d$-dimensional pseudorepresentation. Then there exists a unique (up to isomorphism) $d$-dimensional semi-simple representation $\rho^{ss}_D: R \ra M_d(\bar k)$ such that $\psi(\rho^{ss}_D) = D$.
\end{thm}

From this theorem, we know that there is a unique semi-simple representation in each geometric fiber of $\psi$. Moreover, it is precisely the orbits (under the adjoint action of $\GL_d$) of geometric points corresponding to semi-simple representations in $\Rep^{\square,d}_R(\bar k)$ which are closed orbits (cf.\ \cite[\S II.4.5, Prop.]{kraft1982}). It is equivalent to say that each fiber $\psi^{-1}(D)$ of a geometric point $D \in \PsR^d_R(\bar k)$ has a unique closed point corresponding to the representation $\rho^{ss}_D$. We summarize these facts:
\begin{cor}
\label{cor:ss_to_psr}
If $R$ is finitely generated over $A$, the morphism $\psi: \Rep^d_R \ra \PsR^d_R$ induces an isomorphism of sets from the closed geometric points of $\Rep^d_R$ to the geometric points of $\PsR^d_R$. The inverse map sends $D \in \PsR^d_R(\bar k)$ to the semi-simple representation $\rho^{ss}_D$, where the geometric points of the fiber $\psi^{-1}(D)$ correspond to representations with Jordan-H\"older factors identical to $\rho^{ss}_D$.
\end{cor}

In \S\ref{subsec:IT}, we will refine Corollary \ref{cor:ss_to_psr} using geometric invariant theory. This will rely in part upon understanding what base extensions make a pseudorepresentation become realizable as the determinant of a representation. That is, we are seeking a version of Theorem \ref{thm:chenevier_A} where $k \neq \bar k$. 

Toward this goal, we first recall the following theorem of Chenevier. To state it, we need the following definitions. For any algebraic field extension $K/k$, we write $K^s$ for the maximal separable extension of $k$ in $K$. Let $(f_i,q_i)$ be the exponent of a field extension $k_i/k$, i.e.\ $f_i = [k_i^s: k]$ and $q_i$ is the least power of $p = \ch k$ such that $k_i^{q_i} \subset k_i^s$. Given any simple $k$-algebra $S_i$ with center $k_i$, there is a canonical $n_i f_i q_i$-dimensional pseudorepresentation, denoted $\det_{S_i}: S_i \ra k$ and defined as follows: $\det_{S_i}$ is the composition of the standard reduced norm $S_i \ra k_i$ followed by the $q_i$-Frobenius map $F_{q_i}: k_i \ra k^s_i$ followed by the standard field-theoretic norm $k^s_i \ra k$. 
\begin{thm}[{\cite[Thm.\ 2.16]{chen2014}}]
\label{thm:chen_2.16}
Let $k$ be a field and let $R$ be a $k$-algebra with a $d$-dimensional pseudorepresentation $D: R \ra k$. Then $R/\ker(D)$ is a semi-simple $k$-algebra of the form 
\[
R/\ker(D) \lrisom \prod_{i=1}^s S_i
\]
where $S_i$ is a simple $k$-algebra with center $k_i$ and dimension $n_i^2$ over $k_i$. There are unique integers $m_i$ such that
\[
D = \prod_{i=1}^s {\det}_{S_i}^{m_i} \qquad \text{and} \qquad d = \sum_{i=1}^s m_i n_i q_i f_i.  
\]

$R/\ker(D)$ is finite-dimensional over $k$ if any of the following conditions are satisfied.
\begin{enumerate}
\item $k$ is a perfect field, 
\item $d < \ch(k)$ or $\ch(k) = 0$, 
\item $\ch(k) > 0$ and $[k : k^p] < \infty$. 
\end{enumerate}
\end{thm}
Notice that $R/\ker(D)$ is finite dimensional if and only if $[k_i : k] < \infty$ for all $i$. 

\begin{rem}
Below in Corollary \ref{cor:cond_4}, we give another condition under which $R/\ker(D)$ is finite-dimensional, namely, that $R$ is finitely generated. 
\end{rem} 
 
We consider our refinement to be a corollary of Chenevier's result, and we use the same notation. 
\begin{cor}
\label{cor:ss_degree}
Assume that $R/\ker(D)$ is finite-dimensional over $k$. There exists a field extension $k'/k$ of degree bounded by $\prod_{i=1}^s [k_i : k]n_i/q_i$ such that $D \otimes_k k'$ is realizable, i.e.\ there exists $\rho: R/\ker(D) \ra M_{d\times d}(k')$ such that $\psi(\rho) = D \otimes_k k'$. The extension $k'/k$ may be taken to be a separable extension if and only if $k_i/k$ are simple field extensions for all $i$, and in this case there exists a choice of separable $k'/k$ such that $[k' : k] \leq \sum_{i=1}^s f_i n_i \leq d$. In addition,
\begin{enumerate}
\item When either of conditions (1) or (2) of Theorem \ref{thm:chen_2.16} is true, $k'/k$ may be taken to be a separable extension. 
\item Assume that $k'/k$ may be taken to be a separable extension.  If the Brauer groups of the fields $k_i$ have no elements of order $\leq d$ other than the identity, then $k'/k$ may be taken to be the trivial extension. 
\end{enumerate}
\end{cor}

\begin{proof}
The question of realizing $\det_{S_i}$ as the induced pseudorepresentation of a representation after some scalar extension may be addressed separately for each of the three factors composing $\det_{S_i}$. First we address the reduced norm $\theta_i: S_i \ra k_i$. It is well-known that there exists a minimal finite separable extension $k'_i/k_i$ such that $S_i \otimes_{k_i} k'_i \risom M_{n_i}(k'_i)$ and $[k'_i : k_i] \leq n_i$ (see e.g.\ \cite[Prop.\ 4.5.5]{GS2006}). This representation of $S_i$ over $k'_i$ realizes $\theta_i \otimes_{k_i} k'_i$. Later in the proof, it will be useful to draw this isomorphism as $\alpha_i : S_i \otimes_{k_i^s} {k'_i}^s \risom M_{n_i}(k'_i)$.

Next we address the field-theoretic norm $N_i: k_i^s \ra k$, which is a $f_i$-dimensional $k$-linear pseudorepresentation. By definition, $N_i$ is the pseudorepresentation induced by the determinant of the regular representation $\rho_i : k_i^s \ra M_{f_i}(k)$ of $k_i^s$ over $k$. 

We claim that the $q_i$-Frobenius map $F_{q_i}: k_i \ra k_i^s$, which is a $q_i$-dimensional pseudorepresentation, is realized as the determinant of some representation after a finite scalar extension if and only if $k_i/k_i^s$ is finite. Secondly, when $k_i/k_i^s$ is finite we claim that this scalar extension may be taken to be separable if and only if it may be taken to be trivial if and only if $k_i/k_i^s$ is a simple extension. 

To prove the first claim, observe that the minimal dimensional $k_i^s$-linear representation of $k_i$ is the the regular representation $k_i \ra M_{[k_i : k_i^s]}(k_i^s)$. Assuming $[k_i : k_i^s]$ is finite, let $L_i$ be an intermediate field such that $[k_i: L_i] = q_i$. Such an $L_i$ exists since any purely inseparable extension may be realized as a sequence of extensions of degree $p$ achieved by adjoining $p$th roots. The ring $k_i \otimes_{k_i^s} L_i$ is a local ring with residue field $k_i$. The action of $k_i \otimes_{k_i^s} L_i$ on $k_i$ realized by projection of the regular action to the residue field, which we will call $\beta_i$ below, is then a $q_i$-dimensional $L_i$-linear representation. One can check that its induced $q_i$-dimensional pseudorepresentation is $F_{q_i} \otimes_{k_i^s} L_i$. Indeed, the characteristic polynomial of the $L_i$-linear action of $\alpha \in k_i$ on $k_i$ is $X^{q_i} - \alpha^{q_i}$. 

Having proved the first claim, the second claim follows from the following observations. If $K_i/k_i^s$ is any finite separable extension, then $k_i \otimes_{k_i^s} K_i$ is a field, and so the minimal dimension of a $K_i$-linear representation of $k_i \otimes_{k_i^s} K_i$ remains $[k_i : k_i^s]$. Also, we observe that $q_i = [k_i : k_i^s]$ if and only if $k_i/k_i^s$ is simple if and only if $k_i/k$ is simple. 

Taking the fields $L_i$ and $k'_i$ as above, let $L'_i$ be the composite field $L_i {k'_i}^s \cong {k'_i}^s \otimes_{k_i^s} L_i$, which has degree over $k$ satisfying the bound $[L'_i : k] \leq [k_i : k]n_i/q_i$. We claim that there is a $L'_i$-linear representation of $S_i$ realizing $\det_{S_i} \otimes_k L'_i$. First we note that we can draw an isomorphism
\[
k_i \otimes_k L'_i \simeq k_i \otimes_{k_i^s} k_i^s \otimes_k L'_i \simeq k_i \otimes_{k_i^s} L'_i \otimes_k k_i^s
\simeq k_i \otimes_{k_i^s} {k'_i}^s \otimes_{k_i^s} L_i \otimes_k k_i^s. 
\]
The rightmost ordering of tensor factors makes it clear that we can apply the composition of appropriate scalar extensions of $\alpha_i$ and $\beta_i$, followed by $\rho_i$ on the $k_i^s$ factor to obtain a $L'_i$-linear $n_iq_if_i$-dimensional representation 
\[
\gamma_i : S_i \otimes_k L'_i \lra M_{n_i}(k'_i) \otimes_{k_i^s} L_i \otimes_k k_i^s \lra M_{n_iq_i}(L'_i) \otimes_k k_i^s \lra M_{n_iq_if_i}(L'_i)
\]
realizing the pseudorepresentation $\det_{S_i} \otimes_k L'_i$. 

The composite field $k'$ of all of these field extensions $L'_i$ of $k$ then satisfies the properties sought after in the statement, including the degree bound. In particular, the pseudorepresentation $D$ from the statement of the corollary is realized, after a base change to $D \otimes_k k'$, by the $d$-dimensional $k'$-linear representation 
\[
\prod_{i=1}^s \left(\gamma_i^{\times m_i} \otimes_{L'_i} k'\right) : R/\ker(D) \otimes_k k' \lrisom \prod_{i=1}^s \left(M_{n_iq_if_i}({L'_i}) \otimes_{L'_i} k'\right)^{\times m_i}. 
\]
Moreover, (1) follows from the observation that when $k_i/k$ is simple, $L'_i$ may be taken to be separable with $[L'_i: k]$ bounded by $n_i f_i$. This results on the desired restrictions on $k'$ in the case that $k_i/k$ is simple for all $i$. 

Under the assumptions of (2), the arguments above allow us to take $L'_i/k_i$ to be the trivial extension, so that $k'/k$ is also the trivial extension. In particular, the assumption about the Brauer group guarantees that $k'_i = k_i$. 
\end{proof}

\subsection{Cayley-Hamilton Algebras are Polynomial Identity Rings}
\label{subsec:CHPI}

The notion of a Cayley-Hamilton pseudorepresentation will be critical in what follows. 
\begin{defn}
A pseudorepresentation $D: R \ra A$ is called \emph{Cayley-Hamilton} when for any $A$-algebra $B$ and every element $r \in R \otimes_A B$, $r$ satisfies its characteristic polynomial $\chi(r,t) \in B[t]$, i.e.\ $\chi(r,r) = 0$.
In this case, one also calls $(R,D)$ a \emph{Cayley-Hamilton algebra}.
\end{defn}
This terminology comes from the Cayley-Hamilton theorem, which says that the determinant map $\det: M_{d\times d}(A) \ra A$ is a Cayley-Hamilton pseudorepresentation. 

Chenevier shows 
\begin{prop}[{\cite[\S1.17]{chen2014}}]
\label{prop:CH}
Given a pseudorepresentation $D: R \ra A$, there exists a two-sided ideal $\CH(D) \subset R$ which is the obstruction to $D$ being Cayley-Hamilton. Because $\ker(D) \supseteq CH(D)$,  $D$ factors through $R/\CH(D)$ giving a Cayley-Hamilton pseudopreresentation $D: R/\CH(D) \ra A$ (which we also denote by $D$) and making $(R/\CH(D), D)$ into a Cayley-Hamilton algebra. Also, the Cayley-Hamilton ideal is stable under base change; that is, for any $A$-algebra $B$, there is a natural isomorphism
\[
R/\CH(D) \otimes_A B \lrisom (R \otimes_A B)/\CH(D \otimes_A B).
\]
\end{prop}

It can be useful to take the perspective that a Cayley-Hamilton algebra is a generalization of a matrix algebra, and to consider Cayley-Hamilton algebra-valued representations. For instance, Procesi proved that in equi-characteristic 0, any Cayley-Hamilton algebra admits an embedding into a matrix algebra \cite{procesi1987}. When we take the Cayley-Hamilton algebra produced out of the universal $d$-dimensional pseudorepresentation $D^u: R \otimes_A \Gamma_A^d(R)^\ab \ra \Gamma_A^d(R)^\ab$, which we denote by
\[
E(R)^d := (R \otimes_A \Gamma_A^d(R)^\ab)/\CH(D^u),
\]
we can get a ``universal Cayley-Hamilton algebra'' $(E(R)^d, D^u)$ and ``universal Cayley-Hamilton representation'' $\rho^u: R \otimes_A \Gamma_A^d(R)^\ab \ra E(R)^d$ (see \cite[\S1.22]{chen2014}). The consequence of this universality that we are concerned with is the following
\begin{prop}[{\cite[Prop.\ 1.23]{chen2014}}]
\label{prop:R_E_equivalent}
For any commutative $A$-algebra $B$ and $(V_B, \rho_B) \in \Rep^d_R(B)$, there exists a unique factorization of $\rho_B$ as
\[
R\otimes_A B \lra E(R)^d \otimes_{\Gamma_A^d(R)^\ab} B \ra \End_B(V_B)
\]
where the map $\Gamma_A^d(R)^\ab \ra B$ is induced by the pseudorepresentation $\det \circ \rho_B: R \otimes_A B \ra B$. In partulcar, for $\rho_B = \rho^u$, there is a canonical isomorphism
\[
\Rep^d_R \lrisom \Rep_{E(R)^d, D^u}
\]
to the moduli of representations of $E(R)^d$ compatible with $D^u$ (see Definition \ref{defn:D_compatible}).
\end{prop}

Now we will discuss \emph{polynomial identity rings}, written \emph{PI-rings}; we refer to the book \cite{procesi1973} for the precise definition of a polynomial identity ring. It will suffice to say that an associative ring $R$ is called a polynomial identity ring when there exists some non-commutative polynomial in $n$ variables that every $n$-tuple in $R^{\times n}$ satisfies. For example, every commutative ring $R$ is a polynomial identity ring because any $x,y \in R$ satisfy the equation $xy-yx = 0$. 

\begin{prop}
\label{prop:CH_finite}
If $(R,D)$ is a Cayley-Hamilton $A$-algebra, it is a PI-$A$-algebra with polynomial identity dependent only on the dimension of $D$. If, in addition, $R$ is finitely generated over the Noetherian ring $A$, $R$ is finite as an $A$-module.
\end{prop}
\begin{proof}
By \cite[Prop.\ 3.22]{procesi1973}, given any $d \in \bZ_{\geq 1}$, there is an explicit polynomial identity with coefficients in $\bZ$ such that any associative $A$-algebra $R$ that is integral over $A$ with degree bounded by $d$ is a PI-$A$-algebra with this particular polynomial identity.  Consequently, any Cayley-Hamilton $A$-algebra $(R,D)$ is a PI-$A$-algebra because any element of $R$ is integral over $A$ of degree bounded by $d = \dim(D)$; indeed, $\chi(r,r) = 0$ for all $r \in R$. 
By \cite[Ch.\ VI, Thm.\ 2.7]{procesi1973}, any integral, finitely generated non-commutative PI-algebra over a commutative Noetherian ring is module-finite. 
\end{proof}

Consequently, such $R$ is Noetherian, finite as a module over its center, and Jacobson when $A$ is Jacobson \cite[\S1.1.3, \S9.1.3]{MR2001}. Remarkably, this proposition along with Proposition \ref{prop:R_E_equivalent} implies that the study of $d$-dimensional representation theory of a finitely generated non-commutative $A$-algebra $R$ amounts to the study of representations of a certain module-finite algebra over a Noetherian ring. In particular, we have this strengthening of Chenevier's Theorem \ref{thm:chen_2.16}. 
\begin{cor}
\label{cor:cond_4}
With the assumptions of Theorem \ref{thm:chen_2.16}, $\dim_k R/\ker(D)$ is finite when $R$ is finitely generated as a $k$-algebra. 
\end{cor}

Here are some results from PI-theory that will be useful even in the infinitely generated cases we will study later, namely group algebras of profinite groups. 
\begin{prop}
\label{prop:N(d)}
Let $A = k$ be field and let $R$ be an associative (non-unital) $k$-algebra satisfying the polynomial identity $x^d$, i.e.\ every element of $R$ is nilpotent of degree at most $d \in \bZ_{\geq 1}$. Then there exists some $N = N(d) \in \bZ_{\geq 1}$ depending only upon $d$ such that $R$ is nilpotent of degree $N$, i.e.\ $R^N = 0$.
\end{prop}
\begin{proof}
When $\ch(k) = 0$ or $\ch(k) > d$, the Nagata-Higman theorem states that $R^N = 0$ where $N = 2^d-1$.

On the other hand, the main theorem of \cite{samoilov2009} states that if $\ch(k) = p > 0$, then there exists an integer $N = N(p,d)$ depending only on $p$ and $d$ such that $R^N = 0$. Combining these two results, we may set $N(d) = \max\{2^d-1\} \cup\{N(p,d)\}_{p \leq d}$.
\end{proof}

The work of Samoilov \cite{samoilov2009} is the key input needed to loosen conditions guaranteeing that a deformation ring of pseudorepresentations of a profinite group is Noetherian (see Prop.\ \ref{prop:R_Db_noeth}). It will be used in the form of the following 
\begin{cor}
\label{cor:kernel_CH_nilpotent}
Given a positive integer $d$, there exists a positive integer $N(d)$ with the following property. Let $k$ be a field and let $(R,D)$ be a Cayley-Hamilton $k$-algebra of degree $d$ (which may not be finitely generated over $k$). Then the kernel $\ker(D) \subset R$ is nilpotent of order $N(d)$.
\end{cor}
\begin{proof}
The definition of $\ker(D)$ implies that every element $r \in \ker(D)$ has characteristic polynomial $\chi(r,t) = t^d$, and because $(R,D)$ is Cayley-Hamilton we have that $\chi(r,r) = r^d = 0$. Then Proposition \ref{prop:N(d)} implies the result. 
\end{proof}

\subsection{Invariant Theory}
\label{subsec:IT}

For this paragraph, we will assume that $R$ is finitely generated over the Noetherian commutative ring $A$ so that $\Rep^d_R$ and $\PsR^d_R$ are finite type over $S = \Spec A$. The fact that $\psi: \Rep^d_R \ra \PsR^d_R$ is a bijection on closed geometric points suggests a comparison between $\PsR^d_R$ and the geometric invariant theoretic (GIT) quotient. 
\begin{defn}
\label{defn:git_quotient}
The \emph{GIT quotient} of the action of an affine algebraic $S$-group scheme $G$ on an affine $S$-scheme $X = \Spec B$, written $X /\!/ G$, is given by $X /\!/ G := \Spec B^G$.
\end{defn}
The work of Alper \cite{alper2013, alper2014} provides a useful perspective on geometric invariant theory that is appropriate for our use. We will refer to \emph{loc.\ cit.}~for the definitions of adequate and good moduli spaces, since for our purposes, the following examples of adequate and good moduli spaces suffice. 
\begin{eg}
\label{eg:adequate_good}
\begin{enumerate} 
\item Let $G$ be a reductive $S$-group scheme acting on an affine $S$-scheme $X$. Then the natural morphism from the quotient stack to the GIT quotient $[X / G]  \ra X /\!/ G$ is an example of an \emph{adequate moduli space}.

\item Now $G$ be a linearly reductive $S$-group scheme acting on $X$; see e.g.\ \cite[\S12]{alper2013} for a definition. Then $[X / G]  \ra X /\!/ G$ is an example of a \emph{good moduli space}. We will only require the fact that a torus is linearly reductive over any $S$. If $S = \Spec k$ and $\ch k =0$, reductive is equivalent to linearly reductive; if $\ch k > 0$, linearly reductive means that the connected component of the identity in $G$ is a torus, and the group of components has order prime to $\ch k$. 
\end{enumerate}
\end{eg}
We will be interested in the particular case of the adequate moduli space
\[
\phi: \Rep^d_R \cong [\Rep^{\square,d}_R / \GL_d] \lra \Rep^{\square,d}_R /\!/ \GL_d.
\]

Here are the main results of Alper's work. 
\begin{thm}[\cite{alper2014, alper2013}]
\label{thm:alper}
Let $\phi: \cX \ra Y$ be an adequate moduli space.
\begin{enumerate}
\item $\phi$ is surjective, universally closed, and universally submersive.
\item Two geometric points $x_1, x_2 \in \cX(\bar k)$ are identified in $Y$ if and only if their closures $\overline{\{x_1\}}$ and $\overline{\{x_2\}}$ intersect in $\cX \times_\bZ \bar k$.
\item If $\cX$ is finite type over a Noetherian scheme $S$, then $Y$ is finite type over $S$ and for every coherent $\cO_\cX$-module $\cF$, $\phi_* \cF$ is coherent. 
\item $\phi$ is universal for maps from $\cX$ to algebraic spaces which are either locally separated or Zariski-locally have affine diagonal.
\item Adequate moduli spaces are stable under flat base change and descend in the fpqc topology of the target. 
\item A good moduli space is an adequate moduli space. 
\item Good moduli spaces are stable under arbitrary base change.
\end{enumerate}
\end{thm}

Part (4) of Theorem \ref{thm:alper} immediately implies that $\psi: \Rep^d_R \ra \PsR^d_R$ factors uniquely through the adequate moduli space $\phi: \Rep^d_R \ra \Rep^{\square,d}_R /\!/ \GL_d$, inducing a canonical map $\nu$:
\[
\xymatrix{
\Rep^d_R \ar[d]^\psi \ar[dr]^\phi & \\
\PsR^d_R  & \Rep^{\square,d}_R /\!/ \GL_d \ar[l]^(.6)\nu
}
\]
Combining Corollary \ref{cor:ss_to_psr} with part (2) of Theorem \ref{thm:alper}, we find that $\nu$ induces an isomorphism on geometric points. It is the same to say that $\nu$ is surjective and radicial \cite[3.5.5]{ega1}. 

What we will show is that $\PsR^d_R$ differs from the GIT quotient by at most an \emph{adequate homeomorphism}, i.e.\ an integral universal homeomorphism that is an isomorphism in characteristic zero. In the affine Noetherian case, this means that the kernel and cokernel of a ring map consists of finite modules of $p$-torsion nilpotents. It is possible to eliminate this difference in certain cases (see Theorem \ref{thm:CH_MF_GMA}).
\begin{thm}
\label{thm:psi_main}
If $R$ is finitely generated over $A$, the difference $\nu$ between $\psi: \Rep^d_R \ra \PsR^d_R$ and an adequate moduli space is an adequate homeomorphism.
\end{thm}
We emphasize that the isomorphism in characteristic zero is due to Chenevier, using ideas of Procesi \cite{procesi1987}.
\begin{proof}
The proof that $\psi$ is precisely an adequate moduli space in equi-characteristic zero is due to Chenevier \cite[Prop.\ 2.3]{chen2013}. We know that $\nu$ is surjective and radicial by the comments above, so in light of \cite[Cor.\ 18.12.11]{ega4-4}, it remains to show that $\nu$ is finite. It will suffice to prove that $\nu$ is universally closed, since it is clearly affine, hence separated, and proper affine morphisms are finite. We will prove this by verifying the valuative criterion for universally closed morphisms given in \cite[Thm.\ 7.10]{lmb} (see also \cite[Remark 7.3.9(i)]{ega2}). 

Let $B$ represent a complete discrete valuation ring with an algebraically closed residue field and fraction field $K$.  Given a diagram of $A$-schemes
\[\xymatrix{
\Spec K \ar[r]^(.4)\alpha \ar[d] & \Rep^{\square, d}_R /\!/ \GL_d \ar[d]^\nu \\
\Spec B \ar[r]^{D_B} & \PsR^d_R
}\]
we will show that there exists a finite field extension $K'/K$ and, letting $B'$ denote the integral closure of $B$ in $K'$, a morphism $f: \Spec B' \ra \Rep^d_R$ such that $\phi \circ f: \Spec B' \ra \Rep^{\square, d}_R /\!/ \GL_d$ verifies the valuative criterion. 

Let $D_B$ denote the pseudorepresentation of $R$ over $B$ associated to the $B$-point of $\PsR^d_R$ in the diagram, and let $D$ denote $D_B \otimes_B K$.  Theorem \ref{thm:alper} implies that there exists a semi-simple $\bar K$-representation of $R$ inducing a point of $\Rep^d_R$ lying over $\alpha$. The kernel of the action of $R \otimes_A K$ on this representation factors through $(R \otimes_A K)/\ker(D)$, which is finite-dimensional over $K$ by Corollary \ref{cor:cond_4}. Then Corollary \ref{cor:ss_degree} tells us that this representation is, in fact,  realizable as a representation $\rho: R \otimes_{A} K' \ra M_d(K')$ in $\Rep^{d}_R(K')$ where $K'/K$ is some finite extension of fields, and whose induced pseudorepresentation $\det \circ \rho$ is identical to $D \otimes_K K'$. 

Let $B'$ be the integral closure of $B$ in $K'$, which is a DVR \cite[Prop.\ II.3]{serre1979}. We claim that $\rho$ is isomorphic to $\rho_{B'} \otimes_{B'} K'$, where $\rho_{B'} : R \otimes_{A} B' \ra \End_{B'}(L')$ and $L'$ is a rank $d$ projective $B'$-module, which will complete the proof. Choose a $d$-dimensional $K'$-vector space $V'$ realizing $\rho$, and let $L$ be a $B'$-lattice $L \subset V'$. Now let $L'$ be the $B'$-linear span of the translates of $L$ by $R \otimes_{A} B$. This is a finite $B'$-module because the action of $R \otimes_A B$ factors through its Cayley-Hamilton quotient $(R \otimes_A B)/\CH(D \otimes_A B)$ by Proposition \ref{prop:R_E_equivalent}, and this quotient is $B$-module-finite by Proposition \ref{prop:CH_finite}. Therefore $L'$ is a $B'$-lattice because it is finite and torsion-free, and the induced $\rho_{B'} : R \otimes_A B' \ra \End_{B'}(L')$ yields $\rho$ after applying $\otimes_{B'} K'$. 
\end{proof}

\subsection{Generalized Matrix Algebras}

The concept of a generalized matrix algebra (GMA) with respect to a pseudocharacter  has been carefully studied in \cite[\S1]{BC2009}. It will be helpful in the sequel to develop the notion of GMA relative to a pseudorepresentation in order to eliminate complications with pseudocharacters arising in small characteristic. In particular, this will allow us to adapt the theory of GMAs to characteristic smaller than the dimension. However, no change to the definition of the GMA is necessary: we will show that a GMA admits a canonical pseudorepresentation. This was also shown independently by Ann-Kristin Juschka, following the suggestion of \cite[Remark 2.3.3.6]{WEthesis}.

A \emph{pseudocharacter} is the data of a trace coefficient function $\Lambda_1$ satisfying identities expected of a trace function coming from a representation (see \cite{taylor1991}, \cite[\S1.2]{BC2009} for the definition), while a $d$-dimensional pseudorepresentation $D$ keeps track of all characteristic polynomial coefficients $\{\Lambda_i^D\}_{i=1}^d$ as in \eqref{eq:chi_poly}. In particular, a $d$-dimensional pseudorepresentation $D : R\ra A$ induces a $d$-dimensional pseudocharacter $\Lambda_1^D : R \ra A$. This is a bijective correspondence if $(2d)!$ is invertible in $A$ \cite[Prop.\ 1.29]{chen2014}. 

We follow \cite[\S1.3]{BC2009} closely in what follows, showing that a GMA admits a canonical pseudorepresentation compatible with its canonical pseudocharacter. 
\begin{defn}[{\cite[\S1.3]{BC2009}}]
Let $A$ be a commutative ring and let $R$ be an $A$-algebra. Call $R$ a \emph{generalized matrix algebra} or \emph{GMA} of type $(d_1, \dotsc, d_r)$ if there exists data $\cE = (\{e_i\}, \{\phi_i\})$ as follows.
\begin{enumerate}
\item A set of $r$ orthogonal idempotents $e_1, \dotsc, e_r$ with sum $1$, and
\item A set of isomorphisms of $A$-algebras $\phi_i : e_i R e_i \risom M_{d_i}(A)$,
\end{enumerate}
such that the trace map $\Tr = \Tr_\cE : R \ra A$ defined by
\[
\Tr(x) := \sum_{i=1}^r \Tr \phi_i(x)
\]
is a central function, i.e.\ $\Tr(xy) = \Tr(yx)$ for all $x,y \in R$. We call $\cE$ the \emph{data of idempotents} of $R$ and write $(R, \cE)$ for a GMA.
\end{defn}

We note that all of the arguments of \cite[\S1.3.1-\S1.3.6]{BC2009} have no dependence on the characteristic of $A$ or the invertibility of $d!$ in $A$, except the proof that the trace map $\Tr$ associated to $\cE$ is a pseudocharacter. Therefore, we have access to these results of \cite[\S1]{BC2009} on the structure of a GMA, which we record here in order to introduce notation. 
\begin{itemize}[leftmargin=2em]
\item Write $\delta^{j,k} \in M_{d}(A)$ for the matrix with entries $0$ except in the $(j,k)$th entry, where the value is $1$. 
\item For each $i$, $1 \leq i \leq r$, a primitive decomposition of idempotents $e_i = E_i^{1} + \dotsm + E_i^{d_i}$ where $E_i^j := \phi_i^{-1}(\delta^{j,j})$.
\item We also write $E^l$ for the $l$th primitive idempotent of $R$ given by the order down the diagonal of the idempotents  $E_1^1, E_1^2, \dotsc, E_1^{d_1}, E_2^1, \dotsc$, i.e.\ $E^l = E^j_i$ for 
$l = j + \sum_{i' = 1}^{i-1} d_{i'}$.
\item Set $\cA_{i,j} := E_i^1 R E_j^1$, and write $\varphi_{i,j,k}$ for the map $\cA_{i,j} \otimes_A \cA_{j,k} \ra \cA_{i,k}$ induced by multiplication, where $1 \leq i,j,k \leq r$. These satisfy the properties (UNIT), (COM), and (ASSO) of \cite[Lem.\ 1.3.5]{BC2009}.
\item The argument of \cite[\S1.3.2]{BC2009} applies to show that there is a canonical isomorphism
\begin{equation}
\label{eq:GMA_form}
R \lrisom \begin{pmatrix}
M_{d_1}(\cA_{1,1}) & M_{d_1 \times d_2}(\cA_{1,2}) & \dotsm & M_{d_1 \times d_r}(\cA_{1,r}) \\
M_{d_2 \times d_1}(\cA_{2,1}) & M_{d_2}(\cA_{2,2}) & \dotsm & M_{d_2 \times d_r}(\cA_{2,r}) \\
\vdots & \vdots & \vdots & \vdots \\
M_{d_r \times d_1}(\cA_{r,1}) & M_{d_2}(\cA_{r,2}) & \dotsm & M_{d_r}(\cA_{r,r}) \\
\end{pmatrix},
\end{equation}
\item Likewise, write $\cA^{i,j} := E^i R E^j$, and write $\varphi^{i,j,k}$ for the map $\cA^{i,j} \otimes_A \cA^{j,k} \ra \cA^{i,k}$ induced by multiplication, where $1 \leq i,j,k \leq d$, which will also satisfy (UNIT), (COM), and (ASSO).
\item A canonical isomorphism $\phi_{i,i}: \cA_{i,i} \risom A$ induced by $\phi_i$ and a canonical isomorphism $\phi^l : \cA^{l,l} \risom A$ induced by $\phi_{i_l}$.
\end{itemize}

We recall the definition of an adapted representation of a GMA. 
\begin{defn}[{\cite[Definition 1.3.6]{BC2009}}]
\label{defn:adapted}
Let $B$ be a commutative $A$-algebra and let $(R,\cE)$ be a generalized matrix $A$-algebra.  A representation $\rho_B: R \ra M_d(B)$ is said to be \emph{adapted to $\cE$} if its restriction to the $A$-subalgebra $\bigoplus_{i=1}^r e_i R e_i$ is the composite of the representation $\bigoplus_{i=1}^r \phi_i$ by the natural ``diagonal'' map
\[
M_{d_1}(A) \oplus \dotsm \oplus M_{d_r}(A) \ra M_d(B).
\]
We define $\Rep^\square_\Ad(R, \cE)$ to be the functor associating to an $A$-algebra $B$ the set of adapted representations of $(R,\cE)$ over $B$.
\end{defn}

By \cite[Prop.\ 1.3.9]{BC2009}, $\Rep^\square_\Ad(R,\cE)$ is represented by the affine scheme corresponding to the quotient of the $A$-algebra 
\begin{equation}
\label{eq:adapted_rep_alg}
T := \Sym^*_A \left( \bigoplus_{1 \leq i \neq j \leq r} \cA_{i,j}\right)
\end{equation}
by the ideal $J$ generated by $b \otimes c  - \varphi(b \otimes c)$ for all $\varphi = \varphi_{i,j,k}$, $b \in \cA_{i,j}$, $c \in \cA_{j,k}$ and $1 \leq i,j,k \leq r$. These are precisely the relations required to ensure that the morphism of $A$-modules $R \ra M_d(B)$ induced by the $A$-algebra homomorphism $T \ra B$ along with  \eqref{eq:GMA_form} is actually a morphism of $A$-algebras. Then, the universal adapted representation $f: R \ra M_d(T/J)$ is induced by the natural maps $\cA_{i,j} \ra B$; each $\cA_{i,j} \ra B$ is an $A$-split injection \cite[Prop.\ 1.3.13]{BC2009}. 

With the above notions in place, we are equipped to introduce a canonical pseudorepresentation associated to a GMA.

\begin{prop}
\label{prop:GMA_PsR}
Given a GMA $A$-algebra $(R, \cE)$ of dimension $d$, there exists a natural $d$-dimensional Cayley-Hamilton pseudorepresentation ${D_\cE}: R \ra A$ given by, for any commutative $A$-algebra $B$, the formula 
\[
D_\cE(x) := \sum_{\sigma \in S_d} \sgn(\sigma) \prod_{\text{cycles } \gamma \text{ of } \sigma} {\phi^{k}}\left(\prod_{l =k}^{\sigma^{-1}(k)} E^l x E^{\sigma(l)}\right)
\]
for any $x \in R \otimes_A B$. Here, the product is first over the cycles $\gamma$ of $\sigma$ and then over the elements $l$ of the cycle taken in the order that they appear in the cycle, where $k$ is a choice of initial element of $\gamma$. We also have $\Tr = \Lambda_1^{D_\cE}$. 
\end{prop}
\begin{proof}
It is clear that we have a homogenous degree $d$ polynomial law ${D_\cE}: R \ra A$, and it will be a pseudorepresentation if it is multiplicative. This follows from the fact that, by inspection of the definition of ${D_\cE}$, the injection $f: R \rinj M_d(T/J)$ satisfies ${D_\cE} = \det \circ f$. These maps remain injective after any base extension $\otimes_A B$ because the injections $\cA_{i,j} \ra T/J$ are split. Therefore, the determinant is a multiplicative homogenous degree $d$ polynomial law, i.e.\ a pseudorepresentation. One may check that $\Tr_\cE = \Lambda_1^{D_\cE}$ by computing the characteristic polynomial $\chi_{D_\cE}(r, t)$. 

We now verify that $D_\cE$ does not depend upon the choice of initial element $k$ in each cycle $\gamma$ composing $\sigma$. This follows from the property (COM) of the multiplication maps $\varphi$ deduced from the centrality of $\Tr_\cE$ in \cite[Lem.\ 1.3.5]{BC2009}, which reads as follows:
\[
\text{(COM)} \quad \text{For all } i,j \text{ and all } x \in \cA_{i,j}, y \in \cA_{j,i}, \text{we have } \varphi_{i,j,i}(x \otimes y) = \varphi_{j,i,j}(y \otimes x).
\]
Therefore, for any $\sigma \in S_d$, cycle $\gamma$ of $\sigma$, and $k \in \gamma$ at which we will begin the multiplication, we have that 
\begin{align*}
\prod_{l = k}^{\sigma^{-1}(k)} &E^l x E^{\sigma(l)} = \left(E^k x E^{\sigma(k)} \right) \cdot \prod_{l = \sigma(k)}^{\sigma^{-1}(k)} E^l x E^{\sigma(l)} = \\
& \prod_{l = \sigma(k)}^{\sigma^{-1}(k)} E^l x E^{\sigma(l)}\cdot \left( E^k x E^{\sigma(k)}\right) = \prod_{l = \sigma(k)}^{k} E^l x E^{\sigma(l)}
\end{align*}
where we apply (COM) in the central equality. 
\end{proof}

The proof shows that the determinant of the universal adapted representation $R \ra M_d(T/J)$ is compatible with  the pseudorepresentation ${D_\cE}: R \ra A$ induced by the GMA structure $(R,\cE)$.  Consequently, we have a monomorphism $\Rep^\square_\Ad(R, \cE) \rinj  \Rep^\square_{R,{D_\cE}}$ induced by forgetting the adaptation; it may be easily checked to be a closed immersion. 

Considering the adjoint action of $\GL_d$ on framed representations, the stabilizer subgroup of an adaptation is the center $Z(\cE)$ of the diagonally embedded subgroup $\GL(\cE) := \GL_{d_1} \times \GL_{d_2} \times \dotsm \times \GL_{d_r} \rinj \GL_d$. Therefore $Z(\cE)$ acts on $\Rep^{\square}_\Ad(R,\cE)$ compatibly with the action of $\GL_d$ on $\Rep^\square_{R, D_\cE}$ via the immersion above. This means that the morphism \eqref{eq:adapted_to_reg} exists, and, furthermore, we show the following.
\begin{prop}
\label{prop:adapted_GMA}
Given a GMA $(R,\cE)$ over $A$, the natural morphism 
\begin{equation}
\label{eq:adapted_to_reg}
[\Rep^\square_\Ad(R,\cE) / Z(\cE)] \lra \Rep_{R,{D_\cE}} 
\end{equation}
of $\Spec A$-algebraic stacks is an isomorphism.
\end{prop}
\begin{proof}
Let $X$ be a $\Spec A$-scheme.  Choose $(\rho,V_X) \in \Rep_{R,{D_\cE}}(X)$.  The idempotents $e_i \in R$ break $V_X$ into a direct sum of projective sub-$\cO_X$-modules $V_i := \rho(e_i) V_X$ of rank $d_i$,
\[
V_X \cong \bigoplus_{i=1}^r V_i.
\]
Each $V_i$ receives an $A$-linear action of $e_i R e_i \subset R$. Using the GMA data $\phi_i: e_i R e_i \risom M_{d_i}(A)$ and the fact that the pseudorepresentation induced by $V_X$ lies over $D_\cE$, we see that the action of $e_i R e_i$ on $V_i$ is faithful and induces an isomorphism $\End_{\cO_X}(V_i) \simeq M_{d_i}(\cO_X)$.  Consequently, $V_i$ is isomorphic as an $\cO_X$-module to a twist of a free rank $d_i$ vector bundle $\cO_X^{\oplus d_i}$ by some line bundle $\cL_i$.

Let $\cG_i := \Isom_{\cO_X}(\cL_i, \cO_X)$ be the $\bG_m$-torsor over $X$ corresponding to $\cL_i$.  Then $\cG := \times_{i=1}^r \cG_i$ is naturally a $Z(\cE)$-torsor.  Indeed, the base change of $V_X$ to $\cG$ from $X$ is a free rank $d$ $\cO_\cG$-vector bundle with a canonical basis adapted to $(R,\cE)$.  This defines a map $\cG \ra \Rep^\square_\Ad(R,\cE)$, equivariant for the action of $Z(\cE)$.  We have therefore established a morphism
\[
\Rep_{R,{D_\cE}} \lra [\Rep^\square_\Ad(R,\cE) / Z(\cE)].
\]
We observe that this provides a quasi-inverse to \eqref{eq:adapted_to_reg}. 
\end{proof}

In the case of a generalized matrix algebra, we can improve on Theorem \ref{thm:psi_main}. 
\begin{cor}
\label{cor:GMA_good}
Let $(R, \cE)$ be a generalized matrix $A$-algebra with canonical pseudorepresentation ${D_\cE}: R \ra A$. Then $\Rep_{R,{D_\cE}} \ra \Spec A$ is a good moduli space. 
\end{cor}
\begin{proof}
We will argue that the invariant ring $(T/J)^{Z(\cE)}$ of the $Z(\cE)$-action on the coordinate ring $T/J$ of $\Rep^\square_\Ad(R, \cE)$, given in \eqref{eq:adapted_rep_alg}, is equal to $A$. In light of Example \ref{eg:adequate_good}, $\Rep^\square_\Ad(R, \cE) \ra \Spec A$ is a good moduli space because tori are linearly reductive over any base, and because $A = (T/J)^{Z(\cE)}$. Then the statement of the Corollary follows from Proposition \ref{prop:adapted_GMA}. 

It is clear that $A$ is contained in $(T/J)^{Z(\cE)}$, and we will show any $Z(\cE)$-invariant in $T/J$ is in $A$. Indeed, $Z(\cE)$ acts on $\cA_{i,j}$ as the torus in $\GL_r$ acts by roots on the $(i,j)$-coordinate of $M_r$, so that the invariant subring of $T$ is generated by tensors of the form
\[
\prod_{1 \leq i \leq \ell} a_{i, \sigma(i)}
\]
where $\sigma : \{0, \dotsc, \ell\} \ra \{1, \dotsc, \ell\}$ is a cycle (i.e.\ $\sigma(0) = \sigma(\ell)$) of length $\ell$. The ideal $J$ is stable under $Z(\cE)$, and $(T/J)^{Z(\cE)} = T^{Z(\cE)}/J^{Z(\cE)}$ because $Z(\cE)$ is linearly reductive \cite[Rem.\ 4.11]{alper2013}. By considering the generators of $J$ and the property (ASSO), we conclude that all of the invariant tensors are equivalent to elements of $A$. 
\end{proof}

The following conditions will be useful to show that certain Cayley-Hamilton algebras are GMAs.
\begin{defn}
\label{defn:split_MF}
Let $(A, \frmm_A)$ be a local ring with the usual data $D: R \ra A$ and residue field $\bF := A/\frmm_A$. 
\begin{enumerate}
\item We denote by $\Db$ the \emph{residual pseudorepresentation} $D \otimes_A \bF : R \otimes_A \bF \ra \bF$, and call $\Db$ \emph{split} and $D$ \emph{residually split} if $(R \otimes_A \bF)/\ker(\Db)$ is a product of matrix algebras over $\bF$.
%%%  
\item We call $\Db$ \emph{multiplicity-free} and call $D$ \emph{residually multiplicity-free} when $\Db$ is split and the semi-simple representation $\rho^{ss}_\Db: R \otimes_A \bF \ra M_d(\bF)$ has distinct Jordan-H\"older factors. 
\end{enumerate}
\end{defn}
Recall from Theorem \ref{thm:chen_2.16} that $(R\otimes_A \bF)/\ker(\Db)$ is a semi-simple $\bF$-algebra, so it is split after a base change by a finite extension of $\bF$. 

Chenevier has shown that in a certain case, a Cayley-Hamilton algebra $(R,D)$ may be endowed with the structure of a GMA $(R,\cE)$ such that the pseudocharacter induced by $\cE$ is equal to the trace $\Lambda_1^D$ of $D$ \cite[Thm.\ 2.22(ii)]{chen2014}. We will now remark that his proof also shows that the pseudorepresentation $D_\cE$ is equal to $D$, generalizing \cite[Cor.\ 1.3.16]{BC2009} to any characteristic. 
\begin{thm}
\label{thm:CH_MF_GMA}
Let $(R,D)$ be a finitely generated Cayley-Hamilton $A$-algebra where $A$ is a Noetherian Henselian local ring. Assume that $D$ is residually multiplicity-free. Then $(R,D)$ admits a structure $\cE$ of a generalized matrix $A$-algebra such that the pseudorepresentation induced by $(R,\cE)$ is equal to $D$. Moreover, there is an isomorphism
\[
[\Rep^\square_\Ad(R,\cE) / Z(\cE)] \lrisom \Rep_{R,{D}}
\]
Consequently, $\psi: \Rep_{R,D} \ra \Spec A$ is a good moduli space. 
\end{thm}

\begin{proof}
It will suffice to prove the theorem after replacing the Cayley-Hamilton $A$-algebra $(R,D)$ with the universal Cayley-Hamilton with residual pseudorepresentation $\Db$. This is the Cayley-Hamilton $R_\Db$-algebra $(R^u,D^u)$ where $R^u = (R \otimes_A R_\Db)/\CH(D^u)$ and $D^u$ is the universal pseudorepresentation deforming $\Db$, $D^u : R^u \ra R_\Db$. Indeed, the good moduli space property claimed in the theorem will follow because it is stable under base change (Thm.\ \ref{thm:alper}(7)). 

Using the assumptions of the statement, Chenevier \cite[Thm.\ 2.22(ii)]{chen2014} shows that there exists a data of idempotents $\cE$ of $R^u$ inducing a generalized matrix algebra $(R^u,\cE)$ over $R_\Db$. Both $D^u \otimes_{R_\Db} \bF$ and $D_\cE \otimes_{R_\Db} \bF$ are equal to $\Db$, as they each arise from the product of matrix algebras $R^u \ra (R^u \otimes_{R_\Db} \bF)/\ker(D^u \otimes_{R_\Db} \bF) \cong \prod_{i=1}^r M_{d_i}(\bF)$. Consequently, $D_\cE$ is a deformation of $\Db$ with coefficients in $R_\Db$, and the universal property of $R_\Db$ induces a map $f: R_\Db \ra R_\Db$ induced by $D_\cE$. The desired equality $D^u = D_\cE$ will follow from this map being the identity. This follows immediately from the fact that $D_\cE$ is $R_\Db$-linear by construction. One concrete way to observe this is to use the GMA structure to restrict $D_\cE$ and $D^u$ to the  matrix subalgebra $e_1 R^u e_1 \simeq M_{d_1}(R_\Db)$, resulting in a degree $d_1$ pseudorepresentation $M_{d_1}(R_\Db) \ra R_\Db$. Both restrictions equal the determinant pseudorepresentations (see e.g.\ the proof of \cite[Thm.\ 2.22]{chen2014}). In particular, for $x \in R_\Db$, we can compute the traces of matrices with one non-zero entry: $\Lambda_1^{D^u}(xE^1_1) = x$ and $\Lambda^{D_\cE}_1(xE^1_1) = f(x)$, so $f$ is the identity. 

The rest of the theorem now follows from \ref{prop:adapted_GMA} and \ref{cor:GMA_good}. 
\end{proof}

\section{Algebraic Families of Galois Representations}
\label{sec:fam_gal}

\subsection{Formal Moduli of Representations and Pseudorepresentations of a Profinite Group}

Let $G$ be a profinite group; we will often impose the $\Phi_p$-finiteness condition of Mazur \cite{mazur1989} on $G$. We wish to understand the moduli space of continuous representations of $G$ and how it relates to the moduli space of continuous pseudorepresentations. We will study these representations over integral $p$-adic coefficient rings for some prime $p$ which we fix. We will not insist that these rings are local because of the existence of positive-dimensional algebraic families of residual representations such as the one-dimensional family 
\begin{equation}
\label{eq:ext_fam-2}
\begin{pmatrix}
\bar\rho_1 & \tilde e_1 + x \tilde e_2 \\
0 & \bar\rho_2
\end{pmatrix} \quad \text{ with coefficients in } \bF[x],
\end{equation} 
where $\tilde e_1, \tilde e_2$ are representatives of linearly-independent extension classes $e_1, e_2 \in \Ext^1_G(\bar\rho_2, \bar\rho_1)$. Namely, we will let our category of coefficient rings be admissible continuous $\bZ_p$-algebras $\Adm_{\bZ_p}$, which is anti-equivalent to the category of affine Noetherian $\Spf \bZ_p$-formal schemes \cite[\S10.1]{ega1}. We will use the category of $\Spf \bZ_p$-formal schemes $\FS_{\bZ_p}$ as coefficient spaces.

As in the previous section, we will use $A$ as a base coefficient ring, and will let $A \in \Adm_{\bZ_p}$ be a local ring, so that $A$ is profinite and, in particular, has a finite residue field. Let $R$ be a profinite continuous (non-commutative) $A$-algebra. When we wish to consider the case of group representations, we may set $A = \bZ_p$ and $R = \bZ_p\lb G \rb$. 
\begin{defn}
\label{defn:Rep_groupoids_profinite}
Let $A$ and $R$ be as specified above, and let $d$ be a positive integer.
\begin{enumerate}
\item Define the functor $\CRep^{\square,d}_R$ on $\FS_A$ by
\[
\frX \mapsto \{\text{continuous } \cO_\frX\text{-algebra homomorphisms } R \otimes_{A} \cO_\frX \lra M_d(\frX)\}.
\]
\item Define the groupoid  $\CRep_R^d$, fibered over $\FS_A$, by
\[
\begin{aligned}
\ob \CRep^d_R(\frX) &= \{V / \frX \text{ a rank } d \text{ vector bundle, with a} \\
& \text{continuous } \cO_\frX\text{-algebra homomorphism } R \otimes_{A} \cO_\frX \lra \End_{\cO_\frX}(V) \}
\end{aligned}
\]
and morphisms being isomorphisms of this data.
\item Define the groupoid $\CRepB_R^d$, fibered over $\FS_A$, by 
\[
\begin{aligned}
\ob \CRepB^d_R(\frX) &= \{\cE \text{ a rank } d^2\ \cO_\frX\text{-Azumaya algebra, with a} \\
& \text{continuous } \cO_\frX\text{-algebra homomorphism } R \otimes_{A} \cO_\frX \lra \cE \}
\end{aligned}
\]
and morphisms being isomorphisms of this data.
\end{enumerate}
\end{defn}
It is not difficult to show that $\CRep^{\square,d}_R$ is representable by an affine Noetherian $\Spf A$-formal scheme when $R$ satisfies a finiteness condition equivalent to the $\Phi_p$ finiteness condition on profinite groups. But we will first show that all of these groupoids are algebraizable, from which their representability by formal schemes or formal algebraic stacks follows. 

The moduli functor of continuous pseudorepresentations of a profinite algebra has been defined and studied by Chenevier \cite{chen2014}. Firstly, he shows in \cite[Prop.\ 3.3]{chen2014} that given any finite field-valued pseudorepresentation $\Db : R \otimes_A \bF\ra \bF$, the natural deformation functor $\CPsR_\Db$ to complete local $A$-algebras with residue field $\bF$ is representable by a complete local $A$-algebra $(R_\Db, \mDb)$, i.e.\ $\CPsR_\Db \cong \Spf R_\Db$. We call the objects of $\CPsR_\Db$ ``pseudodeformations.'' When certain finiteness conditions are satisfied, $R_\Db$ is Noetherian. Here we loosen Chenevier's criteria for $R_\Db$ to be Noetherian \cite[Prop.\ 3.7]{chen2014}. 
\begin{prop}
\label{prop:R_Db_noeth}
Assume that the continuous cohomology group $H^1_c(G, \ad \rho^{ss}_\Db)$ is finite dimensional over the coefficient field of $\rho^{ss}_\Db$. Then $R_\Db$ is Noetherian. 
\end{prop}
\begin{proof}
Let $\bF$ represent the coefficient field of $\rho^{ss}_\Db$. We apply Chenevier's strategy to prove \cite[Prop.\ 3.7]{chen2014} via arguing from \cite[2.7, 2.26, 2.35]{chen2014}. The one change we make is that in \cite[Lem.\ 2.7(iii)]{chen2014}, we produce $N(d)$ such that $\ker(D)^{N(d)} = 0$ using Corollary \ref{cor:kernel_CH_nilpotent}. This removes the condition that $d!$ is invertible in $\bF$. 
\end{proof}

\begin{rem}
Proposition \ref{prop:R_Db_noeth} answers a question of Chenevier \cite[Remark 2.29]{chen2014}. It fulfills his suggestion that 
some result in the spirit of Shirshov's height theorem would allow for the proof of the Noetherianness of $R_\Db$ in terms of the finiteness of cohomology at $\rho^{ss}_\Db$ alone, instead of the stronger condition that $G$ satisfies the $\Phi_p$ finiteness condition. This result was provided by Samoilov \cite{samoilov2009} (see Cor.\ \ref{cor:kernel_CH_nilpotent}). On the other hand, the finiteness of cohomology for \emph{every} $\rho^{ss}_\Db$ implies that $G$ satisfies condition $\Phi_p$ \cite[Example 3.6]{chen2014}. 
\end{rem}

The deformation theory of various $\Db$ will suffice to describe the entire moduli functor of $d$-dimensional pseudorepresentations on formal schemes $\frX \in \FS_{\bZ_p}$, 
%%%
\[
\CPsR^d_R: \frX \mapsto \{\text{continuous pseudorepresentations } D: R \otimes_{A} \cO_\frX \lra \cO_\frX\}
\]
It will be helpful to establish notation about residual pseudorepresentations and their fields of definition. 
%%%
There is a natural equivalence relation on continuous $d$-dimensional finite field-valued pseudorepresentations of $R$, namely that $\Db \sim \Db'$ for pseudorepresentations $\Db: R \otimes_A \bF \ra \bF$ and $\Db': R \otimes_A \bF' \ra \bF'$ if $\Db \otimes_\bF \bar \bF_p \cong \Db' \otimes_{\bF'} \bar\bF_p$. 
\begin{defn}
\label{defn:residual_rep}
We will let $\Db: R \otimes_A \bF_\Db \ra \bF_\Db$ represent a \emph{residual pseudorepresentation}, which is the unique representative of each equivalence class with smallest field of definition $\bF_\Db$. By Corollary \ref{cor:ss_degree}(2), there exists a semi-simple representation $(V^{ss}_\Db, \rho^{ss}_\Db)$ over $\bF_\Db$ inducing $\Db$. 
\end{defn}

Notice that the irreducible factors of $\rho^{ss}_\Db$ need not be absolutely irreducible, i.e.\ $R/\ker(\Db)$ may not be a product of matrix algebras over $\bF$ (cf.\ Definition \ref{defn:split_MF}). Rather, it may be a product of matrix algebras over finite extensions of $\bF$. 

Chenevier has shown that the entire moduli of pseudorepresentations is simply a disjoint union of deformation functors of residual representations. 
\begin{thm}[{\cite[Cor.\ 3.14]{chen2014}}]
\label{thm:chen_3.14}
Assuming that $G$ satisfies the $\Phi_p$ finiteness condition and $R = \bZ_p \lb G \rb$, the moduli functor $\CPsR_R^d$ on $\FS_{\bZ_p}$ is representable by the disjoint union 
\[
\CPsR^d_R \cong \coprod_\Db \CPsR_\Db \cong \coprod_\Db \Spf R_\Db
\]
of local Noetherian formal schemes over the set of $d$-dimensional residual pseudorepresentations. 
\end{thm}
That is, the moduli of continuous pseudorepresentations of a profinite group is ``purely formal,'' i.e.\ semi-local, unlike the moduli of its representations. This means that all non-trivial positive-dimensional algebraic families of residual representations consist of varying extension classes with a fixed semi-simplification, with \eqref{eq:ext_fam-2} as the archetypal example. 

Because continuous representations induce continuous pseudorepresentations, there exists a natural morphism from each of the moduli spaces of Definition \ref{defn:Rep_groupoids_profinite} to $\CPsR^d_R$, for example
\[
\hat\psi: \CRep^d_R \lra \CPsR^d_R.
\]
Using the decomposition of Theorem \ref{thm:chen_3.14}, we may study this morphism over one component of the base at the time. Fixing a residual pseudorepresentation $\Db: G \ra \bF_\Db$, we set $\CRep_\Db := \CRep^d_R \otimes_{\CPsR^d_R} \CPsR_\Db$ to be the groupoid of representations with residual pseudorepresentation $\Db$, and write $\hat\psi$ for the base change
\[
\hat\psi: \CRep_\Db \lra \CPsR_\Db = \Spf R_\Db
\]
We may analogously define $\CRep^\square_\Db$ and $\overline{\CRep}_\Db$. 

It is well-known that $\hat\psi$ is an isomorphism when $\Db$ is absolutely irreducible, i.e.\ when $\rho^{ss}_\Db$ is absolutely irreducible (see \cite{nyssen1996, rouquier1996} in the case of pseudocharacters, and \cite[Thm.\ 2.22(i)]{chen2014} for pseudorepresentations). In this case, $\CRep_\Db$ is purely formal. 

\subsection{Algebraization of Moduli of Representations over Moduli of Pseudorepresentations}
\label{subsec:algebraization}
Our goal is to draw a conclusion about $\hat\psi$ similar to Theorem \ref{thm:psi_main}. As before, our principal tool will be the universal Cayley-Hamilton algebra of \S\ref{subsec:CHPI}. As nothing about its construction was particular to the finitely generated case, the universal continuous pseudodeformation of $\Db$
\[
D^u_\Db : \bZ_p\lb G\rb \otimes_{\bZ_p} R_\Db \lra R_\Db 
\]
induces the universal Cayley-Hamilton algebra
\begin{equation}
\label{eq:univ_CH_alg}
E(G)_\Db := (\bZ_p\lb G \rb \otimes_{\bZ_p} R_\Db) / \CH(D^u_\Db). 
\end{equation}
Because $D^u_\Db$ is continuous, it factors through the profinite completion $R_\Db\lb G \rb$ of $\bZ_p\lb G\rb \otimes_{\bZ_p} R_\Db$. Denote the factorization by $\tilde D^u_\Db : R_\Db\lb G \rb \ra R_\Db$ and denote the resulting Cayley-Hamilton algebra by
\[
\tilde E(G)_\Db := R_\Db\lb G \rb / \CH(\tilde D^u_\Db).
\]
As a consequence of the following important properties of the universal Cayley-Hamilton algebra, the two definitions are identical. 
\begin{prop}
\label{prop:E(G)}
Assume that the continuous cohomology group $H^1_c(G, \ad \rho^{ss}_\Db)$ is finite-dimensional over $\bF_\Db$. 
\begin{enumerate}
\item The natural map $E(G)_\Db \ra \tilde E(G)_\Db$ is a topological isomorphism.
\item The quotient map $R_\Db\lb G \rb \rsurj E(G)_\Db$ is continuous.
\item $E(G)_\Db$ is module-finite as an $R_\Db$-algebra, and therefore Noetherian. 
\item On $E(G)_\Db$, the profinite topology, the $\mDb$-adic topology, and the quotient topology from the surjection $R_\Db \lb G \rb \rsurj E(G)_\Db$ are equivalent. 
\item When $\Db$ is multiplicity-free, $E(G)_\Db$ is a generalized matrix algebra with canonical pseudorepresentation equivalent to $D^u_\Db$. In particular, when $\Db$ is absolutely irreducible, $E(G)_\Db \simeq M_d(R_\Db)$.
\end{enumerate}
\end{prop}
\begin{proof}
The two-sided ideal $\CH(\tilde D^u_\Db) \subset R_\Db \lb G \rb$ is a two-sided ideal generated by the image of $\chi^{[\alpha]}(r_1, \dotsc,r_d)$ where $\alpha$ varies over the finite set $I_d^d$ and $r_i$ vary over $R$, where we are using the notation and notions of \cite[\S1.10]{chen2014}. For the moment, let $R$ denote $R_\Db \lb G \rb$ and let $R_l, R_r$ be copies of $R$ distinguished for notational purposes. Let $(r_*^\alpha)$ denote an element of $R_*^{I_d^d}$ where $* = r,l$, and let $(r_i)$ denote an element of $R^d$.  Now define a continuous function 
\[
\begin{aligned}
R_l^{I_d^d} \times R^d \times R_r^{I_d^d} &\lra R \\
((r_l^\alpha) , (r_i), (r_r^\alpha)) &\mapsto \sum_{\alpha \in I_d^d}  r_l^{\alpha}\cdot \chi_D^{[\alpha]}((r_i)) \cdot r_r^\alpha.
\end{aligned}
\]
The image of this map is precisely the two-sided ideal generated by the image of the $\chi^{[\alpha]}$, i.e.\ $\CH(\tilde D^u_\Db)$, and it is closed by the closed map lemma, proving (2). Consequently, the quotient map $R_\Db \lb G \rb \rsurj \tilde E(G)_\Db$ induces a quotient topology equivalent to the profinite topology. 

We now work with $\bar E := \tilde E(G)_\Db \otimes_{R_\Db} \bF_\Db$. We wish to show that $\bar E$ is finite-dimensional over $\bF_\Db$. Firstly, Theorem \ref{thm:chen_2.16}(1) gives us that $\bar E/\ker(\Db)$ is finite-dimensional as a $\bF_\Db$-vector space. Because $\bar E$ is a Cayley-Hamilton algebra over a field, we may also apply Proposition \ref{prop:N(d)} to conclude that $\ker(\Db) \subset \bar E$ is nilpotent. Because of the natural surjection $\left(\ker(\Db)/\ker(\Db)^2\right)^{\otimes i} \rsurj \ker(\Db)^i/\ker(\Db)^{i+1}$ induced by multiplication, it will suffice to show that $\ker(\Db)/\ker(\Db)^2$ is finite-dimensional. We now invoke the finiteness of $H^1_c(G, \ad \rho^{ss}_\Db)$, which, by the argument of \cite[Prop.\ 3.35]{chen2014}, contains $\ker(\Db)/\ker(\Db)^2$ as a sub-vector space. 

The finiteness of the $\bF_\Db$-dimension of $\bar E$ along with the fact that $E(G)_\Db$ is clearly $\mDb$-adically separated implies that $E(G)_\Db$ is finite as a $R_\Db$-module. This completes (3), from which (4) follows. 

Observe that the image of the composite map $\bZ_p\lb G \rb \otimes_{\bZ_p} R_\Db \rsurj E(G)_\Db \ra \tilde E(G)_\Db$ is dense and is an $R_\Db$-submodule. Since $\tilde E(G)_\Db$ is a finite $R_\Db$-module, we have an isomorphism $E(G)_\Db \risom \tilde E(G)_\Db$, establishing part (1).

The algebra structure results of part (5) come from \cite[Thm.\ 2.22]{chen2014}, and the equivalence of the pseudorepresentations was proved in Theorem \ref{thm:CH_MF_GMA}.
\end{proof}

Using these results on $E(G)_\Db$, we get a result for continuous representations analogous to Proposition \ref{prop:R_E_equivalent}. 
\begin{thm}
\label{thm:R_E_rep_equiv}
Assume that $\dim_{\bF_\Db} H^1_c(G, \ad \rho^{ss}_\Db)$ is finite. There is a functorial equivalence of categories between continuous representations of $G$ with coefficients in $\Adm_{\bZ_p}$ with induced residual pseudorepresentation $\Db$ and continuous representations $(V_B, \rho_B)$ of $E(G)_\Db$ with coefficients in $\Adm_{\bZ_p}$ compatible with the universal pseudorepresentation $D^u_\Db : E(G)_\Db \ra R_\Db$. In particular, there is an isomorphism 
\[
\CRep_\Db \lrisom \CRep_{E(G)_\Db, D^u_\Db}
\]
of $\Spf \bZ_p$-groupoids. 
\end{thm}

\begin{proof}
Choose $B \in \Adm_{\bZ_p}$ and $(V_B, \rho_B) \in \CRep_\Db(B)$. The continuous homomorphism $\rho_B: \bZ_p\lb G\rb \otimes_{\bZ_p} B \ra \End_B(V_B)$ induces a continuous homomorphism $R_\Db \ra B$ corresponding to the pseudorepresentation $\hat\psi(V_B)$, giving $B$ the structure of an admissible $R_\Db$-algebra. Therefore there exists a canonical continuous map $\bZ_p\lb G \rb \otimes_{\bZ_p} R_\Db \ra \End_B(V_B)$ which induces $\rho_B$ by applying $\otimes_{R_\Db} B$ to the source. This factors uniquely and continuously through the Cayley-Hamilton quotient $E(G)_\Db$ by Propositions \ref{prop:CH} and \ref{prop:E(G)}(2,4), giving an object of $\CRep_{E(G)_\Db, D^u_\Db}$. Functorality may be checked using the stability of the Cayley-Hamilton quotient under base change described in Proposition \ref{prop:CH}. 

There exists an inverse functor, associating to an admissible $R_\Db$-algebra $B$ with $(V_B, \eta_B) \in \CRep_{E(G)_\Db, D^u_\Db}$ the $\bZ_p$-algebra $B$ and the composite map $\bZ_p \lb G \rb \ra E(G)_\Db \buildrel{\eta_B}\over \lra \End_B(V_B)$. 
\end{proof}

The following important result shows that $\hat\psi$ is algebraizable. This allows us to deduce further properties of $\hat\psi$ from the study of \S\ref{sec:Rep}. We write $\PsR_\Db$ for $\Spec R_\Db$. 

Following \cite[\S10.13]{ega1}, we will call a homomorphism $A \ra B$ in $\Adm_{\bZ_p}$ ``formally finitely generated'' if this map is compatible with a presentation of $B$ as a quotient of a restricted power series ring $A \langle x_1, \dotsc, x_n\rangle$, and use the term ``formally finite type'' to describe the corresponding morphisms of formal schemes. 

\begin{thm}
\label{thm:PsR_background}
Assume that $\dim_{\bF_\Db} H^1_c(G, \ad \rho^{ss}_\Db)$ is finite. The groupoids $\CRep_\Db$ and $\CRepB_\Db$ (resp.\ the functor $\CRep^\square_\Db$) over $\Spf \bZ_p$ are representable by formally finite type formal algebraic stacks (resp.\ formally finite type formal scheme) over $\Spf R_\Db$ which are algebraizable of finite type over $\Spec R_\Db$ with algebraizations $\Rep_\Db := \Rep_{E(G)_\Db, D^u_\Db}$ and $\RepB_\Db := \RepB_{E(G)_\Db, D^u_\Db}$ (resp.\ $\Rep^\square_\Db := \Rep^\square_{E(G)_\Db, D^u_\Db}$).

The natural morphism $\psi: \Rep_\Db \ra \PsR_\Db$ has $\frmm_\Db$-adic completion $\hat\psi$ and satisfies the following properties. 
\begin{enumerate}
\item $\psi$ is universally closed, 
\item $\psi$ has connected geometric fibers with a unique closed point corresponding to the unique semi-simple representation inducing the pseudorepresentation corresponding to the base of the fiber, 
\item $\psi$ consists of an adequate moduli space $\Rep_\Db \ra \psi_*(\cO_{\Rep_\Db})$ followed by an adequate homeomorphism $\nu: \psi_*(\cO_{\Rep_\Db}) \ra \Spec R_\Db$, and  
\item if $\Db$ is multiplicity free, then $\psi$ is precisely an adequate moduli space; moreover, it is a good moduli space.
\end{enumerate}
\end{thm}

The condition on $\nu$ in (3) means that there exist finitely many $p$-power torsion nilpotents $x_i \in \psi_*(\cO_{\Rep_\Db})$, $y_j \in R_\Db$ such that $R_\Db$ is generated over $\psi_*(\cO_{\Rep_\Db})$ by $\{y_j\}$ and the kernel of $\nu^*: \psi_*(\cO_{\Rep_\Db}) \ra R_\Db$ is generated by $\{x_i\}$. 

\begin{proof}
Under the assumption that $\dim_{\bF_\Db} H^1_c(G, \ad \rho^{ss}_\Db)$ is finite, Propositions \ref{prop:R_Db_noeth} and \ref{prop:E(G)} tell us that $R_\Db$ is a Noetherian ring and that the $R_\Db$-algebra $E(G)_\Db$ is finite as a $R_\Db$-module and has a GMA structure compatible with $D^u_\Db$ when $\Db$ is multiplicity-free. Theorem \ref{thm:R_E_rep_equiv} allows us to study representations of $E(G)_\Db$ in place of those of $G$. Then the existence and properties of each of the representation functors $\Rep^\square_\Db$, $\Rep_\Db$, $\RepB_\Db$ and $\psi$ are the content of \S\ref{sec:Rep}. More precisely, we have these statements for $\Rep^d_{E(G)_\Db}$ over $\PsR^d_{E(G)_\Db}$, and the condition that the induced pseudorepresentation lies over $D^u_\Db$ simply cuts out a connected component (namely $\Spec R_\Db$) of the base $\PsR^d_{E(G)_\Db}$. In particular, we apply Theorem \ref{thm:CH_MF_GMA} to obtain part (4). 

In order to see that the $\mDb$-adic completion of $(\Rep_\Db, \psi)$ is $(\CRep_\Db, \hat\psi)$ (and similarly for $\RepB_\Db$, $\Rep^\square_\Db$), we observe that all (non-topological) homomorphisms from $R_\Db$ to admissible $\bZ_p$-algebras $B$ or from $E(G)_\Db$ to $\End_B(V_B)$ are automatically continuous, as $E(G)_\Db$ is a finite $R_\Db$-module with the $\mDb$-adic topology.
\end{proof}

\begin{rem}
The condition ``constant residual pseudorepresentation $\Db$'' is no real restriction to the scope of Theorem \ref{thm:PsR_background} in view of the bijective correspondence between $d$-dimensional residual pseudorepresentations $\Db$ and connected components of $\CRep^d_G$. Putting together the connected components, we can say that there is an algebraization $\Rep^d_G$ of $\CRep^d_G$ of finite type over $\PsR^d_G = \coprod_{\Db} \Spec R_\Db$.
\end{rem}

This algebraization result implies that the topology on an integral $p$-adic family $V_A$ of representations of $G$ with coefficients in $A \in \Adm_{\bZ_p}$ and residual pseudorepresentation $\Db$ can always be strengthened to the $\mDb$-adic topology, and that there are subrings of $A$ that are finitely generated over $R_\Db$ over which a model for $V_A$ exists. 
\begin{cor}
\label{cor:G_coefs}
Let $A \in \Adm_{\bZ_p}$ and let $(\rho_A, V_A) \in \CRep_\Db(A)$. Then with assumptions as in Theorem \ref{thm:PsR_background}, 
\begin{enumerate}
\item the $G$-action on $V_A$ remains continuous for the possibly stronger $\mDb A$-adic topology on $A$, where $A$ has the structure of a continuous $R_\Db$-algebra induced by the pseudorepresentation $\det \circ \rho_A$,
\item there exists a minimal formally finitely generated sub-$R_\Db$-algebra $A' \subseteq A$ with a model $V_{A'} \in \CRep_\Db(A')$ for $V_A$, i.e.\ $V_A \simeq V_{A'} \otimes_{A'} A$, and
\item there exists a minimal finitely generated $\mDb$-adically separated sub-$R_\Db$-algebra $A'_\alg \subseteq A'$ with a model $V_{A'_\alg} \in \Rep_\Db(A'_\alg)$ such that $A'$ is the $\mDb$-adic completion of $A'_\alg$ and $V_{A'} \simeq V_{A'_\alg} \otimes_{A'_\alg} A'$.
\end{enumerate}
\end{cor}
The corollary follows directly from the fact that $\hat\psi$ is formally finite type and algebraizable by the finite type morphism $\psi$, or, alternatively, directly from Proposition \ref{prop:E(G)}. In particular, this means that the matrix coefficients of a family of representations of $G$ with residual pseudorepresentation $\Db$ generate a finite type algebra over $R_\Db$. 

The algebraization theorem also suggests that there exists a notion of continuous $D^u_\Db$-compatible representation of $G$ valued in an arbitrary $R_\Db$-algebra, i.e.\ not just those in $\Adm_{\bZ_p}$. When such an algebra $A$ is $\mDb$-adically separated, the $\mDb$-adic topology may be used, and the usual notions of continuity may be applied. On the other hand, there are common cases of concern where $A$ is not $\mDb$-adically separated. For example, one often wants to consider continuous Galois representations with coefficients in $\bQ_p$. However, $\bQ_p$ will never be $\mDb$-adically separated because $p \in \mDb$ is a unit in $\bQ_p$. 
 
\begin{defn}
Let $A$ be a $R_\Db$-algebra. A continuous representation of $G$ valued in $\GL_d(A)$ with residual pseudorepresentation $\Db$ is a group homomorphism $\rho: G \ra \GL_d(A)$ arising from an $R_\Db$-algebra morphism $E(G)_\Db \ra M_d(A)$ compatible with the structure morphism $R_\Db \ra A$ by concatenating it with the natural map $G \ra E(G)_\Db^\times$. 
\end{defn}
The moduli functor $\Rep^{\square}_\Db$ parameterizes these continuous representations. As the image of $G$ under $\rho$ lies in a module-finite $R_\Db$-subalgebra of $M_d(A)$, the $\mDb$-adic topology on this subalgebra is separated and $\rho$ is continuous with respect to this topology. 

\begin{rem}
Isomorphism classes of continuous representations of $G$ valued in a $p$-adic field $E$ with residual pseudorepresentation $\Db$ are in bijective correspondence with $E$-points of $\Rep_\Db$. Moreover, such representations uniquely characterize $\Rep_\Db$ in the following sense. Let $\Rep_\Db[1/p]$ represent the generic fiber of $\Rep_\Db$ over $\Spec \bZ_p$. The $p$-adic field-valued representations of $G$ control the reduced substack $\Rep_\Db[1/p]^{\mathrm{red}}$. Likewise, representations of $G$ with coefficients in finite $\bQ_p$-algebras $B$ control $\Rep_\Db[1/p]$. These claims are proved in Lemma \ref{lem:jacobson}(3). 

Also, notice that when $\Db$ is not absolutely irreducible, $\Rep_\Db$ is generally non-separated. If $\cO_E$ is the ring of integers of a $p$-adic field $E$, then there may be multiple $\cO_E$ points inducing a single $E$-point, reflecting the existence of multiple isomorphism classes of $G$-stable $\cO_E$-lattices in an $E$-valued representation. 
\end{rem}

\begin{eg}
Let $\bar\rho: G \ra \GL_d(\bF)$ be a residual representation with induced residual pseudorepresentation $\Db := \det \circ \bar\rho$. Write $R_{\bar\rho}$ for the versal deformation ring of $\bar\rho$ with versal representation $V_{\bar\rho}$. There exists a canonical map $R_\Db \ra R_{\bar\rho}$, and $V_{\bar\rho}$ is continuous with respect to the $\mDb$-adic topology on $R_{\bar\rho}$, which is often a strictly stronger topology than its native topology. Also, there exists a canonical, finite-type, $\mDb$-adically separated $R_\Db$-subalgebra $R_{\bar\rho, \alg}$ of $R_{\bar\rho}$ with a canonical $\mDb$-adically continuous representation $V_{\bar\rho, \alg}$ such that $V_{\bar\rho} \simeq V_{\bar\rho, \alg} \otimes_{R_{\bar\rho, \alg}} R_{\bar\rho}$. 
\end{eg}

\begin{rem}
The influence that Galois cohomology exerts on the structure of $R_{\bar\rho}$ is well-understood. Analogously, appropriate Galois cohomology groups control the structure of $\Rep_\Db$, which will be explained in forthcoming work. See the following example for a basic case. 
\end{rem}

Let us give an explicit example of a fiber of $\psi$, illustrating how $\psi$ satisfies the properties of Theorem \ref{thm:PsR_background}.
\begin{eg}
\label{eg:fiber}
Let $G = G_{\bQ_p}$ where $p > 3$ and let $\Db = \psi(\bar\chi \oplus 1)$ over $\bF_p$, where $\bar\chi$ is the mod $p$ cyclotomic character. Then, using local Tate duality, we calculate that the special fiber $\psi^{-1}(\Db)$ in $\Rep_\Db$ consists of
\begin{itemize}
\item extensions of $1$ by $\bar\chi$, parameterized by $\bP^1_{\bF_p} \cong \bP(\Ext^1_G(1,\bar\chi))$,
\item the semi-simple representation $\bar\chi\oplus 1$, and
\item extensions of $\bar\chi$ by $1$, parameterized by the trivial projective space $\mathrm{pt} = \bP^0_{\bF_p} \cong \bP(\Ext^1_G(\bar\chi,1))$.
\end{itemize}
Notice that the only closed point in $\psi^{-1}(\Db)$ is $\bar\chi\oplus 1$, because each of the projective spaces are open in the fiber, in analogy to the open immersion $\bP^{n-1} \rinj [\bA^n / \bG_m]$. 
\end{eg}

\subsection{Consequences of Formal GAGA for $\psi$}
\label{subsec:FGAMS}
In order to descend closed loci under $\psi$, it will be helpful to know formal GAGA for $\psi$. By ``formal GAGA for a morphism $f$'' we mean that the completion functor on coherent sheaves on the domain of $f$ is an equivalence of categories. The classical case is a proper morphism \cite[Cor.\ 5.1.3]{ega3-1}. We also know formal GAGA for $\psi$ when $\psi$ a good moduli space by \cite{GZB2015}. Even when $\psi$ is not a good moduli space, it satisfies formal GAGA under the following hypothesis. 
\[
\text{(FGAMS)} \quad \begin{array}{ll} \text{Formal GAGA holds for adequate moduli spaces} \\ \qquad \qquad \text{realized as quotient stacks by } \GL_d.\end{array}
\]
Consequently, we have
\begin{thm}
\label{thm:gaga}
 If $\Db$ is multiplicity free, or if the assumption (FGAMS) holds, then formal GAGA holds for the morphism $\psi: \Rep_\Db \ra R_\Db$. 
\end{thm} 
\begin{proof}
If $\Db$ is multiplicity free, $\psi$ is a good moduli space by Theorem \ref{thm:PsR_background}, and \cite{GZB2015} shows that formal GAGA holds for good moduli spaces. Otherwise, $\psi$ is an adequate moduli space followed by a finite morphism, and (FGAMS) implies that formal GAGA holds along $\psi$.
\end{proof}

\begin{rem}
According to the authors of \cite{GZB2015}, it is unclear whether to expect that formal GAGA holds for adequate moduli spaces. However, they can prove formal GAGA for adequate moduli spaces such as $BG$ for $G$ a reductive algebraic group.
\end{rem}

The following lemma shows how we can apply formal GAGA. The foremost use will be to algebraize loci of Galois representations that we initially produce only formally. (See, however, Remark \ref{rem:GAGA}.)

\begin{lem}
\label{lem:AG_background}
Let $(\cR, \frmm_\cR)$ be a complete Noetherian local $\bZ_p$-algebra, and let $X$ be an algebraic stack of finite type over $\Spec \cR$. Write $\hat X$ for its $\frmm_\cR$-adic completion. We assume that formal GAGA holds for $X$ over $\Spec \cR$.
\begin{enumerate}
\item There is a natural bijective correspondence between 
\begin{enumerate}
\item projective morphisms $Y \ra X$ and projective morphisms $\hat Y \ra \hat X$,
\item finite schematic morphisms $Y \ra X$ and finite schematic morphisms $\hat Y \ra \hat X$, and 
\item closed immersions $Y \rinj X$ and closed immersions $\hat Y \rinj \hat X$.
\end{enumerate}
\item If $Y$ is a finite type $\Spec \cR$-scheme that is a presentation $Y \ra X$ of $X$, then 
\begin{enumerate}
\item $\hat Y \ra \hat X$ is a fppf cover of $\hat X$ as a map of formal algebraic stacks.
\item $\hat Y \ra X$ is a fpqc cover of $X$ as a map of algebraic stacks.
\item $\hat Y[1/p] \ra X[1/p]$ is a fpqc cover (as a map of algebraic stacks).
\end{enumerate}
\end{enumerate}
\end{lem}

\begin{rem}
\label{rem:local_proj}
It is important to specify the notion of ``projective morphism,'' as there are definitions which differ over non-local bases. A projective morphism over a scheme $S$ is a morphism of the form $\Proj_{\cO_S} \cB$ for some quasi-coherent sheaf $\cB = \bigoplus_{i \geq 0} \cB_i$ of graded algebras which is generated by $\cB_1$ and where $\cB_1$ is finite type. As we will work in the case of a Noetherian base, we note that this notion of projectivity is Zariski-local on the base provided that the data of an ample line bundle is included with the morphism (cf.\ \cite[\S17.3.4]{vakil2014}). 
\end{rem}

\begin{proof}
To prove (1), one observes that each of the possible types of schematic morphisms over $X$ or $\hat X$ is controlled by coherent sheaves, whence the statements follow from formal GAGA. The cases (b) and (c) are covered by \cite[Prop.\ 5.4.4]{ega3-1}. The case (a) is also controlled by coherent sheaves, as a projective morphism $Y \ra X$ is by definition $Y = \Proj_{\cO_X} \cB$, where the quasi-coherent $\cO_X$-algebra $\cB = \bigoplus_{i \geq 0} \cB_i$ is a direct sum of coherent sheaves with multiplication law composed of morphisms of coherent sheaves. The corresponding projective morphism to $\hat X$ is given by $\hat Y = \Proj_{\cO_{\hat X}} \bigoplus_{i \geq 0} \hat \cB_i$.

Part (2a) is clear.  Because $Y$ is locally Noetherian, $\hat Y \ra Y$ is flat. Since $Y \ra X$ is smooth, $\hat Y \ra X$ is then flat as well. It is also clearly quasi-compact. The surjectivity of $\hat Y \ra X$ may be deduced as follows: For any point $z \in X$, its closure in the Zariski topology is realized by a closed substack $Z \rinj X$ (cf.\ \cite[Cor.\ 5.6.1(ii)]{lmb}). Then there is a point of $\hat Y$ lying over $z$, namely, a generic point of the base change $\hat Z \times_{\hat X} \hat Y$  of the closed substack $\hat Z \rinj \hat X$ corresponding to $Z \rinj X$ by (1c), completing the proof of (2b). Finally, (2b) implies (2c).  
\end{proof}

\section{Families of \'Etale $\phz$-modules and Kisin Modules}
\label{sec:etale_kisin}

After introducing notation in \S\ref{subsec:PHT_background_1}, we will describe the main point of this section in \S\ref{subsec:families_etale}.

\subsection{Background for Families of $G_{K_\infty}$-Representations of Bounded $E$-height}
\label{subsec:PHT_background_1}

For a reference to the following fundamental definitions in $p$-adic Hodge theory, see e.g.\ \cite{BCon2009}. 

Let $k$ be a finite field of characteristic $p > 0$ and $W := W(k)$ its ring of $p$-typical Witt vectors.  $W$ is the ring of integers of the finite unramified extension $K_0 := W(k)[1/p]$ of $\bQ_p$.  Let $K/K_0$ be a totally ramified extension of degree $e$.  Fix an algebraic closure $\bar K$ of $K$, and a completion $\bC_p$ of $\bar K$ and let $G_K := \Gal(\bar K/K)$.

We recall the definitions of some $p$-adic period rings. Let $\cO_{\bar K}$ be the ring of integers of $\bar K$ and $\cO_{\bC_p}$ the ring of integers of $\bC_p$.  Let $R = \varprojlim \cO_{\bar K}/p$, where each transition map is the Frobenius endomorphism of the characteristic $p$ ring $\cO_{\bar K}/p$.  This is a complete valuation ring which is perfect of characteristic $p$ and whose residue field is $\bar k$ and is also canonically a $\bar k$-algebra.  The fraction field $\Fr R$ of $R$ is a complete non-archimedean algebraically closed characteristic $p$ field. The elements $x$ of $R$ are in natural bijection with sequences of elements $(x_{(n)})_{n \geq 0}$ of $\cO_{\bC_p}$ such that $x_{(n+1)}^p = x_{(n)}$ for all $n \geq 0$.  A canonical valuation on $R$ is given by taking the valuation $v$ on $\bC_p$ normalized so that $v(p) = 1$ and setting $v_R((x_{(n)})_{n\geq 0}) = v(x_{(0)})$.  

Consider the ring $W(R)$, and write an element of $W(R)$ as $(x_0, x_1, \dotsc, x_n, \dotsc)$.  There is a unique continuous surjective $W$-algebra map 
\[
\begin{aligned}
\theta: W(R) &\lra \cO_{\bC_p}\\
(x_0, x_1, \dotsc) &\mapsto \sum_{n=0}^\infty p^n x_{n, (n)}
\end{aligned}
\]
lifting the projection to the first factor $R \ra \cO_{\bar K}/p$ onto the $0$th truncation $W_0(R)$ of the limit of truncated Witt vectors defining  $W(R)$.  The natural Frobenius action on $R$ induces a Frobenius map $\varphi$ on $W(R)$ which sends $(x_0, x_1, \dotsc)$ to $(x_0^p, x_1^p, \dotsc)$.

We fix the notation $\frS := W\lb u \rb$, the power series ring in the variable $u$.  We equip $\frS$ with a Frobenius map denoted $\varphi$, which acts by the usual Frobenius map on $W$ and sends $u$ to $u^p$.  We think of $\frS$ as the functions bounded by $1$ on the open analytic unit disk over $K_0$, and $\frS[1/p]$ as the ring of bounded functions on the open unit disk.  Fix a uniformizer $\pi \in K$, and elements $\pi_n$\footnote{In the notation above, these would be $\pi_{(n)}$.} for $n \geq 0$ such that $\pi_0 = \pi$ and $\pi_{n+1}^p = \pi_n$.  Write $E(u) \in W[u]$ for the minimal, Eisenstein polynomial of $\pi$ over $K_0$.

Write $\underline{\pi} := (\pi_n)_{n \geq 0} \in R$, and let $\tupi \in W(R)$ be its Teichm\"uller lift $(\underline{\pi}, 0, 0, \dotsc)$.  Because $R$ is canonically a $\bar k$-algebra, we have a canonical embedding $W \rinj W(\bar k) \rinj W(R)$.  We consider $W(R)$ as a $W[u]$-algebra by sending $u$ to $\tupi$.  Since $\theta(\tupi) = \pi$, this embedding extends to an embedding of $\frS$ into $W(R)$, and we will consider $W(R)$ and rings derived from $W(R)$ as $\frS$-algebras via this map from now on. From the discussion above, this map is visibly $\varphi$-equivariant. 

We define another important element $\tuep \in W(R)$. Firstly define a sequence of $p^n$th roots of unity
\begin{equation}
\label{eq:tuep_defn}
\varep_0 = 1, \varep_1 \neq 1, \text{ and } \varep_{n+1}^p = \varep_n \quad \forall n \geq 0.
\end{equation}
This sequence defines an element $\underline{\varep}$ in $R$.  Let $\tuep \in W(R)$ be its Teichm\"ulller lift.  Notice that $\theta(\tuep - 1) = 0$.

Let $\cO_\cE$ be the $p$-adic completion of $\frS[1/u]$.  Then $\cO_\cE$ is a discrete valuation ring with residue field $k\lp u\rp$ and maximal ideal generated by $p$.  Write $\cE$ for its fraction field $\Fr \cO_\cE = \cO_\cE[1/p]$. The inclusion $\frS \rinj W(R)$ extends to an inclusion $\cO_\cE \rinj W(\Fr R)$, since $\underline{\pi} \in \Fr R$ and $W(\Fr R)$ is $p$-adically complete.  Let $\cE^\ur \subset W(\Fr R)[1/p]$ denote the maximal unramified extension of $\cE$ contained in $W(\Fr R)[1/p]$, and $\cO_{\cE^\ur}$ its ring of integers.  Since $\Fr R$ is algebraically closed, the residue field $\cO_{\cE^\ur}/p\cO_{\cE^\ur}$ is a separable closure of $k\lp u\rp$. If $\cO_{\widehat{\cE^\ur}}$ is the $p$-adic completion of $\cO_{\cE^\ur}$, or, equivalently, the closure of $\cO_{\cE^\ur}$ in $W(\Fr R)$ with respect to its $p$-adic topology, set $\frS^\ur := \cO_{\widehat{\cE^\ur}} \cap W(R) \subset W(\Fr R)$.  All of these rings are subrings of $W(\Fr R)[1/p]$, and are equipped with a Frobenius operator coming from $W(\Fr R)[1/p]$.

Let $K_\infty = \cup_{n \geq 0} K(\pi_n)$ and $G_{K_\infty} := \Gal(\bar K/K_\infty)$.  Clearly the action of $G_{K_\infty}$ on $W(R)$ fixes the subring $\frS$, since it fixes both $W$ and $\pi_n \ \forall n \geq 0$.  Therefore $G_{K_\infty}$ has an action on $\frS^\ur$ and $\cE^\ur$. 
 
Recall that for any $\Zp$-algebra $S$, $S_A$ denotes the completion of $S \otimes_\Zp A$ with respect to a defining system of ideals for the topology of $A$. It will be important to know that the following such rings are Noetherian. 

\begin{lem}
\label{lem:frS_A_noetherian}
Let $\cR$ be an admissible local $\bZ_p$-algebra and let $A$ be a formally finite type $\cR$-algebra. The commutative rings $\frS_\cR$, $\frS_A$, $\cO_{\cE, \cR}$, and $\cO_{\cE, A}$ are Noetherian. Also, there are isomorphisms $\frS_\cR \simeq \frS \hat\otimes_{\bZ_p} \cR$, $\cO_{\cE, \cR} \simeq  \cO_\cE \hat\otimes_{\bZ_p} \cR$; the latter is a topological isomorphism. Each of these rings admits a unique $\cR$ or $A$-linear extension of the Frobenius map $\phz$ on $\frS$ or $\cO_\cE$. 
\end{lem}
\begin{proof}
As $\cR$, $\frS$, and $\cO_\cE$ are Noetherian, the rings $\frS \hat\otimes_{\bZ_p} \cR$ and $\cO_\cE \hat\otimes_{\bZ_p} \cR$ are Noetherian by \cite[Ch.\ 0, Lem.\ 19.7.1.2]{ega4-1}. As the ideals $(p \otimes 1) + \cO_\cE \otimes_{\bZ_p} \frmm_\cR$ and $\frmm_\cR \cO_{\cE, \cR}$ of $\cO_\cE \otimes_{\bZ_p} \cR$ are identical, we have a topological isomorphism $\cO_\cE \hat\otimes_{\bZ_p} \cR \simeq \cO_{\cE,\cR}$. There is also an isomorphism of rings $\frS_\cR \risom \frS \hat\otimes_{\bZ_p} \cR$, but it is not necessarily a topological isomorphism.

By \cite[Prop.\ 10.13.5(ii)]{ega1}, the rings $\cO_{\cE, A}$ and $\frS_A$ are Noetherian because $\Spf A \ra \Spf \cR$ is a formally finite type morphism of formal schemes.
\end{proof}

\subsection{Algebraic Families of \'{E}tale $\varphi$-modules}
\label{subsec:families_etale}

In this section, we will work with representations $V_A$ of $G_{K_\infty}$ with coefficients in admissible $\bZ_p$-algebras $A$ with the discrete topology, which are quotients of $\bZ/p^i[x_1, \dotsc, x_j]$ for some integers $i,j$.  Unlike the previous sections, we will not study the most general moduli space of these families, but simply fix such an $A$ and study the category of $A$-linear representations. Later, these results will be applied to a family of $G_K$-representations in $\Rep_\Db$, considered as a $G_{K_\infty}$-representation. In this section, we also fix $R$ to be an arbitrary Artinian (and, therefore, finite cardinality) subring of $A$. In \S\ref{subsec:universality_uM}, we apply these results where $R$ is $\bZ_p$-subalgebra of $A$ generated by characteristic polynomial coefficients of the $G_K$ or $G_{K_\infty}$-action. 

Our goal is to compare these families of Galois representations to \'{e}tale $\phz$ modules. These \'etale $\phz$-modules are finite modules $M$ over certain of the Noetherian rings of Lemma \ref{lem:frS_A_noetherian} with a $\phz$-semilinear automorphism. We will often write this automorphism as a linear automorphism $\phz^*(M) \ra M$, where $\phz^*(M)$ denotes $\cO_\cE \otimes_{\phz, \cO_\cE} M$. Because the $\mDb$-adic topology is discrete in our current setting, we have an isomorphism $\cO_{\cE, A} \cong \cO_\cE \otimes_{\bZ_p} A$, with an $A$-linear extension $\varphi$ of the Frobenius on $\cO_\cE$. 

\begin{defn-lem}
\label{defn-lem:M_and_V}
\leavevmode
\begin{enumerate}
\item Let $\Mod_{G_{K_\infty}}(A)$ be the category of finite $A$-modules with a $A$-linear action of $G_{K_\infty}$ with open kernel.  Let $\Rep_{G_{K_\infty}}(A)$ be the full subcategory whose objects are finite, projective, and constant rank  as $A$-modules.  
\item Let $\Phi'_M(A)$ be the category of finite $\cO_{\cE,A}$-modules $M$ equipped with an $A$-linear isomorphism $\varphi^*(M) \risom M$, called ($A$-linear) \emph{\'etale $\phz$-modules}.  Let $\Phi_M(A)$ be the full subcategory whose objects are finite, projective, and constant rank  as $\cO_{\cE,A}$-modules.
\item Let $M$ be the covariant functor
\[
\begin{aligned}
M: \Mod_{G_{K_\infty}}(A) & \lra \Phi'_M(A) \\
V_{A} &\mapsto (\cO_{\cE^\ur} \otimes_{\bZ_p} V_{A})^{G_{K_\infty}}.
\end{aligned}
\]
\item Let $\Phi^{' \Gal}_M(A)$ be the essential image of $M$ on $\Mod_{G_{K_\infty}}(A)$ in $\Phi'_M(A)$, and let $\Phi^{\Gal}_M(A)$ be the essential image of $M$ on $\Rep_{G_{K_\infty}}(A)$ in $\Phi_M(A)$.
\item Let $V$ be the covariant functor
\[
\begin{aligned}
V: \Phi^{' \Gal}_M(A) &\lra \Mod_{G_{K_\infty}}(A) \\
M_A &\mapsto (\cO_{\cE^\ur} \otimes_{\cO_\cE} M_A)^{\varphi = 1}.
\end{aligned}
\]
\end{enumerate}
\end{defn-lem}
It remains to be confirmed that parts (3) and (5) above are valid, e.g.\ that $M(V_A)$ is finite as a $\cO_{\cE, A}$-module when $V_A$ is finite as an $A$-module. The proof will be given below after stating one more result. 

\begin{rem}
\label{rem:no_filtration}
The functor $V$ behaves well only after restriction to the full subcategory $\Phi^{' \Gal}_M(A)$ of $\Phi'_M(A)$. In the proof of the following proposition, we will see that the obstruction to $M$ being essentially surjective is that $M_A \in \Phi'_M(A)$ may not be a filtered direct limit of finite $\cO_{\cE,R}$-submodules $M_i$ such that the structure $\varphi^*(M_A) \risom M_A$ is the limit of such maps on $M_i$. Take for example $A = \Fp[x,x^{-1}]$ and $M_A$ free of rank one over $\cO_{\cE,A}$ with the map $\phz^*(M_A)\ra M_A$ defined by $1 \otimes z \mapsto xz$. In fact, $V(M_A) = 0$. 
\end{rem}

When $A$ is Artinian, the functors $M$ and $V$ define an equivalence between $\Rep_{G_{K_\infty}}(A)$ and $\Phi_M(A)$ that is functorial in $A$; see \cite[Lem.\ 1.2.7]{mffgsm}, which builds upon the original case with trivial ($A = \Zp$) coefficients \cite[A.1.2.6]{fontaine1990}. For more general $A$, this is not true, and the following result gives as much as one can ask for. 

\begin{prop}[{Generalizing \cite[Lem.\ 1.2.7]{mffgsm}}]
\label{prop:main_etale} \leavevmode 
\begin{enumerate}
\item The functor $M: \Mod_{G_{K_\infty}}(A) \ra \Phi'_M(A)$  is exact and fully faithful, and is an equivalence onto the full subcategory $\Phi^{' \Gal}_M(A)$ with quasi-inverse $V$.
\item For $W$ a finite $A$-module and $V_A \in \Rep_{G_{K_\infty}}(A)$, there is a natural isomorphism
\[
M(V_A \otimes_A W) \cong M(V_A) \otimes_A W.
\]
\item If $A'$ is a finitely generated $A$-algebra, then there is a commutative diagram of functors
\[
\xymatrix{
\Mod_{G_{K_\infty}}(A) \ar[r]^(.55)M \ar[d] & \Phi^{' \Gal}_M(A) \ar[d] \\
\Mod_{G_{K_\infty}}(A') \ar[r]^(.55)M & \Phi^{' \Gal}_M(A')  
}
\]
where the downward functors are induced by $- \otimes_A A'$.
\item $M$ restricts to an equivalence of categories 
\[
M: \Rep_{G_{K_\infty}} \lrisom \Phi^{\Gal}_M(A)
\]
with quasi-inverse 
\[
V: \Phi^{\Gal}_M(A) \lrisom \Rep_{G_{K_\infty}}(A). 
\]
In particular, 
\begin{enumerate}
\item $V_A$ is projective as an $A$-module of constant rank  $d$ if and only if $M(V_A)$ is a projective $\cO_{\cE, A}$-module of constant rank  $d$.
\item $V_A$ is free as an $A$-module with rank $d$ if and only if $M(V_A)$ is a free $\cO_{\cE,A}$-module of rank $d$.
\end{enumerate}
\end{enumerate}
\end{prop}

First we assemble useful facts about limits.  We will append $(-)^\infty$ to categories defined in Definition/Lemma \ref{defn-lem:M_and_V} to indicate that the $A$-module finiteness condition has been dropped. 

\begin{fact}
\label{fact:tensor_colimit}
In a category of modules over a ring, tensor products commute with direct limits. 
\end{fact}
\begin{fact}
\label{fact:invariant_colimit}
If the maps of a filtered direct limit of finite modules in $\Mod_{G_{K_\infty}}^\infty(R)$ (resp.\ in $\Phi^{' \infty}_M(R)$) are all injective, then the functor $(-)^{G_{K_\infty}}$ (resp.\ $(-)^{\varphi = 1}$) commutes with this direct limit.
\end{fact}
\begin{fact}
\label{fact:invariant_limit}
Inverse limits in $\Mod_{G_{K_\infty}}(R)$ (resp.\ $\Phi'_M(R)$) commute with the invariant functor $(-)^{G_{K_\infty}}$ (resp.\ $(-)^{\varphi = 1}$), since it is a right-adjoint functor to the functor giving the trivial $G_{K_\infty}$-action to an $R$-module, and therefore commutes with limits.
\end{fact}

In order to substantiate Definition/Lemma \ref{defn-lem:M_and_V} and Proposition \ref{prop:main_etale}, $A$-linear structure on the objects will be forgotten down to $R$-linear structure. Then, the objects are direct limits of finite $R$-submodules for which the statements are known, and we establish appropriate compatibility with the limits. 

\begin{proof}[Proof (Definition/Lemma \ref{defn-lem:M_and_V})]
Let $V_A \in \ob \Mod_{G_{K_\infty}}(A)$.  Because the action of $G_{K_\infty}$ has a finite index kernel, we have a canonical isomorphism as $R[G_{K_\infty}]$-modules of $V_A$ with $\varinjlim_i V_i$, where $(V_i)_{i \in I} \in \ob \Mod_{G_{K_\infty}}(R)$ are the $R$-module-finite $R[G_{K_\infty}]$-submodules of $V_A$. We note that the functor $M$ (resp.\ $V$) commutes with injective direct limits in $\Mod_{G_{K_\infty}}(R)$ (resp.\ $\Phi'_M(R)$), using Facts \ref{fact:tensor_colimit} and \ref{fact:invariant_colimit} above along with the fact that the tensor product $\otimes_{\bZ_p} \cO_{\cE^\ur}$ (resp.\ $\otimes_{\cO_\cE} \cO_{\cE^\ur}$) preserves injective maps.

Therefore there are canonical isomorphisms in $\Phi^{' \infty}_M(R)$ of colimits of objects of $\Phi'_M(R)$,
\[
M(V_A) \lrisom M(\varinjlim_i V_i) \lrisom \varinjlim_i M(V_i),
\]
and the fact that $M$ is an equivalence of categories out of $\Mod_{G_{K_\infty}}(R)$ commuting with the necessary colimits implies that there is a canonical isomorphism respecting all structures
\begin{equation}
\label{eq:compare_V_M}
\xymatrix{
V_A \otimes_{\bZ_p} \cO_{\cE^\ur} \ar[rrr]^\sim_{\cO_{\cE^\ur}, R, G_{K_\infty}, \varphi} & & & M(V_A) \otimes_{\cO_\cE} \cO_{\cE^\ur}.
}\end{equation}
The $A$-linear structure on the left hand side then provides a canonical $A$-linear structure on the right hand side, commuting with the action of $\cO_{\cE^\ur}, G_{K_\infty},$ and $\varphi$.  Therefore, $M(V_A)$, being the $G_{K_\infty}$-invariants of the right hand side, has the structure of an $A$-module; moreover, it is an $\cO_{\cE,A}$-module with a Frobenius semi-linear endomorphism.  To complete the proof that $M$ is well-defined in Definition/Lemma \ref{defn-lem:M_and_V}, we must show that $M(V_A)$ is finite as an $\cO_{\cE,A}$-module.

Let $H$ be the open kernel of the action of $G_{K_\infty}$ on $V_A$.  Since $H$ acts trivially on $V_A$, the canonical isomorphism above induces a canonical isomorphism
\begin{equation}
\label{eq:finite_compare_V_M}
V_A \otimes_{\bZ_p} (\cO_{\cE^\ur})^H \lrisom M(V_A) \otimes_{\cO_\cE} (\cO_{\cE^\ur})^H.
\end{equation}
Since $G_{K_\infty}/H$ is finite and $(\cO_{\cE^\ur})^{G_{K_\infty}} = \cO_\cE$, we know that $(\cO_{\cE^\ur})^H$ is finite as a $\cO_\cE$-module.  Therefore the left hand side is finite as a $\cO_{\cE,A}$-module, so that the right hand side is as well.  Therefore $M(V_A)$ is a finite $\cO_{\cE,A}$-module by \'etale descent.  This confirms part (3) of Definition/Lemma \ref{defn-lem:M_and_V}.

Since $V$ commutes with the same limits as $M$ does and is quasi-inverse to $M$ on $\Phi'_M(R)$, we observe that $V$ defines a $R$-linear quasi-inverse on the essential image $\Phi^{' \Gal}_M(A)$ of $M$.  Then the finiteness of $V(M_A)$ for $M_A \in \Phi^{\Gal}_M(A)$ can be proven in the same way that the finiteness of $M(V_A)$ was proven, using the isomorphism \eqref{eq:finite_compare_V_M}.  This confirms part (5) of Definition/Lemma \ref{defn-lem:M_and_V}.
\end{proof}

\begin{proof}[Proof (Prop.\ \ref{prop:main_etale})]
We have proved part (1) in the argument for Definition/Lemma \ref{defn-lem:M_and_V}  that we just gave. For part (2), observe that this statement is clearly true for free, finite rank $A$-modules $W$; then, use the exactness of $M$ on a finite free presentation for a general finite $A$-module $W$. To prove part (3), write $A'$ as an increasing union of its finite $A$-submodules $A' = \varinjlim B_i$, and then apply part (2) along with the arguments involving compatibility with limits in the proof of Definition/Lemma \ref{defn-lem:M_and_V}. 

For part (4), first observe that the exactness of $M$ along with (3) implies that $M(V_A)$ is flat over $\cO_{\cE, A}$ if and only if $V_A$ is flat over $A$.  As these modules are finite over Noetherian rings, they are projective. Having verified this equivalence, to prove (4a) it suffices to show that if the rank of $V_A$ holds constant at $d$, then so does the rank of $M(V_A)$. Since both $V_A$ and $M(V_A)$ are flat, the rank function is locally constant.  At a maximal ideal $\frmm \subset A$, we know that the ranks $\dim_{A/\frmm} V_A \otimes_A A/\frmm$ and $\rk_{\cO_{\cE,A/\frmm}} M(V_A \otimes_A A/\frmm)$ are the same since, by (2),
\[
M(V_A) \otimes_A A/\frmm \cong M(V_A \otimes_A A/\frmm)
\]
and since $A/\frmm$ is a finite field, \cite[A.1.2.4(i)]{fontaine1990} tells us that the $\cO_{\cE,A/\frmm}$-rank of $M(V_A/\frmm)$ is constant and is the same as the $A/\frmm$-dimension of $V_{A/\frmm}$.    Because $A$ is $p$-power torsion, any maximal ideal $I$ of $\cO_{\cE,A}$ contains the kernel of the factor map $\cO_{\cE,A} \ra \cO_{\cE,A/\frmm}$ for some maximal ideal $\frmm$ of $A$. In other words, there exists a natural map $\MaxSpec \cO_{\cE,A} \ra \MaxSpec A$.  Because every connected component of $\Spec A$ and $\Spec \cO_{\cE,A}$ has a closed point, we have completed the proof of (4a).

For (4b), it suffices to show that $M(V_A)$ is free when $V_A$ is free.  The isomorphism \eqref{eq:finite_compare_V_M} shows that when one of $V_A$ or $M(V_A)$ is free, then $V_A \otimes_{\bZ_p} (\cO_{\cE^\ur})^H$ and $M(V_A) \otimes_{\cO_\cE} (\cO_{\cE^\ur})^H$ are both free $(\cO_{\cE^\ur,A})^H$-modules.  Because $\Spec (\cO_{\cE^\ur,A})^H \ra \Spec \cO_{\cE,A}$ is a finite surjective \'{e}tale morphism, any vector bundle trivialized by such a cover was already free over the base, proving the assertion.  Indeed, this is Hilbert Theorem 90 applied to this \'etale extension; that is, $\GL_d$ is special in the sense of Serre \cite[Exp.\ 1, \S4]{Serre1958}, cf.\ \cite[Exp.\ XI, \S5]{sga1}. 
\end{proof}

\subsection{Functors of Lattices and Affine Grassmannians}
\label{subsec:lattice_grass}

The assumptions on $A$, $R$, and $V_A$ remain the same as in the previous section. We will study lattices in the \'etale $\phz$-module $M_A := M(V_A^*)$, where $V_A^*$ denotes the $A$-linear dual of $V_A$ with its natural $G_{K_\infty}$-action. Since $p$ is nilpotent in $A$ (choose $i$ so that $p^i = 0$ in $A$), $\cO_{\cE, A} \cong (\bZ/p^i \bZ)[[u]][1/u] \otimes_{\bZ_p} A$.  Therefore $\frS_A[1/u] \cong \cO_{\cE,A}$, and we may consider $\frS_A$-lattices within $M_A$ which are stable under the Frobenius semi-linear endomorphism $\phz$ on $M_A$.  We will show in this section that the functor associating to an $A$-algebra $B$ the $\phz$-stable $\frS_B$-sublattices of $M_A \otimes_A B$ satisfying the condition ``$E$-height $\leq h$'' is represented by a projective $A$-scheme. We will use the affine Grassmannian for this, generalizing the result of \cite[\S2.1]{mffgsm} and \cite[\S1]{pssdr}, which was done in the case that $A$ is Artinian. 

First we will briefly review the theory of the affine Grassmannian; see \cite[\S1]{richarz2015} for thorough and general treatment of affine Grassmannians. Affine Grassmannians for $\GL_d$ and related groups (see below) are functors of sublattices of projective, constant rank  modules.  The local affine Grassmannian parameterizes these vector bundles over the formal one-dimensional disk $\cD$ which are trivialized on the punctured disk.  The global affine Grassmannian parameterizes these vector bundles over the affine line $\bA^1$ which are trivialized on the punctured line. 

\begin{defn}
Let $V_A$ be a projective rank $d$ $A$-module. Write 
\[
G = \Res_{W/\bZ_p}\GL_{A \otimes_{\bZ_p} W}(V_A^* \otimes_{\bZ_p} W).
\]
Then the affine Grassmannians we will require are the following functors.
\begin{enumerate}
\item The \emph{local affine Grassmannian} $\Gr^{loc}_{G}$ for $G$ is the functor associating to a $A$-algebra $B$ the set of pairs $(P_{\cD}, \eta)$ where $P_{\cD}$ is a projective rank $d$ $W_B\lb t \rb$-module and $\eta$ is an isomorphism
\[
P_{\cD} \otimes_{W_B\lb t\rb} W_B \lp t\rp \lrisom V_A^* \otimes_A W_B\lp t\rp.
\]
\item The \emph{global affine Grassmannian} $\Gr^{glob}_{G}$ for $G$ is the functor assigning to an $A$-algebra $B$ the set of pairs $(P_{\bA^1}, \eta)$, where $P_{\bA^1}$ is a projective rank $d$ $W_B[t]$-module and $\eta$ is an isomorphism
\[
P_{\bA^1} \otimes_{W_B[t]} W_B[t][1/t] \lrisom V_A^* \otimes_A W_B[t][1/t].
\]
\end{enumerate}
\end{defn}

We observe that there is a natural functor 
\begin{equation}
\label{eq:BL_theorem}
\Gr^{glob}_{G} \lra \Gr^{loc}_{G}
\end{equation}
given by restriction from a line to the disc, i.e.\ $\otimes_{W_B[t]} W_B\lb t\rb$. In fact, this restriction is an isomorphism \cite[Prop.\ 6.2]{PZ2013}. It is not difficult to deduce that this map is an isomorphism from published results. When $W = \bZ_p$ and $V_A$ is free, this is the result of \cite{BL1995}. When $V_A$ is not free, we can verify the result on a Zariski cover of $\Spec A$ trivializing $V_A$, and deduce the result from base change from $\Spec \bZ_p$ to $\Spec A$. The result over $\bZ_p$, for arbitrary $W$, follows from \cite[Prop.\ 6.2]{PZ2013}. 

We also want to know that $\Gr_{G}$ is ind-projective over $\Spec A$ with a canonical ample line bundle. Like before, we will reduce to the case $A = \bZ_p$ by replacing $\Spec A$ with a Zariski cover trivializing $V_A$. The fact that $\Gr_{\Res_{W/\bZ_p}\GL_d}$ is ind-proper over $\Spec \bZ_p$ is given in e.g.\ \cite[Prop.\ 6.3]{PZ2013}. We thank Brandon Levin for explaining the following construction of the canonical ample line bundle. Let $H$ be the linear automorphism group of $\Lie(\Res_{W/\bZ_p}\GL_d)$. $H$ is a general linear group over $\bZ_p$, and the ``determinant'' line bundle $\cL_{\det}$ of \cite[Definition 4.3.4]{levin2014} is ample on $\Gr_H$. Even though the adjoint action homomorphism $\Res_{W/\bZ_p}\GL_d \ra H$ may not be injective, the pullback of $\cL_{\det}$ along the induced morphism $\Gr_{\Res_{W/\bZ_p}\GL_d} \ra \Gr_H$ is ample on both fibers over $\Spec \bZ_p$ by \cite[Prop.\ 4.3.6]{levin2014}. Because we already know $\Gr_{\Res_{W/\bZ_p}\GL_d}/\Spec \bZ_p$ is ind-proper, it is therefore ind-projective as well. By base change from $\bZ_p$ to $A$, we know that $\Gr_{G}$ is ind-projective over $\Spec A$ locally on the base. And because the ample line bundle was canonically constructed, it interpolates over all of $\Spec A$. Because we have this canonical ample line bundle, ind-projectivity is preserved under gluing a Zariski cover in light of Remark \ref{rem:local_proj}. 

We summarize what we have proved in the following 
\begin{thm}
\label{thm:AG-IP}
Let $S$ be a locally Noetherian scheme, and let $V$ be a projective, coherent, constant rank  $\cO_S$-module.  Then $\Gr_{\Res_{W/\bZ_p} \GL_{W_A}(V \otimes_{\bZ_p} W)}$ is an ind-projective $S$-scheme with a canonical ample line bundle. 
\end{thm}

The functor of sublattices that arises in our study is not identical to the global nor the local affine Grassmannian, but it is scheme-theoretically isomorphic. In what follows, we write $\hat \frS_B$ for the $u$-adic completion of $\frS_B$; they are both Noetherian (Lem.\ \ref{lem:frS_A_noetherian}).

\begin{prop}
\label{prop:grass_equiv}
If $V_A$ is an object of $\Rep_{G_{K_\infty}}(A)$, projective of constant rank $d$, and $M_A := M(V_A^*)$ is the corresponding $\cO_{\cE, A}$-module in $\Phi^{\Gal}_M(A)$, then there exist equivalences of functors on $A$-algebras
\begin{enumerate}
\item The global affine Grassmannian $\Gr^{glob}_{G}$ for $G/A$.
\item The functor $F_{V_A}$ associating to a finitely generated $A$-algebra $B$ the $\frS_B$-sublattices of $M_B := M_A \otimes_A B$
\item The local affine Grassmannian $\Gr^{loc}_{G}$ for $G/A$.
 \end{enumerate}
induced by tensoring
\begin{equation}
\label{eq:glob_L_loc}
\xymatrix{
\Gr^{glob}_{G} \ar[rr]^{\otimes_{W_B[u]} \frS_B} & & F_{V_A} \ar[rr]^{\otimes_{\frS_B} W_B\lb u\rb} & & \Gr^{loc}_{G}
}
\end{equation}
which factors the composite isomorphism \eqref{eq:BL_theorem}.
\end{prop}

\begin{rem}
\label{rem:basis_choice}
We will see in the proof that the natural isomorphisms between the affine Grassmannians and $F_{V_A}$ is not canonical.  This is not a new phenomenon that arises when $A$ is no longer Artinian as it was in \cite{pssdr}; bases were implicitly chosen there as well. However, because of the nature of its construction, the canonical ample line bundle on $F_{V_A}$ does not depend on the choice of basis. 
\end{rem}

\begin{proof}
We will prove the case $W = \bZ_p$. First let us assume that $V_A$ is free of rank $d$, so that $M_A$ is as well, by Proposition \ref{prop:main_etale}(4). We observe that the two morphisms in \eqref{eq:glob_L_loc} factor \eqref{eq:BL_theorem}, and therefore it will suffice to show that the latter morphism $\otimes_{\frS_B} B\lb u\rb$ is a monomorphism of functors.

Choose a basis for $M_A$. For any finitely generated $A$-algebra $B$, let $\frM_B \in F_{V_A}(B)$ denote the $\frS_B$-lattice generated by the induced basis of $M_B = M_A \otimes_A B$. As remarked in \cite[Example 2.2.7]{levin2014}, we note that $F_{V_A}(B)$ is a direct limit over $n \geq 0$ of functors of lattices $\frN_B \subset M_B$ satisfying $u^{-n}\frM_B \supseteq \frN_B \supseteq u^n\frM_B$. 
This filtered direct limit exists for each of the global/local affine Grassmannian functors as well as $F_{V_A}$, and is compatible with and unchanged by the tensor maps of \eqref{eq:glob_L_loc}. Therefore the latter map of \eqref{eq:glob_L_loc} is an isomorphism.

In the case that $V_A$ is a projective, rank $d$ $A$-module trivialized by a Zariski cover $\Spec \tilde A \ra \Spec A$, Proposition \ref{prop:main_etale}(2,4) implies that the same cover trivializes $M_A$. Therefore the argument above applies after base change to $\Spec \tilde A$, and by descent we have the statement of the proposition.  
\end{proof}

The functor of $\frS$-lattices of $E$-height at most $h$ for $V_A$ is defined on the category of $A$-algebras as follows, where $h$ is a non-negative integer.  Recall that $M_A := M(V_A^*)$, and is equipped with an $\cO_{\cE,A}$-linear isomorphism $\phz: \varphi^*(M_A) \risom M_A$. 

\begin{defn}
For $B$ an $A$-algebra, $M_B = M_A \otimes_A B$ has a $B$-linear isomorphism $\varphi^*(M_B) \risom M_B$.  Choose a non-negative integer $h$.  A \emph{$\frS_B$-lattice in $M_B$ of $E$-height $\leq h$} is a $\frS_B$-sublattice $\frM_B$ of the $\cO_{\cE, B}$-module $M_B$ such that $\frM_B$ is stable by $\varphi$ and the cokernel of the induced injective map $\varphi^*(\frM_B) \ra \frM_B$ is killed by $E(u)^h$. 
\end{defn}

Write $L_{V_A}^{\leq h}(B)$ for the set of $\frS_B$-lattices of $E$-height at most $h$ in $M_B$.  It is naturally a functor on $A$-algebras (cf.\ \cite[p.\ 517]{pssdr}), and is a subfunctor of $F_{V_A}$. In the case that $A$ is Artinian, $L_{V_A}^{\leq h}$ is represented by a projective $A$-scheme \cite[Prop.\ 2.1.7]{mffgsm} (see also \cite[Prop.\ 1.3]{pssdr}). The same proof will apply in the non-Artinian case. 
\begin{prop}[{Generalizing \cite[Prop.\ 2.1.7]{mffgsm}}]
\label{prop:kisin_proj}
The functor $L_{V_A}^{\leq h}$ is represented by a projective $A$-scheme $\cL^{\leq h}_{V_A}$ with a canonical ample line bundle.  If $A \ra A'$ is finitely generated and $V_{A'} = V_A \otimes_A A'$, then there is a canonical isomorphism $\cL^{\leq h}_{V_A} \otimes_A A' \risom \cL^{\leq h}_{V_{A'}}$ compatible with the canonical ample line bundle.
\end{prop}

\begin{proof}
Proposition \ref{prop:grass_equiv} gives us that $L_{V_A}^{\leq h}$ is a subfunctor of $F_{V_A}$. The desired result is Zariski-local on $\Spec A$ by Remark \ref{rem:local_proj}, allowing us to assume that $V_A$ is free. Then $M_A$ is free over $\cO_{\cE,A}$ as well by Proposition \ref{prop:main_etale}(4b). Choosing a basis for $M_A$ (which induces a choice of isomorphism of $F_{V_A}$ with the affine Grassmannian), one can follow the proof of \cite[Prop.\ 2.1.7]{mffgsm}, checking that the $\phz$-stable condition is a Zariski-closed and finite type condition in the affine Grassmannian. Then $\cL_{V_A}^{\leq h}$ is projective with a canonical ample line bundle arising as the restriction of the canonical ample line bundle on the affine Grassmannian (Theorem \ref{thm:AG-IP} and Remark \ref{rem:basis_choice}). The compatibility with base change $A \ra A'$ follows from Proposition \ref{prop:main_etale}(2-3). 
\end{proof}

Write $\Theta_A$ for the projective map $\Theta_A: \cL^{\leq h}_{V_A} \ra \Spec A$.  Write $\underline{\frM}$ for the universal sheaf of $\Theta_A^*(\frS_A)$-modules on $\cL^{\leq h}_{V_A}$ and $\underline{\hat \frM}$ for its $u$-adic completion.

Next we will compare $\underline{\frM}$ with Galois representations, showing that the global sections of the universal $\frS_A$-lattice in $M_A$, with its Frobenius semi-linear structure, can recover $V_A$ in a similar fashion to the correspondence between $V_A$ and $M_A = M(V_A)$ in Proposition \ref{prop:main_etale}, but without losing $u$-integrality.
\begin{lem}[{Generalizing \cite[Lem.\ 1.4.1]{pssdr}}]
\label{lem:S-periods}
Set $\tilde A := \Theta_{A*}(\cO_{\cL_{V_A}^{\leq h}})$.  There is a canonical $\tilde A$-linear $G_{K_\infty}$-equivariant isomorphism
\begin{equation}
\label{eq:universal_V_A}
V_{\tilde A} := V_A \otimes_A \tilde A \lrisom \Hom_{\frS_{\tilde A}, \varphi}(\Theta_{A*}(\underline{\frM}), \frS^\ur_{\tilde A}).
\end{equation}
\end{lem}
This was proved in \cite[Lem.\ 1.4.1]{pssdr} in the case of Artinian $A$, and the proof of this general case, like that of Proposition \ref{prop:main_etale}, will depend on showing that the proof of the Artinian case can be made compatible with limits. 

We will require the result \cite[B.1.8.4]{fontaine1990} of Fontaine : if $\frN$ is a finite $\frS$-module with a Frobenius semi-linear endomorphism of bounded $E$-height, then the natural $\bZ_p[G_{K_\infty}]$-linear map
\begin{equation}
\label{eq:fontaine_input}
\Hom_{\frS, \varphi}(\frN, \frS^\ur) \lra \Hom_{\frS, \varphi}(\frN, \cO_{\cE^\ur}).
\end{equation}
induced by the inclusion $\frS^\ur \rinj \cO_{\cE^\ur}$ is an isomorphism.  When, in addition, $\frN$ has $R$-linear structure, then taking $R$-linear maps induces a canonical $R[G_{K_\infty}]$-linear isomorphism
\begin{equation}
\label{eq:fontaine_input_R}
\Hom_{\frS_R, \varphi}(\frN, \frS_R^\ur) \lrisom \Hom_{\frS_R, \varphi}(\frN, \cO_{\cE^\ur, R}).
\end{equation}

\begin{proof}
Let $M_A^*$ denote the $\cO_{\cE,A}$-dual of $M_A := M(V_A^*)$, equipped with the induced structure of an object of $\Phi^{\Gal}_M(A)$.  Using the canonical isomorphism 
\[
\Hom_{\cO_{\cE,\tilde A}}(M_{\tilde A}^*, \cO_{\cE^\ur,\tilde A}) \cong M_{\tilde A} \otimes_{\cO_{\cE,\tilde A}} \cO_{\cE^\ur,\tilde A}
\]
and applying $(-)^{\varphi = 1}$ to the canonical isomorphism \eqref{eq:compare_V_M}, we have a canonical isomorphism
\begin{equation}
\label{eq:pssdr_1.4.2}
V_{\tilde A} \lrisom (M_{\tilde A}^* \otimes_{\cO_{\cE,\tilde A}} \cO_{\cE^\ur, \tilde A})^{\varphi = 1}       \lrisom \Hom_{\cO_{\cE,\tilde A}, \varphi}(M_{\tilde A}, \cO_{\cE^\ur,\tilde A}).
\end{equation}
It remains to show that the rightmost factor of \eqref{eq:pssdr_1.4.2} and the rightmost factor of \eqref{eq:universal_V_A} are canonically $G_{K_\infty}$-equivariantly isomorphic.

The arguments given in \cite[Lem.\ 1.4.1]{pssdr} apply verbatim in the present case to show that $\Theta_{A*}(\uM)$ is a finite $\Theta_{A*}\Theta^*_A(\frS_A) = \frS \otimes_A \tilde A = \frS_{\tilde A}$-module, with the additional structure of a $\frS_{\tilde A}$-lattice in $\Theta_{A*}(\uM) \otimes_\frS \cO_\cE$ of $E$-height $\leq h$, where $\Theta_{A*}(\uM) \otimes_\frS \cO_\cE$ is naturally isomorphic to $M_A \otimes_A \tilde A$ in $\Phi^{\Gal}_M(\tilde A)$. It remains to show that $\Hom_{\frS_{\tilde A}, \varphi}(\Theta_{A*}(\underline{\frM}), \frS^\ur_{\tilde A}) \cong
\Hom_{\cO_{\cE^\ur, \tilde A}, \varphi}(\Theta_{A*}(\uM) \otimes_\frS \cO_\cE, \cO_{\cE^\ur, \tilde A})$. 

Choose now some $V_i \subset V_{\tilde A}$, an $R[G_{K_\infty}]$-submodule, finite as an $R$-module (i.e.\ an object of $\Mod_{G_{K_\infty}}(R)$), such that the natural map $V_i \otimes_R \tilde A \ra V_{\tilde A}$ is a surjection. Let $M_i = M(V_i) \subset M(V_{\tilde A}) = M_{\tilde A}$ be the corresponding $\cO_{\cE,R}$-submodule; Proposition \ref{prop:main_etale} gives $M_i$ the structure of an object of $\Phi^{' \Gal}_M(R)$ such that the canonical $\Phi^{' \Gal}_M(A)$-morphism $M_i \otimes_R \tilde A \ra M_{\tilde A}$ is surjective.  Let $\frN$ be the intersection
\[
\frN := \Theta_{A*}(\uM) \cap M_i \subset M_{\tilde A},
\]
which we observe is a $\frS_R$-submodule of $M_{\tilde A}$.  We have the natural surjection $\frN \otimes_R \tilde A \rsurj \Theta_{A*}(\uM)$.

Upon applying $\otimes_R \tilde A$, the isomorphism \eqref{eq:fontaine_input_R} induces an $\tilde A$-linear isomorphism
\[
\Hom_{\frS_R, \varphi}(\frN, \frS_{\tilde A}^\ur) \lrisom \Hom_{\frS_R, \varphi}(\frN, \cO_{\cE^\ur, {\tilde A}}).
\]
Then tensor-Hom adjunction results in an isomorphism 
\[
\Hom_{\frS_{\tilde A}, \varphi}(\frN \otimes_R \tilde A, \frS_{\tilde A}^\ur) \lrisom \Hom_{\frS_{\tilde A}, \varphi}(\frN \otimes_R \tilde A, \cO_{\cE^\ur, \tilde A}).
\]
Finally, because the map $\frS_{\tilde A}^\ur \ra \cO_{\cE,{\tilde A}}$ inducing this isomorphism may be checked to be an injection, an element of the left hand side factors through the quotient $\Theta_{A*}(\uM)$ if and only if its image on the right hand side factors through $\Theta_{A*}(\uM)$. That is, we may replace $\frN \otimes_R \tilde A$ by $\Theta_{A*}(\uM)$ in the equation. The adjunction $\Hom_{\frS_{\tilde A}, \varphi}(\Theta_{A*}(\uM), \cO_{\cE^\ur, \tilde A}) \cong \Hom_{\cO_{\cE^\ur, \tilde A}, \varphi}(\Theta_{A*}(\uM) \otimes_\frS \cO_\cE, \cO_{\cE^\ur, \tilde A})$ completes the proof. 
\end{proof}

\subsection{A Universal Family of Kisin Modules in Characteristic 0}
\label{subsec:universality_uM}

While the previous parts of \S\ref{sec:etale_kisin} have been carried out over a fixed discrete coefficient ring $A$, we now fix a residual pseudorepresentation $\Db$ of $G_K$ and let $A$ be a formally finitely generated $R_\Db$-algebra with a $G_K$-representation $V_A$ with induced pseudorepresentation compatible with the $R_\Db$-algebra structure of $A$. 

The results above can be applied to $(V_A \otimes_{R_\Db} R_\Db/\frmm_\Db^i)\vert_{G_{K_\infty}}$ for each $i \geq 1$ and extend to the limit, where $R$ above may be set to be the image of $R_\Db$ in $A/\frmm_\Db^i A$. For example, the functor $M$ generalizes to this setting naturally from the above, since the map of limits
\begin{equation}
\label{eq:et_comp}
M_A = (\cO_{\cE^{ur}} \hat\otimes_{\bZ_p} V_A^*)^{G_{K_\infty}} \lrisom \varprojlim (\cO_{\cE^{ur}} \otimes_{\bZ_p} V_A^* \otimes_A A/(\mDb A)^i)^{G_{K_\infty}}
\end{equation}
is an isomorphism by Fact \ref{fact:invariant_limit} and the fact that the ideal $(p \otimes 1) + \cO_{\cE^\ur} \otimes \mDb A$ (for which the left side is the completion) is equal to $\cO_{\cE^{ur}} \otimes_{\bZ_p} \mDb A$ (for which the right side is the completion).  This means that $M_A$ is a projective $\cO_{\cE, A}$-module of rank $d$ as expected.

For $B$ an $A$-algebra such that $\frmm_\Db^i \cdot B = 0$ for some $i \geq 1$, set $L^{\leq h}_{V_A}(B) = L^{\leq h}_{V_{A/(\mDb A)^i}}(B)$.
\begin{cor}
\label{cor:L-functor}
The functor $L_{V_A}^{\leq h}$ on $A$-algebras $B$ such that $\frmm_\Db^i \cdot B = 0$ for some $i \geq 1$ is represented by a projective $A$-scheme $\cL^{\leq h}_{V_A}$.
\end{cor}
\begin{proof}
By Proposition \ref{prop:kisin_proj} and Remark \ref{rem:local_proj}, this functor is represented by a projective formal scheme with a ample line bundle compatible with its limit structure.  By applying formal GAGA \ref{lem:AG_background}(1a), we conclude that $L^{\leq h}_{V_A}$ is the $\mDb$-adic completion of a projective $A$-scheme.
\end{proof}

We now study the the map $\Theta_A: \cL^{\leq h}_A \ra \Spec A$, showing that it is a closed immersion in equi-characteristic zero. 

\begin{prop}[{Generalizing \cite[Prop.\ 1.6.4]{pssdr}}]
\label{prop:pssdr_1.6.4}
Let $A$ and $V_A$ be as specified above.  Then
\begin{enumerate}
\item The map $\Theta_A: \cL_{V_A}^{\leq h} \ra \Spec A$ is a closed immersion after inverting $p$.
\item If $A^{\leq h}$ is the quotient of $A$ corresponding to the scheme-theoretic image of $\Theta_A$, then for any finite $W(\bF)[1/p]$-algebra $B$, a continuous $A \ra B$ factors through $A^{\leq h}$ if and only if $V_B = V_A \otimes_A B$ is of $E$-height $\leq h$.
\end{enumerate}
\end{prop}
Part (1) expresses the uniqueness of $\frS$-lattices of $E$-height $\leq h$ in characteristic zero \cite[Prop.\ 2.1.12]{crfc}. According to part (2), scheme-theoretic image of $\Theta_A$ has the property we expect. 

Proposition \ref{prop:pssdr_1.6.4} follows from the arguments in the case of local $A$, done in \cite[Prop.\ 1.6.4]{pssdr}. These arguments apply immediately to the case of non-local $A$, once Lemma \ref{lem:jacobson} is established. The reason that the arguments apply verbatim is that they need only address the $B$-points of $\Theta_A$, where $B$ is an Artinian $\bQ_p$-algebra. Indeed, Lemma \ref{lem:jacobson} verifies that closed loci in $\Spec A[1/p]$ are characterized entirely by their behavior on $B$-points. 

We establish the following notation: let $B$ be an Artinian local $\bQ_p$-algebra with residue field $E$. Let $\Int_B$ denote the set of finitely generated $\cO_E$-subalgebras of $B^0$, the preimage of $\cO_E$ via $B \rsurj E$. 
\begin{lem}
\label{lem:jacobson}
Let $A$ be a formally finitely generated $\cR$-algebra, where $\cR$ is a local admissible $\bZ_p$-algebra with maximal ideal $\frmm_\cR$. Then 
\begin{enumerate}
\item $A[1/p]$ is Jacobson and Noetherian,
\item all residue fields of $A[1/p]$ at maximal ideals are finite extensions of $\bQ_p$, 
\item a quotient $A'$ of $A[1/p]$ is characterized by those homomorphisms from $A[1/p]$ to finite $\bQ_p$-algebras which factor through $A'$, 
\item a finite $A[1/p]$-module $V$ is projective if and only if $V \otimes_A B$ is projective for every finite $\bQ_p$-algebra $B$ receiving a homomorphism $A \ra B$, 
\item the image of $A$ in a residue field from part (2) is an order in its ring of integers, and 
\item the image of $A$ in a finite local $\bQ_p$-algebra $B$ with residue field $E$ lies in some $C \in \Int_B$.
\end{enumerate}
Statements (1), (2), (3), and (4) are true for $A_\alg$, a finitely generated $\frmm_\cR$-adically separated $\cR$-algebra with completion $A$, in the place of $A$. 
\end{lem}
\begin{proof}
Because $A$ is Noetherian, $A[1/p]$ is also Noetherian. Every maximal ideal of $A$ has a finite characteristic $p$ residue field, so $p$ is in the Jacobson radical of $A$. Then \cite[Cor.\ 10.5.8]{ega4-3} implies that $A[1/p]$ is Jacobson, proving (1) for $A$.

Statement (2) follows directly from \cite[Lem.\ 7.1.9]{dejong1995}, which gives a natural bijection between the maximal ideals $\frmm$ of $A[1/p]$ and the points of the associated rigid analytic space and, furthermore, implies that the respective residue fields are isomorphic. Because all maximal ideals of $A$ have characteristic $p$, the image of $A$ in the $p$-adic field $E = A[1/p]/\frmm$ is not a field, but generates $E$ upon adjoining $1/p$. Statement (5) follows. Finally, it is clear that (5) implies (6).

Consider two quotient rings $A'$ and $A''$ of $A[1/p]$. Assume that the functors of points $\Hom_{\bQ_p}(A', -)$ and $\Hom_{\bQ_p}(A'', -)$ are identical after restriction to finite $\bQ_p$-algebras $B$. Applying (1) and (2), this assumption implies that for any closed point of $\Spec A'$ corresponding to maximal ideals $\frmm \subset A[1/p]$ and $\frmm' \subset A'$, there is a corresponding maximal ideal $\frmm'' \subset A''$ and there is an isomorphism of complete local rings $\hat A'_{\frmm'} \simeq \hat A''_{\frmm''}$ as quotient rings of $\widehat{A[1/p]}_{\frmm}$. Therefore, for any maximal ideal $\frmm$ of $A[1/p]$, we have $A' \otimes_{A[1/p]} \widehat{A[1/p]}_{\frmm} \simeq A'' \otimes_{A[1/p]} \widehat{A[1/p]}_{\frmm}$. Because $A[1/p] \rinj \prod_{\frmm \subset A[1/p]} \widehat{A[1/p]}_\frmm$ is faithfully flat, this implies that $A' = A''$ by descent, as desired, proving (3). Because faithfully flat morphisms are descent morphisms for the flat property, and flatness is equivalent to projectivity for finite modules over Noetherian rings, we also have (4). 

To prove (1) and (2) for $A_\alg$, consider that (1) and (2) apply to $\cR[1/p]$, and that $A_\alg[1/p]$ is finitely generated over $\cR[1/p]$. The Hilbert basis theorem and the nullstellensatz for Jacobson rings provide statements (1) and (2) for $A_\alg[1/p]$. Properties (3) and (4) follow. 
\end{proof}

We now show that there exists a family of $\frS$-lattices of $E$-height $\leq h$ with coefficients in $A^{\leq h}$ which are  universal in characteristic $0$ in the sense of part (4) below. Only the construction of \cite{pssdr} needs to be modified. 
\begin{prop}[{Generalizing \cite[Cor.\ 1.7]{pssdr}}]
\label{prop:dim1}
There exists a finite $\frS_{A^{\leq h}}$-module $\frM_{A^{\leq h}}$ with the following structures and properties. 
\begin{enumerate}
\item $\frM_{A^{\leq h}}$ is equipped with a linear map $\varphi^*(\frM_{A^{\leq h}}) \ra \frM_{A^{\leq h}}$ whose cokernel is killed by $E(u)^h$.
\item $\frM_{A^{\leq h}} \otimes_{\bZ_p} \bQ_p$ is a projective $\frS_{A^{\leq h}}[1/p]$-module.
\item For any finite $W(\bF_\Db)[1/p]$-algebra $B$, any map $f : A^{\leq h} \ra B$ and any $C \in \mathrm{Int}_B$ through which $f$ factors, there is a canonical, $\varphi$-compatible isomorphism of $\frS \otimes_{\bZ_p} B$-modules
\[
\frM_{A^{\leq h}} \otimes_{A^{\leq h}} B \lrisom \frM_C \otimes_C B,
\]
where $\frM_C$ is the unique $\frS$-lattice of $E$-height $\leq h$ in $M_C = M_A \otimes_A C$.
\item There is a canonical isomorphism
\[
V_{A^{\leq h}} \otimes_{\bZ_p} \bQ_p \lrisom \Hom_{\frS_{A^{\leq h}}[1/p], \varphi}(\frM_{A^{\leq h}} \otimes_{\bZ_p} \bQ_p, \frS^\ur_{A^{\leq h}}[1/p]).\]
\end{enumerate}
\end{prop}
\begin{proof}
Let $\hat \cL^{\leq h}_{V_A}$ be the $\mDb$-adic completion of $\cL^{\leq h}_{V_A}$.  Then
\[
\hat\Theta_{\frS_A}: \hat \cL^{\leq h}_{V_A} \times_{\Spf A} \Spf \frS_A \ra \Spf \frS_A
\]
is a projective morphism of $\Spf(A)$-formal schemes, and its base is Noetherian by Lemma \ref{lem:frS_A_noetherian}.  By Corollary \ref{cor:L-functor}, there exists a universal Kisin module $\hat\uM$ on the $\mDb$-adic completion $\hat\cL^{\leq h}_{V_A}$ of $\cL^{\leq h}_{V_A}$. Relative to the formal scheme $\hat \cL^{\leq h}_{V_A} \times_{\Spf A} \Spf \frS_A$, it is a locally free coherent sheaf.  Applying formal GAGA for $\Theta_{\frS_A}$, $\hat\uM$ is the completion of a finite locally free module $\uM$ on the projective $\frS_A$-scheme
\[
\Theta_{\frS_A}: \Spec \cL^{\leq h}_{V_A} \times_{\Spec A} \Spec \frS_A \ra \Spec \frS_A.
\]
The scheme theoretic image of $\Theta_{\frS_A}$ is $\frS_{A^{\leq h}}$.  We set
\[
\frM_{A^{\leq h}} := \Theta_{\frS_{A}*}(\uM).
\]

With this work done, the proofs of part (1), (2), (3) and (4) may be repeated from \cite[Cor.\ 1.7]{pssdr}.  For part (4), we remark that just as in \emph{loc.\ cit.},  there is a canonical isomorphism
\[
V_{\tilde A} \lrisom \Hom_{\frS_{\tilde A}, \varphi}(\hat\Theta_{\frS_A *}(\hat\uM), \frS^{ur}_{\tilde A})
\]
produced by combining Lemma \ref{lem:S-periods}, which gives this isomorphism when $\tilde A$ is replaced by $\tilde A/\frmm_\Db^i A$ for each $i \geq 1$, and the theorem on formal functions \cite[Thm.\ 4.1.5]{ega3-1}.  Then the right side is de-completed by formal GAGA, and (4) then follows from the fact that the kernel and cokernel of $A^{\leq h} \ra \tilde A$ are $p$-torsion. 
\end{proof}

\section{Period Maps and $(\phz,N)$-modules in Families}
\label{sec:periods}

So far, we have cut out loci of $G_{K_\infty}$-representations with $E$-height $\leq h$. In this section and in \S\ref{sec:PHT_conditions}, we refine this locus to cut out semi-stable $G_K$-representations with Hodge-Tate weights in $[0,h]$. In order to do this, we will construct a family of $(\varphi,N)$-modules from the family of Kisin modules already produced. The locus will be cut out in characteristic 0, but the construction relies on the family having an integral model. 

In \S\ref{subsec:PHT_background_2} we will set up notation and prove lemmas so that we can use the required period rings in families. In \S\ref{subsec:PHT_families} we will carry out the constructions needed to cut out the semi-stable locus; the result will appear in \S\ref{sec:PHT_conditions}. 

\subsection{Background and Notation}
\label{subsec:PHT_background_2}

We now change notation, writing $A^\circ$ for what we called $A$ above, and writing $A$ for what we called $A[1/p]$ above. Also, assume that $A^\circ$ is $p$-torsion free. We also make analogous notation changes to $R_\Db$ as follows. 

\begin{center}
\begin{tabular}{| l | c | c | c | c | }
\hline
Formerly: & $A$ & $R_\Db/(p\text{-tors})$ & $A[1/p]$ &  $R_\Db[1/p]$   \\
\hline
Henceforth: & $A^\circ$ & $R^\circ$ & $A$ & $R$ \\
\hline
\end{tabular} 
\end{center}

For $S$ a $\bZ_p$-algebra we write $S_A := S_{A^\circ}[1/p]$, where we recall that $S_{A^\circ}$ is the $\mDb A$-adic completion of $S \otimes_{\bZ_p} A^\circ$.   We will use the canonical isomorphism $\frS_{A}/u\frS_{A} \risom W_{A} \cong W[1/p] \otimes_{\bQ_p} A$.

Let $\cO := \varprojlim_n (W\lb u, u^n/p\rb[1/p])$, which is the ring of rigid analytic functions on the open disk of radius $1$ (cf.\ \cite[\S1.1.1]{crfc}), including $\frS[1/p]$ as the dense subring of bounded functions. The Frobenius endomorphism $\varphi$ has a unique continuous extension from $\frS[1/p]$ to each ring $W\lb u, u^n/p\rb[1/p]$, and therefore to $\cO$ as well. 

In order to study families over $A$ of $\varphi$-modules over $\cO$, we need to define the correct notion of the ring of coefficients.  Two candidate definitions end up being the identical: 
\begin{equation}
\label{eq:analytic_A_disk_same}
\cO_A := \varprojlim_n(W\lb u, u^n/p\rb_A) \lrisom \varprojlim_n(W_{A^\circ}\lb u, u^n/p\rb[1/p]).
\end{equation}
While it is clear that these rings are isomorphic when $A^\circ$ is local, we prove the isomorphism here in the general case. 
\begin{lem}
\label{lem:cO_equiv_defns}
The natural inclusions
\[
W\lb u, u^n/p\rb_A \rinj W_{A^\circ}\lb u, u^n/p\rb[1/p], \quad n \geq 1,
\]
induce the isomorphism \eqref{eq:analytic_A_disk_same}.
\end{lem}

\begin{proof}
Write $B_n := W\lb u, u^n/p\rb_A$ and $C_n := W_{A^\circ}\lb u, u^n/p\rb[1/p]$, with the canonical map $B_n\rinj C_n$ that we get from considering an element of $B_n$ as a power series in $u$.  Since the maps making up these limits are injective, it will suffice to show for $f \in C_{2n}$ that its image in $C_n$ under the inclusions making up the limit lies in the image of $B_n$ in $C_n$.  With $f \in C_{2n}$ chosen, write it as
\[
f = \sum_{m \geq 0} f_m \frac{u^m}{p^{\floor{m/2n}}}
\]
where, for some fixed $N \geq 0$, $f_m \in p^{-N}(W \otimes_{\bZ_p} A^\circ) \subseteq W \otimes_{\bZ_p} A$ for all $m \geq 0$.  We rewrite it as
\[
f = \sum_{m \geq 0} f_m p^{\floor{m/n}-\floor{m/2n}} \frac{u^m}{p^{\floor{m/n}}}.\]
Because $p \in \frmm_\Db A^\circ$, the coefficient $f_m p^{\floor{m/n}-\floor{m/2n}}$ of $u^n/p^{\floor{m/n}}$ lies in $(\mDb A^\circ)^{i(n)}$ where $\lim_{n \ra +\infty} i(n) = + \infty$. This means that $f$ lies in the image of $B_n$ in $C_n$, which is what we wanted to prove. 
\end{proof}

 We observe that $\cO_A$ has an $A$-linear Frobenius endomorphism compatible with the natural map $\frS_A \ra \cO_A$. 

There are natural maps from $\frS_A$ and $\cO_A$ to other period rings in families. First we recall the period rings and some properties; see e.g.\ \cite{BCon2009} for a reference. Let $\Ac$ be the $p$-adic completion of the divided power envelope of $W(R)$ (see \S\ref{subsec:PHT_background_1}) with respect to $\ker(\theta)$, and let $\Bpc := \Ac[1/p]$. The action of $\varphi$ on $W(R)$ extends to an action on $\Ac$.  Also recall from \S\ref{subsec:PHT_background_1} that we have fixed a map $\frS \rinj W(R)$ sending $u$ to $\tupi$. The resulting map $\frS[1/p] \rinj \Bpc$ extends uniquely to a continuous inclusion $\cO \rinj \Bpc$, because $\frS[1/p]$ is dense in $\cO$, and the $e$th power of the image $\tupi$ of $u$ in $W(R)$ is in the divided power ideal $(\ker \theta, p)$ for $\Ac$. Define $\BpdR$ to be the $\ker(\theta)$-adic completion of $W(R)[1/p]$, where $\theta$ is extended to a map $\theta: W(R)[1/p] \rsurj \bC_p$, and let $\BdR$ be its fraction field. 

Recall from \S\ref{subsec:PHT_background_1} the definition of $\tupi$, the image of $u$ in $W(R)$, and $\tuep$ of \eqref{eq:tuep_defn}. Write $\ell_u, t \in \BpdR$ for the elements defined by
\[
\ell_u = \log\tupi := \sum_{i=1}^\infty \frac{(-1)^{i-1}}{i} \left(\frac{\tupi-\pi}{\pi}\right)^i, \qquad
t = \log [\underline{\varep}] := \sum_{i=1}^\infty (-1)^{i+1} \frac{([\underline{\varep}]-1)^i}{i};
\]
one can check that these series converge in $\BpdR$. 

We may now define several more period rings: $\Bc := \Bpc[1/t] \subset \BdR$, $\Bpst := \Bpc[\ell_u] \subset \BpdR$, and $\Bst := \Bc[\ell_u]$.  We can and will think of $\Bpst$ as a polynomial ring over $\Bpc$, for $\ell_u$ is transcendental over the fraction field of $\Bc$.  One computes that Frobenius $\phz$ acts as $\varphi(\ell_u) = p\ell_u$ and $\varphi(t) = pt$.

Equip $\Bpst$ with an endomorphism $N$ by formal differentiation $d/d\ell_u$ of the variable $\ell_u$ with coefficients in $\Bpc$, i.e.\ so that $N(\Bpc) = 0$.  Extend $\varphi$ to $\Bpst$ as well, with $\varphi(\ell_u) = p\ell_u$.  We note that $\varphi$ and $N$ define endomorphisms of the polynomial subring $K_0[\ell_u] \subset \Bpst$, and that $p\varphi N = N \varphi$ on $\Bst$. 

There is an exhaustive, decreasing filtration on each of $\Ac, \Bpc$, written 
\[
\Fil^i \Ac, \Fil^i \Bpc,
\]
induced by their inclusion in the filtered ring $\BpdR$, such that $\Fil^0 \Ac = \Ac$ and $\Fil^0 \Bpc = \Bpc$. The filtration on $\BpdR$ is given by
\[
\Fil^i \BpdR := (\ker\theta)^i, \quad i \geq 0.
\]
In fact, $t \in \Fil^1 \BpdR$ and $t \not\in \Fil^2 \BpdR$, so also $t \in \Fil^1 \Ac$, and $t$ is a generator for the maximal ideal of $\BpdR$.  

There is an action of $G_K$ on these rings arising from its action on $\cO_{\bar K}/p$ to a continuous action on $R$, $W(R)$, and the derivative rings above. In particular, it will be useful to know the action of $G_K$ on $t$ is given by $\sigma(t) = \chi(\sigma)\cdot t$ where $\chi$ represents the $p$-adic cyclotomic character. It is also well-known that $\Bpst$ is stable under $G_K$; this also follows from the following calculation (see \cite[Lem.\ 4.6.4]{WEthesis}), which we will need later. 
\begin{lem}
\label{lem:beta}
The map $\beta$ given by 
\[
\beta(\sigma) := \sigma(\ell_u) - \ell_u \text{ for } \sigma \in G_K
\]
is a $1$-cocycle $\beta: G_K \ra \bZ_p(1)$ with respect to the cyclotomic character, belonging to the cohomology class associated to $\pi$ by Kummer theory.  When $\beta(\sigma) \neq 0$, it generates the maximal ideal of $\BpdR$.
\end{lem}
As this maximal ideal generates the filtration on $\BpdR$, if $\beta(\sigma) \neq 0$ then 
\begin{equation}
\label{eq:beta_location}
\beta(\sigma) \in \Fil^1 \Bpc, \quad \beta(\sigma) \not\in \Fil^2 \Bpc.
\end{equation}

We will use the following families of period rings over $A$. Define $\Bpca := \Acac[1/p]$, where $\Acac$ is, as usual, the $\mDb A^\circ$-adic completion of $\Ac \otimes_{\bZ_p} A^\circ$.  For any $A$-algebra $B$, we write $\Bpcb$ for $\Bpca \otimes_A B$.  Set $\Bpsta := \Bpca[\ell_u]$ and $\Bpstb := \Bpsta \otimes_A B$.  The map $\varphi$ extends to each of these rings $B$-linearly, with $N$ again acting as formal differentiation with respect to $\ell_u$ here.  In particular, $N(\Bpcb) = 0$. Analogous notation is used for the elements of the filtration on these rings: denote by $\Fil^i \Acac$ the $\mDb A^\circ$-adic completion of $\Fil^i \Ac \otimes_{\bZ_p} A^\circ$, and for any $A$-algebra $B$ let $\Fil^i \Bpcb := \Fil^i \Acac \otimes_{A^\circ} B$.  

It will be important to know in the construction of \eqref{eq:period_main} that there is a canonical inclusion $\cO_A \rinj \Bpca$ extending the map $\cO \rinj \Bpc$ discussed above, and also a map $\frS^\ur_A \rinj \Bpca$.  By Lemma \ref{lem:cO_equiv_defns}, it will suffice to show that for large enough $n$,
\[
W\lb u, u^n/p\rb_{A^\circ} \rinj \Acac.
\]
In order to construct the map, it will suffice to draw, for sufficiently large $n$, maps
\[
W\lb u, u^n/p\rb \otimes_{\bZ_p} A^\circ/(\mDb A^\circ)^j \rinj \Ac \otimes_{\bZ_p} A^\circ/(\mDb A^\circ)^j
\]
for each $j \geq 1$.  We will get such maps by showing, for large enough $n$, the existence of maps
\[
W\lb u, u^n/p\rb \rinj \Ac.
\]
Indeed, these maps exist for $n \geq e$ because the $e$th power of $u$ maps to a divided power ideal for $A_\cris$ relative to $W(R)$, as mentioned above. With the construction complete, Lemma \ref{lem:pssdr_2.3.1} implies that this map will remain injective after tensoring with $A^\circ$ and completing with respect to the $\mDb A^\circ$-adic topology.  This same construction gives us a canonical map $\frS^\ur_A \rinj \Bpca$.

We will now record some lemmata to ensure that the large rings $\Bpc$, $\Ac$, and so forth behave well in families. 

\begin{lem}[{\cite[Lem.\ 2.3.1]{pssdr}}]
\label{lem:pssdr_2.3.1}
Let $\cR$ be a Noetherian ring that is $I$-adically complete and separated for some ideal $I \subset \cR$. For any $\cR$-module $M$, denote by $\widehat M$ its $I$-adic completion.  If $M$ is a flat $\cR$-module, then 
\begin{enumerate}
\item For any finite $\cR$-module $N$, the natural map
\[
N \otimes_{\cR} \widehat M \ra \widehat{N\otimes_{\cR} M}
\]
is an isomorphism.
\item $\widehat M$ is flat over $\cR$.  If $M$ is faithfully flat over $\cR$, then so is $\widehat M$.
\item The functor $M \mapsto \widehat M$ preserves short exact sequences of flat $\cR$-modules.
\end{enumerate}
\end{lem}
Lemma \ref{lem:pssdr_2.3.1} is given in \cite{pssdr} in the case that $\cR$ is local, but its proof is valid for any adically complete Noetherian ring. 

The following lemma requires some generalization from the local case. 
\begin{lem}[{Generalizing \cite[Lem.\ 2.3.2]{pssdr}}]
\label{lem:pssdr_2.3.2}
Let $\cR$ be a admissible $\bZ_p$-algebra, $I$-adically complete and separated. Also, assume that $\cR$ is $p$-torsion free. 
\begin{enumerate}
\item For $i \geq 0$, the ideal $\Fil^i \AcR$ of $\AcR$ is a faithfully flat $\cR$-module.
\item For $i \geq 0$, $\Fil^i \AcR/\Fil^{i+1} \AcR$ is a faithfully flat $\cR$-module, which is isomorphic to the $I$-adic completion of $(\Fil^i \Ac/\Fil^{i+1} \Ac) \otimes_{\bZ_p} \cR$.
\item For any $\cR[1/p]$-algebra $B$, $i \geq 1$, and $\sigma \in G_K$, $\Bpcb/(\beta(\sigma)B + \Fil^i \Bpcb)$ is a flat $B$-module.  If $\beta(\sigma) \not\in \Fil^i \Bpc$, then $\beta(\sigma) \not\in \Fil^i \Bpcb$.
\item Let $B^\circ$ be a finite continuous $\cR$-algebra with ideal of definition $J$. Then the $J$-adic completion of $\Ac \otimes_{\bZ_p} B^\circ$ is canonically isomorphic to $\AcR \otimes_{\cR} B^\circ$.
\item The map 
\[\AcR \ra \prod_\frq A_{\mathrm{cris}, \cR/\frq}\]
is injective, where $\frq$ runs over ideals of $\cR$ such that $\cR/\frq$ is a finite flat $\bZ_p$-algebra.
\item If $0 \neq f \in \Ac$, then $f$ is not a zero divisor in $\AcR$.
\end{enumerate}
\end{lem}

\begin{proof}
Parts (1), (2), (3), (4), and (6) are proved by the same arguments as the corresponding parts of \cite[Lem.\ 2.3.2]{pssdr}, where for part (4) the  the fact that $f: \cR \ra B^\circ$ is finite and continuous implies that $f$ is adic, i.e.\ that $f(I)\cdot B^\circ$ is an ideal of definition for the $J$-adic topology of $B^\circ$. 

To prove part (5), consider that if $0 \neq f \in A_{\cris, \cR}$, we may fix some $n$ such that $0 \neq f \in A_{\cris, \cR/I^n} = \Ac \otimes_{\bZ_p} \cR/I^n$. There is an injective map
\[
\cR/I^n \rinj \prod_\frmm (\cR/I^n)^\wedge_\frmm 
\]
with $\frmm$ varying over the maximal ideals of $\cR/I^n$, which are in natural bijective correspondence with the maximal ideals of $\cR$. Because $\Ac$ is $\bZ_p$-flat, there exists some maximal ideal $\frmm' \subset \cR$ such that the projection of $f$ to $A_{\cris,(\cR/I^n)^\wedge_{\frmm'}}$ is non-zero. Then there exists a positive integer $a$ such that the projection of $f$ to $A_{\cris, \cR/(I^n+\frmm'^a)}$ is non-zero. Notice that $f$ naturally projects from $A_{\cris, \cR}$ to $A_{\cris, (\cR/I^n)^\wedge_{\frmm'}}$, which then projects to $A_{\cris, \cR/(I^n+\frmm'^a)}$. Therefore the images of $f$ in $A_{\cris, (\cR/I^n)^\wedge_{\frmm'}}$ and $A_{\cris, \cR^\wedge_{\frmm'}}$ are non-zero. Because the statement of (5) is known in the case that $\cR$ is local \cite[Lem.\ 2.3.2(5)]{pssdr}, we apply this case to $\cR^\wedge_{\frmm'}$ to produce an ideal $\frq \subset \cR^\wedge_{\frmm'}$ with the desired property. Then the kernel of the surjective composite $\cR \ra \cR^\wedge_{\frmm'} \ra \cR^\wedge_{\frmm'}/\frq$ is an ideal of $\cR$ with the desired property. 
\end{proof}

The following lemma will be useful to construct loci cutting out conditions realized over a family of period rings. 

\begin{lem}[Generalizing {\cite[Lem.\ 2.3.3]{pssdr}}]
\label{lem:pssdr_2.3.3}
Let $M$ be an $A^\circ$-module and $x \in \Acac \otimes_{A^\circ} M$.  The set of $A^\circ$-submodules $N \subset M$ such that $x \in \Acac \otimes_{A^\circ} N$ has a smallest element $N(x)$.
\end{lem}
\begin{proof}
Assume that $\frmm_\Db^n\cdot M = 0$ for some $n \geq 1$, and choose some $x \in \Acac \otimes_{A^\circ} M$.  Therefore there is a natural isomorphism of $A^\circ$-modules
\[
\Ac_{,R_\Db} \otimes_{R_\Db} M \lrisom \Acac \otimes_{A^\circ} M.
\]
Because the statement was shown to be true when $A^\circ$ is local in \cite[Lem.\ 2.3.3]{pssdr}, we may apply the statement, with $R_\Db$ in place of $A^\circ$, to the left hand side. Here we are using the assumption that $n$ exists as above, since this implies that $\Ac_{,R_\Db} \otimes_{R_\Db} M \cong \Acac \otimes_{A^\circ} M$. As a result, there exists a smallest $R_\Db$-submodule $P$ of $M$ such that $x \in \Ac_{,R_\Db} \otimes_{R_\Db} P$.  We claim that the image $N(x)$ of the natural map
\[
P \otimes_{R_\Db} A^\circ \ra M
\]
is the smallest $A^\circ$-submodule of $M$ with the required property.  Clearly it contains $x$.  If there were a $A^\circ$ submodule $N$ with the property, then $N \supset P$ since $N$ is also a $R_\Db$-module with the property.  But then $N$ must contain $N(x)$, which is the $A^\circ$-span of $P$.  This shows that $N(x)$ is the smallest $A^\circ$-submodule of $M$ with the property.

The proof of \cite[Lem.\ 2.3.3]{pssdr} deduces the Lemma in the case that $A^\circ$ is local from the case that $M$ has finite length. The same proof works in this setting, so we briefly sketch it. Assume that $M$ is a finitely generated $A^\circ$-module, as the general case will follow. Denote by $x_n$ the image of $x$ in $\Acac \otimes_{A^\circ} M/\mDb^n M$, and denote by $N(x_n)$ the submodule of $M/\mDb^n M$ obtained from the argument above. One then observes that the image of $N(x_{n+1})$ in $M/\mDb^n M$ is $N(x_n)$, and that the limit $N(x) := \varprojlim_n N(x_n)$ is the desired submodule of $M$ (using Lem.\ \ref{lem:pssdr_2.3.1}(1)). 
\end{proof}

\subsection{Period Maps in Families}
\label{subsec:PHT_families}

Suppose now that $V_{A^\circ}$ has an $A^\circ$-linear action of $G_K$ such that its restriction to $G_{K_\infty}$ has $E$-height $\leq h$, in the sense that $(A^\circ)^{\leq h} = A^\circ$ (cf.\ Prop.\ \ref{prop:pssdr_1.6.4}(2)). Write $V_A := V_{A^\circ} \otimes_{A^\circ} A$. We will now follow \cite[\S2.4]{pssdr} in constructing a period map \eqref{eq:period_main} comparing the family $V_A$ of $G_K$-representations to a family of $(\phz, N)$-modules. Using the results of \S\ref{sec:etale_kisin}, we can compare $V_A\vert_{G_{K_\infty}}$ to a family of Kisin modules. We will produce a family of $(\phz, N)$-modules from Kisin modules, and add additional structure needed to descend $G_{K_\infty}$-representations to $G_K$-representations. Our goal is to produce $A^\st$, the maximal quotient of $A$ over which $V_A$ is semi-stable with Hodge-Tate weights in $[0,h]$. 
 
Proposition \ref{prop:dim1} has produced $\frM_{A^\circ}$, a finite $\frS_{A^\circ}$-module, such that  $\frM_A := \frM_{A^\circ} \otimes_{\bZ_p} \bQ_p$ is a projective, rank $d$ $\frS_A$-module with a $\varphi$-semilinear, $A$-linear endomorphism $\varphi: \frM_A \ra \frM_A$ such that the induced $\frS_A$-linear map $\varphi^*(\frM_A) \ra \frM_A$ has cokernel killed by $E(u)^h$. One part of the period map comes from Prop.\ \ref{prop:dim1}(4), which provides a canonical, $G_{K_\infty}$-equivariant isomorphism
\begin{equation}
\label{eq:period_iota}
\iota: V_A \lrisom \Hom_{\frS_A, \varphi}(\frM_A, \frS^{\mathrm{ur}}_A).
\end{equation}

The following lemma supplies the other half of the period map, comparing $V_A$ with the candidate family of $(\varphi,N)$-modules $D_A$. Indeed, we write 
\[
\cM_A := \frM_A \otimes_{\frS_A} \cO_A, \quad \text{ and } \quad D_A := \cM_A/u\cM_A,
\]
each of which have a natural induced action of $\varphi$. Denote by $\widehat \frS_{0, A}$ the completion of $K_0[u] \otimes_{\bQ_p} A$ at the ideal $(E(u))$. We also define $\lambda := \prod_{n=0}^\infty \varphi^n(E(u)/E(0)) \in \cO$. 

\begin{lem}[\cite{pssdr}, Lem.\ 2.2]
\label{lem:pssdr_2.2}
There is a unique, $\varphi$-compatible, $W_A$-linear map $\xi: D_A \ra \cM_A$, whose reduction modulo $u$ is the identity on $D_A$.

The induced map $D_A \otimes_{W_A} \cO_A \ra \cM_A$ has cokernel killed by $\lambda^h$, and the image of the map $D_A \otimes_{W_A} \widehat \frS_{0, A} \ra \cM_A \otimes_{\cO_A} \widehat \frS_{0, A}$ is equal to that of
\[
\varphi^*(\cM_A) \otimes_{\cO_A} \widehat \frS_{0, A} \ra \cM_A \otimes_{\cO_A} \widehat \frS_{0, A}.
\]
\end{lem}
\begin{proof}
The proofs of \cite[Lem.\ 2.2 and Lem.\ 2.2.1]{pssdr} go through verbatim.  None of it depends on $A^\circ$ being local.  They are a generalization of \cite[Lem.\ 1.2.6]{crfc}, where $A^\circ = \bZ_p$. Additional detail may be found in \cite[Lem.\ 4.7.1]{WEthesis}.
\end{proof}

We may now produce the period map from $\xi$ and $\iota$ as follows. The map $\iota$ induces a $\frS_A$-linear, $\varphi$-equivariant map
\[
\frM_A \lra \Hom_A(V_A, \frS^\ur_A); \quad m \mapsto (v \mapsto \langle m, \iota(v)\rangle).
\]
Applying $\otimes_{\frS_A} \cO_A$ to it and composing it with $\xi$, we have a $\varphi$-equivariant map
\[
D_A \buildrel\xi\over\lra \cM_A \lra \Hom_A(V_A, \frS^\ur_A) \otimes_{\frS_A} \cO_A \lra \Hom_A(V_A, \Bpca).
\]
Tensoring the composition of these maps by $\otimes_A B$, where $B$ is any $A$-algebra, there is a $\Bpcb$-linear map
\begin{equation}
\label{eq:period_main}
D_B \otimes_{W_B} \Bpcb \ra \Hom_A(V_A, \Bpca) \otimes_A B \cong \Hom_B(V_B, \Bpcb).
\end{equation}

We see that the right hand side has an action of $G_K$, and the left hand side has an action of $G_{K_\infty}$ through the action on $\Bpcb$.  This map is $G_{K_\infty}$-equivariant because $G_{K_\infty}$ acts equivariantly on the inclusions $\frS \rinj \cO \rinj \Bpc$ and because $\iota$ above is $G_{K_\infty}$-equivariant.  In order to extend the action of $G_{K_\infty}$ on the left hand side of \eqref{eq:period_main} to an action of $G_K$, we suppose that there is a $W_B$-linear map
\[
N: D_B \ra D_B
\]
which satisfies the identity $p\varphi N = N \varphi$.  Then the action of $G_K$ on $D_B \otimes_{W_B} \Bpcb$ is
\begin{equation}
\label{eq:hG_action}
\sigma(d \otimes b) = \left(\sum_{i=0}^\infty \frac{N^i(d) \otimes \beta(\sigma)^i}{i!} \right)\sigma(b) = \exp(N \otimes \beta(\sigma)) \cdot d \otimes \sigma(b)
\end{equation}
for $\sigma \in G_K$.  One can check that this action of $G_K$ commutes with $\varphi$.

In order to parameterize semi-stable representations, we must work with $\Bpst$. Recall that we adjoin $\ell_u$ to $\Bpcb$ to get $\Bpstb = \Bpcb \otimes_{K_0} K_0[\ell_u]$ with a $B$-linear action of $N$ and $\phz$. Consider the composite of the isomorphisms
\begin{equation}
\label{eq:pssdr_2.4.4}
D_B \otimes_{K_0} K_0[\ell_u] \risom (D_B \otimes_{K_0} K_0[\ell_u])^{N=0} \otimes_{K_0} K_0[\ell_u] \buildrel{(\ell_u \mapsto 0) \otimes 1}\over\lra D_B \otimes_{K_0} K_0[\ell_u]
\end{equation}
where the first map is the inverse to the natural isomorphism
\[
(D_B \otimes_{K_0} K_0[\ell_u])^{N=0} \otimes_{K_0} K_0[\ell_u] \lrisom D_B \otimes_{K_0} K_0[\ell_u]
\]
induced by polynomial multiplication in $K_0[\ell_u]$.  Tensoring \eqref{eq:pssdr_2.4.4} by $\Bpcb$ over $W_B$ and tensoring \eqref{eq:period_main} by $K_0[\ell_u]$ over $K_0$, we obtain the composite map
\begin{equation}
\label{eq:period_st}
D_B \otimes_{W_B} \Bpstb \buildrel{\eqref{eq:pssdr_2.4.4}}\over\lra D_B \otimes_{W_B} \Bpstb \buildrel{\eqref{eq:period_main}}\over \lra \Hom_B(V_B, B) \otimes_B \Bpstb.
\end{equation}

We claim that \eqref{eq:period_main} is $G_K$-equivariant if and only if $\eqref{eq:period_st}$ is equivariant when $G_K$ is regarded as acting trivially on $D_B$.  A key observation is that the an inverse to the bijection $(D \otimes_{K_0} K[\ell_u])^{N=0} \buildrel{\ell_u \mapsto 0}\over\lra D$ is given by $d \mapsto \exp(-N \otimes \ell_u)\cdot d$.

The following lemma is an important step toward the comparison of semi-stable Galois representations and filtered $(\varphi, N)$-modules in families. 
\begin{lem}[{Generalizing \cite[Lem.\ 2.4.6]{pssdr}}]
\label{lem:period_injective}
For each $A$-algebra $B$, the maps \eqref{eq:period_main} and \eqref{eq:period_st} are injective, and their cokernels are flat $B$-modules.
\end{lem}
\begin{proof}
First we note that it suffices to prove the assertions only for \eqref{eq:period_main}, and for $B = A$. Indeed, the case for general $B$ arises from the case $B = A$ by applying $\otimes_A B$, and this map will remain injective after $\otimes_A B$ if the cokernel for $B = A$ is a flat $A$-module. 

Lemma \ref{lem:pssdr_2.3.2}(5) immediately reduces the injectivity claim to the case that $A^\circ$ is finite over $\bZ_p$, which was proved in \cite[Lem.\ 2.4.6]{pssdr}.

To show that the cokernel of \eqref{eq:period_main} is flat, it suffices to show that \eqref{eq:period_main} remains injective after applying $\otimes_A A/I$ for any finitely generated ideal $I$ of $A$.  If we had started our proof with $A/I$ in the place of $A$, we would still have the injectivity statement for $A/I$, just as we proved it for $A$ above.  Therefore it only remains to show that 
\[
D_A \otimes_{W_A} \Bpca \otimes_A A/I \buildrel{\eqref{eq:period_main} \otimes_A A/I}\over\lra \Hom_A(V_A, \Bpca) \otimes_A A/I
\]
is injective provided that
\[
D_{A/I} \otimes_{W_{A/I}} B^+_{\cris, A/I} \buildrel{\eqref{eq:period_main}}\over\lra \Hom_{A/I}(V_{A/I}, B^+_{\cris, A/I})
\]
is injective.  This is precisely what Lemma \ref{lem:pssdr_2.3.2}(4) tells us -- indeed, the sources and targets are isomorphic, respectively -- completing the proof.
\end{proof}

Now we can produce $A^\st$, the maximal quotient of $A$ over which $V_A$ is semi-stable with Hodge-Tate weights in $[0,h]$. This means that for any $A$-algebra $B$ which is finite as a $\bQ_p$-algebra, the representation $V_A \otimes_A B$ is semi-stable with Hodge-Tate weights in $[0,h]$ if and only if $A \ra B$ factors through $A^\st$.  

\begin{prop}[{Generalizing \cite[Prop.\ 2.4.7]{pssdr}}]
\label{prop:pssdr_2.4.7}
The functor which to an $A$-algebra $B$ assigns the collection of $W_B$-linear maps $N: D_B \ra D_B$ which satisfy $p\varphi N = N \varphi$ and such that \eqref{eq:period_main} is compatible with the action of $G_K$ is representable by a quotient $A^\st$ of $A$.  
\end{prop}

\begin{proof}
We may freely assume that $V_A$ is a free rank $d$ $A$-module. In this case, the construction of $A^\st$ in \cite[Prop.\ 2.4.7]{pssdr} of a finitely generated $A$-algebra $A^\st$ representing the functor of the statement generalizes to this setting, since the ingredients, Lemmas \ref{lem:pssdr_2.3.3} and \ref{lem:pssdr_2.3.1} and the map \eqref{eq:period_main}, generalize. We sketch the argument to demonstrate these dependencies.

Firstly, the functor assigning to $B$ the set of $W_B$-linear maps $N: D_B \ra D_B$ satisfying $p\phz N = N \phz$ is representable by a finitely generated $A$-algebra $A^N$.

Write $\eta_B$ for the map of \eqref{eq:period_main}, and for $d \in D_{A^N}$ and $\sigma \in G_K$ set
\[
\delta_\sigma(d) = \eta_{A^N}(\sigma(d)) - \sigma(\eta_{A^N}(d)),
\]
which are elements of $Q := \Hom_{A^N}(V_{A^N},B^+_{\cris,A^N})$. We wish to show that the vanishing of this map for all $\sigma, d$ is cut out by an ideal of $A$. Choose a $B^+_{\cris,A^N}$-basis for $Q$ and let $x_1, \dotsc, x_r$ be the coordinates of $\delta_\sigma(d)$ with respect to the basis. Applying Lemma \ref{lem:pssdr_2.3.3} with $M = A^N$ and $x = x_i$ for $x_i$ varying over a $B^+_{\cris, A^N}$-basis for $Q$, the span of the resulting ideals of $A^N$ is the kernel of the quotient $A^\st$ of $A^N$. Then, because $Q$ is a faithfully flat $A^N$-module by Lemma \ref{lem:pssdr_2.3.1}(1), $A^N \ra B$ will factor through $A^\st$ if and only if $\eta_B$ is compatible with the action of $G_K$.

With this construction complete, the same proof of \cite[Prop.\ 2.4.7]{pssdr} shows that $A \ra A^\st$ is surjective, replacing appeals to \cite[Lem.\ 2.4.6 and Lem.\ 2.3.2(3)]{pssdr} with Lemma \ref{lem:period_injective} and Lemma \ref{lem:pssdr_2.3.2}(3), respectively. 
\end{proof}

\section{Algebraic Families of Potentially Semi-Stable Galois Representations}
\label{sec:PHT_conditions}

We will maintain the notation for $A, A^\circ$, etc.\ established in \S\ref{subsec:PHT_background_2}, but we drop the assumption of \S\ref{subsec:PHT_families} that the family $V_{A^\circ}$ of representations of $G_K$ has $E$-height $\leq h$. We have studied the period map
\[
D_{A} \otimes_{W_{A}} \Bpca \lra \Hom_{A}(V_{A}, \Bpca)
\]
of a family of $G_K$-represetations with bounded $E$-height and demonstrated that it is injective with flat cokernel, and also $G_K$-equivariant over a Zariski-closed locus. In this section, we will show that these properties allow for the construction of Zariski-closed loci of crystalline and semi-stable Galois representations with bounded Hodge-Tate weights. In addition, loci corresponding to a given Hodge type or potentially semi-stable Galois type will be cut out. We will conclude by stating these results for the universal spaces of Galois representations $\CRep^\square_\Db$ and $\CRep_\Db$, producing algebraic versions of these spaces, and drawing conclusions about their geometry in equi-characteristic 0. 

\subsection{Families of Semi-stable Galois Representations with Bounded Hodge-Tate Weight}

The following theorem shows that a semi-stability condition with bounded Hodge-Tate weights cuts out a closed locus and that the corresponding period maps interpolate along this locus. We have followed the techniques of Kisin \cite{pssdr}; see \cite{HH2013} for another approach.

\begin{thm}[{Generalizing \cite[Thm.\ 2.5.5 and Prop.\ 2.7.2]{pssdr}}]
\label{thm:pssdr_2.5.5}
If $h$ is a non-negative integer, then there exists a quotient $A^{\st,h}$ of $A$ with the following properties. 
\begin{enumerate}
\item If $B$ is a finite $\bQ_p$-algebra, and $\zeta: A \ra B$ a map of $\bQ_p$-algebras, then $\zeta$ factors through $A^{\st, h}$ if an only if $V_B = V_A \otimes_A B$ is semi-stable with Hodge-Tate weights in $[0,h]$.
\item There is a projective $W_{A^{\st, h}}$-module $D_{A^{\st, h}}$ of rank $d$ equipped with a Frobenius semi-linear automorphism $\varphi$ and with a $W_{A^{\st, h}}$-linear automorphism $N$ such that for all $\zeta:A \ra B$ factoring through $A^{\st,h}$, there is a canonical isomorphism
\[
D_B = D_{A^{\st,h}} \otimes_{A^{\st,h}} B \lrisom \Hom_{B[G_K]}(V_B, \Bpst \otimes_{\bQ_p} B)
\]
respecting the action of $\varphi$ and $N$.
\item Suppose that $A = A^{\st,h}$.  Then the map
\begin{equation}
\label{eq:pssdr_2.7.3}
D_A \otimes_{W_A} \Bsta \lra \Hom_A(V_A, \Bsta)
\end{equation}
induced from \eqref{eq:period_st} by setting $B = A$ and tensoring by $\otimes_{\Bpsta} \Bsta$ is an isomorphism compatible with $\varphi, N$, and the action of $G_K$. In particular, 
\begin{equation}
\label{eq:pssdr_2.7.4}
D_A \lrisom \Hom_{A[G_K]}(V_A, \Bpsta).
\end{equation}
\end{enumerate}
\end{thm}
\begin{proof}
For parts (1) and (2), one may apply the arguments of \cite[\S2.5]{pssdr}, as these are matters of points valued in finite $\bQ_p$-algebras and the constructions of the appropriate quotients of $A$ have been made. We will give a sketch of these arguments for the convenience of the reader. 

If $V_B$ is semi-stable with Hodge-Tate weights in $[0,h]$, then it has $E$-height $\leq h$, and $\zeta$ factors through $\Spec A^{\circ, \leq h}[1/p] = \Spec A^{\leq h}$. The theory of \cite[Cor.\ 1.3.15]{crfc} (summarized in \cite[\S2.5]{pssdr}) gives a finite free $\frS[1/p]$-module $\tilde \frM_B$ associated to $\tilde D_B = \Hom_{B[G_K]}(V_B, \Bpstb)$, equipped with a Frobenius semi-linear map of $E$-height $\leq h$ and satisfying certain additional properties. Such lattices are unique (\cite[Prop.\ 2.1.12]{crfc}), so $\tilde \frM_B$ may be identified with the specialization of the universal such lattice, $\frM_B = \frM_{A} \otimes_{A} B$. This induces a map $N$ on $D_B = \frM_B/u\frM_B$. Checking that this map $N$ is compatible with the action of $G_K$ in the sense of Proposition \ref{prop:pssdr_2.4.7}, it follows that $A^{\leq h} \ra B$ must factor through $A^{\st, \leq h}$.

Conversely, if $A^{\leq h} \ra B$ factors through $A^{\st, h}$, then \eqref{eq:period_st} is injective and compatible with Galois actions by Proposition \ref{prop:pssdr_2.4.7}. Comparing dimensions, it is in fact an isomorphism, making $V_B$ semi-stable. Its Hodge-Tate weights are in $[0,h]$ by \cite[Lem.\ 1.2.2]{crfc}.

Part (2) follows from the identification of $D_B$ with $\tilde D_B$ for each map $A^{\st, h} \ra B$ where $B$ is a finite $\bQ_p$-algebra, implying that $D_B$ is projective. The projectivity of $D_{A^{\st, h}}$ as a $W_{A^{\st, h}}$-module is implied by the projectivity of this module after $\otimes_{W_{A^{\st,h}}} W_B$ for each such $B$. Indeed, it suffices to check projectivity over $A^{\st, h}$, which is equivalent to projectivity after base extension to each $B$ by Lemma \ref{lem:jacobson}(4).

We must fully explain the proof of part (3). By part (2) and Lemma \ref{lem:period_injective}, \eqref{eq:pssdr_2.7.3} is an injective map of projective $\Bsta$-modules of rank $d$.  Therefore it will suffice to show that this map induces an isomorphism on top exterior powers, and we may freely restrict ourselves to the case that $d=1$.  
 
 In the one-dimensional case, $V_{A^\circ}$ arises by extension of scalars $\otimes_{R_\Db} A^\circ$ \cite[Prop.\ 3.13]{chen2014}; this is the case because 1-dimensional representations are identical to 1-dimensional pseudorepresentations.  It is then evident that \eqref{eq:pssdr_2.7.3} arises by $\otimes_{B_{\st,R}} B_{\st,A}$ from the same map where $A$ is replaced by $R$, and that it suffices to prove that \eqref{eq:pssdr_2.7.3} is an isomorphism when $A  = R$. This was done in \cite[Prop.\ 2.7.2]{pssdr}. Then \eqref{eq:pssdr_2.7.4} is an isomorphism by Lemma \ref{lem:pssdr_2.7.1}.
\end{proof}

We need the following lemma in order to find the $G_K$-invariants in $\Bsta$.
\begin{lem}[{Generalizing \cite[Lem.\ 2.7.1]{pssdr}}]
\label{lem:pssdr_2.7.1}
For $i \geq 0$ there is an isomorphism 
\[
W_A \cdot t^i \risom \Hom_{A[G_K]}(A(i), \Bpsta)
\]
induced by multiplication by $p^{-r_i}$ for $r_i$ defined below, where $A(i)$ denotes $A$ with $G_K$ acting via the $i$th power of the $p$-adic cyclotomic character $\chi$.  In particular, if $B_{\mathrm{st}, A} := \Bpsta[1/t]$, then $B^{G_K}_{\mathrm{st}, A} = W_A$.
\end{lem}
\begin{proof}
The key part of the proof of the local case in \cite[Lem.\ 2.7.1]{pssdr} is that the $\chi^i$-isotypic part of $\Acac$ is given by $W_{A^\circ} \cdot t^i/p^{r_i}$ where $r_i$ is the greatest non-negative integer such that $t^i/p^{r_i}$ is in $\Ac$. Applying this to $A_{\cris, R_\Db}$, which we may do because $R_\Db$ is local, it follows that the $\chi^i$-isotopic part of $A_{\cris, R_\Db} \otimes_{R_\Db} A^\circ$ is $W_{R_\Db} \cdot t^i/p^{r_i} \otimes_{R_\Db} A^\circ \simeq W_{A^\circ} \cdot t^i/p^{r_i}$. Because the action of $G_K$ on $\Acac$ is continuous (where the topology is the $\mDb$-adic topology), the closure of $W_{A^\circ} \cdot t^i/p^{r_i}$ in $\Acac$ is the $\chi^i$-isotypic part. However, this module is already closed.

With this fact in place, the proof of \emph{loc.\ cit.}~supplies the rest of the argument.
\end{proof}

\subsection{$p$-adic Hodge Type} 
Our remaining goal is to find loci corresponding to more refined $p$-adic Hodge theoretic conditions, namely, a certain Hodge type or being potentially semi-stable of a certain Galois type.  In fact, these conditions will cut out connected components (in equi-characteristic 0). First we will address the Hodge type, following \cite[\S2.6]{pssdr} and \cite[\S A.4]{fmc}. For this, we fix an finite extension field $E$ of $\bQ_p$ and suppose that $A$ admits the structure of an $E$-algebra.

\begin{defn}
\label{defn:p-adic_Hodge_type}
Suppose we are given a finite dimensional $E$-vector space $D_E$ with a filtration of $D_{E,K} := D_E \otimes_{\bQ_p} K$ by $E \otimes_{\bQ_p} K$-submodules such that the associated graded is concentrated in degrees $[0,h]$.  Let $\bv:= \{D_E, \Fil^i D_{E,K}, i = 0, \dotsc, h\}$.
If $B$ is a finite $E$-algebra and $V_B \in \Rep^d_{G_K}(B)$ such that $V_B$ is a de Rham representation, then we say that $V_B$ is of \emph{$p$-adic Hodge type $\bv$} if the Hodge filtration on the associated filtered $(\varphi, N)$-module has the same graded degrees as $\bv$.  That is, $V_B$ has all its Hodge-Tate weights in $[0,h]$ and for $i = 0, \dotsc, h$ there is an isomorphism of $B  \otimes_{\bQ_p} K$-modules
\[
\gr^i \Hom_{B[G_K]}(V_B, \BdR \otimes_{\bQ_p} B) \lrisom \gr^i D_{E,K} \otimes_E B.
\]
\end{defn}

\begin{thm}[{\cite[Cor.\ 2.6.2]{pssdr}}]
\label{thm:pssdr_2.6.2}
With $\bv$ as above, there exists a quotient $A^{\st,\bv}$ of $A^{\st,h}$ corresponding to a union of connected components of $\Spec A^{\st, h}$ with the following property: if $B$ is a finite $E$-algebra and $\zeta: A \ra B$ is a map of $E$-algebras, then $\zeta$ factors through $A^{\st,\bv}$ if and only if $V_B = V_{A} \otimes_A B$ is semi-stable of $p$-adic Hodge type $\bv$. 
\end{thm}
The proof from \cite[Cor.\ 2.6.2]{pssdr} does not require any generalization to account for $A^\circ$ being non-local, in light of Lemma \ref{lem:jacobson}. However, we will sketch the proof in order to incorporate the erratum to \emph{loc.\ cit}.\ given in \cite[\S A.4]{fmc}.
\begin{proof}
By applying parts (1) and (2) of Thm.\ \ref{thm:pssdr_2.5.5}, the later parts of the proof of \cite[Cor.\ 2.6.2]{pssdr} explains that the finite $A$-module 
\[
\Fil^i \phz^*(\frM_A)/(E(u)\phz^*(\frM_A) \cap \Fil^i\phz^*(\frM_A))
\]
 realizes the $i$th part of the Hodge filtration of $D_B \otimes_{K_0} K$ when specialized to any finite $\bQ_p$-algebra $B$. Therefore, because these pieces of the filtration are projective $B$-modules, the $A$-module is projective by Lemma \ref{lem:jacobson}(4). Because the rank of a finite projective module is locally constant, $\Spec A^\bv$ is a union of connected components of $\Spec A$. One may then set $A^{\st, \bv} := A^{\st, h} \otimes_A A^\bv$. 
\end{proof}

\subsection{Galois Type} Next we will study families of \emph{potentially} semi-stable $G_K$-representations, following \cite[\S2.7.5]{pssdr}. We stipulate that $B$ is an Artinian local $E$-algebra with residue field $E$.  Let $V_B \in \Rep^d_{G_K}(B)$.  Following \cite{fontaine1994}, set
\[
D^*_\pst(V_B) = \varinjlim_{K'} \Hom_{B[G_{K'}]}(V_B, \Bst \otimes_{\bQ_p} B),
\]
where $K'$ runs over finite field extensions of $K$.

Let $\bar K_0 \subset \bar K$ denote the maximal unramified extension of $K_0$, and let $G_{K_0} \subset G_K$ be the inertia group of $G_K$.  Then $D^*_\pst(V_B)$ is a $B \otimes_{\bQ_p} \bar K_0$-module with a Frobenius semi-linear Frobenius automorphism $\varphi$, a nilpotent endomorphism $N$ such that $p \varphi N = N \varphi$, and a $B \otimes_{\bQ_p} \bar K_0$-linear action of $G_{K_0}$ which has open kernel and commutes with $\varphi$ and $N$.

Following \cite[\S2.7]{pssdr} along the line of reasoning of \cite[Lem.\ 1.2.2(4)]{mffgsm}, we see that $D^*_\pst(V_B)$ is finite and free as a $B \otimes_{\bQ_p} \bar K_0$-module. Since the action of $G_{K_0}$ commutes with the action $\varphi$, the traces of elements of $G_{K_0}$ are contained in $B$, and $D^*_\pst$ descends to a representation of $G_{K_0}$ on a finite free $B$-module $\tilde P_B$.  Because characteristic zero representations of finite groups are rigid, this representation must be an extension of scalars from a representation $P_B$ of $G_{K_0}$ over $E$. 

We have associated to a potentially semi-stable $d$-dimensional representation $V_B$ of $G_K$ over $B$ a representation of the inertia group of $K$ over $E$ which reflects the failure of $V_B$ to be semi-stable.  We will call this the ``Galois type'' of $V_B$, as follows.

Fix an algebraic closure $\bar \bQ_p$ of $\bQ_p$.  
\begin{defn}
Let $T: G_{K_0} \ra \GL_d(\bar \bQ_p)$ be a representation with open kernel.  We say that $V_B$ is \emph{potentially semi-stable of Galois type $T$} provided that $P_B$ defined above is isomorphic to $T$ over $\bar\bQ_p$.
\end{defn}
It is equivalent to say that for any $\gamma \in G_{K_0}$, the trace of $T(\gamma)$ is equal to the trace of $\gamma$ on $D^*_\pst(V_B)$.

Let $\bv$ be a $p$-adic Hodge type as in Definition \ref{defn:p-adic_Hodge_type}; fix a representation
\[
T: G_{K_0} \ra \End_E(D_E) \lrisom \GL_d(E).
\]
\begin{thm}[{Generalizing \cite[Thm.\ 2.7.6 and Cor.\ 2.7.7]{pssdr}}]
\label{thm:pssdr_2.7.6}
\ 
\begin{enumerate}
\item There exists a quotient $A^{T,\bv}$ of $A$ such that for any finite $E$-algebra $B$, a map of $E$-algebras $\zeta: A \ra B$ factors through $A^{T, \bv}$ if and only if $V_B = V_{A} \otimes_{A} B$ is potentially semi-stable of Galois type $T$ and Hodge type $\bv$.
\item There exists a quotient $A_{\mathrm{cr}}^{T,\bv}$ of $A$ such that for any finite $E$-algebra $B$, a map of $E$-algebras $\zeta: A \ra B$ factors through $A_{\mathrm{cr}}^{T,\bv}$ if and only if $V_B = V_{A} \otimes_{A} B$ is potentially crystalline of Galois type $T$ and Hodge type $\bv$.
\end{enumerate}
\end{thm}
These constructions may be repeated verbatim from \cite[Thm.\ 2.7.6 and Cor.\ 2.7.7]{pssdr}. We will give a sketch.
\begin{proof}
Let $L/K$ be a finite Galois extension such that $I_L \subseteq \ker T$. Theorem \ref{thm:pssdr_2.6.2} gives the existence of a quotient $A^{\pst, \bv}$ of $A$ such that $\zeta$ factors through $A^{\pst, \bv}$ if and only if $V_B \vert_{G_L}$ is semi-stable with Hodge type $\bv$. One then applies Theorem \ref{thm:pssdr_2.5.5}(3) and studies the action of the inertia subgroup $I_{L/K}$, which is $L_0$-linear and commutes with $\phz$, and therefore has trace function in $A^{\pst,\bv}$. As this inertia group is finite, its trace function is locally constant on $\Spec A^{\pst,\bv}$. The condition that its trace is $T$ therefore cuts out a union of connected components of $A^{\pst, \bv}$, as desired.

For the second result, first produce $A^{T,\bv}$ as above, and then take the quotient corresponding to the equation $N=0$, where $N$ is the endomorphism of $D_{A^{\st,h}}$ defined in Theorem \ref{thm:pssdr_2.5.5}(3).
\end{proof}

\subsection{Universal Families of Representations, and Algebraization}
\label{subsec:universal_C}

We will summarize what we have proved by producing universal spaces of potentially semi-stable Galois representations with bounded Hodge-Tate weights. These can then be algebraized using Theorem A, under some assumptions. In particular, let $C$ be one of the following conditions on representations of $G_K$ over a finite $\bQ_p$-algebra.
\begin{enumerate}
\item Crystalline with Hodge-Tate weights in the range $[a,b]$.
\item Semi-stable with Hodge-Tate weights in the range $[a,b]$.
\item Any of the above two conditions, with fixed Hodge type $\bv$.
\item Any of the above three conditions after restriction to $G_L$, for some finite field extension $L/K$.
\item Condition (4) with $L/K$ a Galois extension, and in addition, a particular Galois type $T$.
\end{enumerate}

We will use the following notation for the universal spaces, in analogy with \S\ref{subsec:PHT_background_2}.
\begin{center}
\begin{tabular}{| l | c | c | c | c | }
\hline
Formerly: & $\Rep_\Db^{(\square)}$ & $\CRep_\Db^{(\square)}$ & $\Rep_\Db^{(\square)}[1/p]$ & $\CRep_\Db^{\square}[1/p]$ \\
\hline
Henceforth: & $\Rep^{(\square),\circ}_\Db$ & $\CRep^{(\square),\circ}_\Db$  & $\Rep_\Db^{(\square)}$ & $\CRep_\Db^{\square}$  \\ 
\hline
\end{tabular}
\end{center}

Consider the case where $A^\circ$ is the coordinate ring of $\CRep^{\circ, \square}_\Db$, the universal formal moduli scheme of framed representations of $G_K$ with residual pseudorepresentation $\Db$. This admits an action of $\widehat{\GL}_d$, the $\mDb$-adic completion of $\GL_d \otimes_\bZ R_\Db$. The results from the previous sections produce a quotient of $A^C$ of $A = A^\circ[1/p]$ of representations satisfying condition $C$. The ideal $I^C \subset A$ such that $A/I^C = A^C$ is stable under the action of $\widehat{\GL}_d$; this is the case because the $B$-valued points in the $C$-locus of $\Spec A$, where $B$ is a finite $\bQ_p$-algebra, are obviously preserved by the action of $\widehat{\GL}_d$, and these points characterize the $C$-locus by Lemma \ref{lem:jacobson}(3). 

The kernel $I^{\circ, C} \subset A^\circ$ of the natural map $A^\circ \ra A^C$ cuts out a quotient $A^{\circ, C} := A^\circ/I^{\circ, C}$ such that $A^{\circ,C}[1/p] \risom A^C$ and therefore $A^{\circ,C}$ has the same property on $B$-points. Moreover, $I^{\circ, C}$ is $\widehat{\GL}_d$-stable. We summarize our discussion in this 
\begin{thm}
\label{thm:C_univ}
Given any of the conditions $C$ above, there is a closed substack $\CRep_\Db^{\circ,C}$ of $\CRep^\circ_\Db$, formally of finite type over $\Spf R_\Db$, such that for any finite $\bQ_p$-algebra $B$ and representation $V_B$ of $G_K$ with residual pseudorepresentation $\Db$, there exists a model $V_{B'}$ for $V_B$, where $B' \in \Int_B$, such that the corresponding map $\zeta: \Spf B' \ra \CRep^\circ_\Db$ factors through $\CRep_\Db^{\circ,C}$ if and only if $V_B$ has property $C$.
\end{thm}

The statement of the theorem is an example of the way that we think of the ``generic fiber'' over $\bQ_p$ of a $\Spf \bZ_p$-formal stack while considering the formal stack only as a limit of algebraic stacks over $\Spec R_\Db/\frmm_\Db^i$. The theorem expresses that the locus of such points has an integral model and is $\GL_d$-stable, as $\CRep_\Db$ is a quotient stack of $\CRep^{\circ, \square}_\Db = \Spf A^\circ$. 

\begin{proof}
Without loss of generality we may assume that $B$ is local. Choose a basis for $V_B$, so that we have a corresponding map $A \ra B$. If $V_B$ satisfies $C$, it factors through $A^C$. By Lemma \ref{lem:jacobson}(4), there exists $B' \in \Int_B$ and a natural map $A^{\circ,C} \ra B'$ such that the resulting representation $V_{B'}$ satisfies $V_{B'} \otimes_{B'} B \cong V_B$, as desired.

Conversely, if there exists $B' \in \Int_B$ and a $G_K$-representation $V_{B'}$ such that the corresponding map $\Spf B' \ra \CRep^\circ_\Db$ factors through $\CRep^{\circ,C}_\Db$, then it is clear that $V_{B'} \otimes_{B'} B$ satisfies condition $C$.

Consequently, $\CRep^{\circ, C}_\Db$ may be taken to be the quotient stack $[\Spf A^{\circ, C}/ \widehat{\GL}_d]$.
\end{proof}

Recall from Theorem \ref{thm:PsR_background} that there exists a algebraic model of $\Rep^\circ_\Db$ of $\CRep^\circ_\Db$ given by representations of the associated universal Cayley-Hamilton algebra. When formal GAGA holds for $\psi: \Rep^\circ_\Db \ra \Spec R_\Db$ (see \S\ref{subsec:FGAMS}), we can algebraize the universal family of potentially semi-stable representations. 
\begin{cor}
\label{cor:C_univ_alg}
Assume that either 
\begin{enumerate}
\item $\Db$ is multiplicity-free, or assume
\item the algebraization hypothesis (FGAMS) of \S\ref{subsec:FGAMS} is true.
\end{enumerate}
Then there is a closed substack $\Rep_\Db^{\circ,C}$ of $\Rep^\circ_\Db$ such that for any finite $\bQ_p$-algebra $B$, and representation $V_B$ of $G_K$ with residual pseudorepresentation $\Db$, the corresponding map $\zeta: \Spec B \ra \Rep^\circ_\Db$ factors through $\Rep_\Db^{\circ,C}$ if and only if $V_B$ has property $C$.
\end{cor}
\begin{proof}
In either case, formal GAGA holds for $\psi: \Rep_\Db \ra R_\Db$ by Theorem \ref{thm:gaga}. We apply Lemma \ref{lem:AG_background}(1c) to find a natural corresponding closed immersion $\Rep^{\circ,C}_\Db \rinj \Rep^\circ_\Db$ to the closed immersion $\CRep^{\circ,C}_\Db \rinj \CRep^\circ_\Db$. 
\end{proof}

\begin{rem}
\label{rem:GAGA}
In the corollary above, we have invoked formal GAGA for $\psi$ produce an $R_\Db$-algebraic universal family of potentially semi-stable representations $\Rep^{\circ,C}_\Db$ after first producing a formal version. However, if one freely invokes formal GAGA from the start, it is possible to carry out the construction of algebraic universal families of potentially semi-stable Galois representations in Theorem \ref{thm:C_univ} directly. That is, using formal GAGA for $\psi$ freely, it is possible to carry out all of the work of \S\ref{subsec:universality_uM} and \S\ref{sec:periods} with the $\mDb$-adically separated finitely generated $R_\Db$-subalgebra $A^\circ_\alg \subset A^\circ$ of Corollary \ref{cor:G_coefs} in place of $A^\circ$. For example, even the Cauchy sequence used to construct the map $\xi$ of Lemma \ref{lem:pssdr_2.2} can be shown to be have algebraic coefficients, i.e.\ defining a map $D_{A_\alg} \otimes_{W_{A_\alg}} \cO_{A_\alg} \ra \cM_{A_\alg}$, where $\cO_{A_\alg} = \cO_{R} \otimes_{R} A_\alg$.

In this sense, once we know formal GAGA, the construction of $\Rep^{\circ,C}_\Db$ is not merely algebro-geometric, but is natural in that all of the semi-linear algebraic data and period maps exist algebraically relative to $R_\Db$. However, we have constructed potentially semi-stable loci in the formal setting first, so that Theorem \ref{thm:C_univ} is not conditional on assumption (FGAMS). 
\end{rem}

Here are some geometric properties of the generic fiber over $\bZ_p$ of these algebraic stacks of representations, deduced from established ring-theoretic properties of equi-characteristic zero deformation rings of Galois representations. 
\begin{prop}
\label{prop:univ_C}
For any of the properties $C$ as above, $\Rep^C_\Db$, $\Rep^{\square, C}_\Db$, and $\CRep^{\square, C}_\Db$
are reduced and locally complete intersection.  They are also formally smooth over $\bQ_p$ on a dense, open substack. 
\begin{enumerate}
\item When $C$ is a potentially crystalline condition, each of these spaces is formally smooth over $\bQ_p$.
\item When $C$ has a fixed $p$-adic Hodge type $\bv$, $\Rep^{\square, C}_\Db$ and $\CRep^{\square, C}_\Db$ are equi-dimensional of dimension $d^2 + \dim_E \ad D_{E,K}/\Fil^0 \ad D_{E,K}$.
\end{enumerate}
\end{prop}

In view of \cite[Ch.\ 0, Thm.\ 22.5.8]{ega4-1}, the formal smoothness of these spaces over $\bQ_p$ is equivalent to their being regular. We will work with the latter condition in the proof. 

\begin{proof}
Bellovin \cite{bellovin2014} proves that for any $p$-adic field-valued representation $\rho$ satisfying $C$, the complete local ring ring parameterizing liftings of $\rho$ with property $C$ is complete intersection and reduced. It is also equi-dimensional of the dimension given in the statement of part (2) by \cite[Thm.\ 3.3.4]{pssdr} when $C$ has a fixed Hodge type $\bv$. As these rings are the complete local rings of the closed points of the excellent Jacobson scheme $\Rep^{\square,C}_\Db$, we know that $\Rep^{\square, C}_\Db$ is reduced and locally complete intersection. Indeed, see \cite[Cor.\ 3.3]{GM1978} for the openness of the complete intersection locus of an excellent ring. We also know that $\Rep^{\square, C}_\Db$ is equi-dimensional as in statement (2) when $C$ implies a fixed $p$-adic Hodge type.  Because being complete intersection and reduced is local in the smooth topology, these properties hold for $\Rep^C_\Db$ as well. By the main result of \cite{valabrega1976}, the coordinate ring of $\CRep^{\square, \circ, C}_\Db$ is excellent, and therefore so is the coordinate ring of $\CRep^{\square,C}_\Db$. The arguments above may then be applied in this case as well. 

Kisin in \cite[Thms.\ 3.3.4 and 3.3.8]{pssdr} and \cite[Thm.\ A.2]{fmc} proves the generic regularity statement and part (1) of the proposition above, but with the generic fiber of a framed deformation ring $R^\square_{V_\bF}$ of a residual representation $V_\bF$ in place of $\Rep_\Db$. We will deduce our claim from this case. First, let us prove generic regularity in $\Spec A$ for $A = A^\circ[1/p]$ where $\Spf A^\circ = \CRep^{\square, \circ, C}_\Db$. The arguments for generic regularity in \cite{pssdr} are statements 3.1.6, 3.2.1, and 3.3.1 of \emph{loc.\ cit}. Their validity and their application in the proof of \cite[Thm.\ 3.3.4]{pssdr} generalize verbatim from the case that $A^\circ$ is a complete Noetherian local $\bZ_p$-algebra to the case that $A^\circ$ is topologically finite type over $\bZ_p$, with the exception of Prop.\ 3.3.1 of \emph{loc.\ cit}. We deduce the non-local case of this statement in Lemma \ref{lem:pssdr_3.3.1} below. This gives us the generic regularity of $\CRep^{\square, C}_\Db$. 

In particular, Kisin's arguments identify the singular locus with the support of a finite $A$-module $H^2(D_A)$ produced out of $D_A$ and its structure maps $\phz, N$. This construction may be carried out over $\Rep^C_\Db$ to produce a coherent sheaf $H^2(D_\Db)$. As the support of $H^2(D_\Db)$ is nowhere dense after changing base to the fpqc cover $\CRep^{\square,C}_\Db$ (Lemma \ref{lem:AG_background}(2c)), it is also nowhere dense in $\Rep^C_\Db$, showing that $\Rep^C_\Db$ is regular on a dense, open substack. This can also be seen by noting that complete local rings at closed points in $\CRep^{\square,C}_\Db$ are formally smooth over $\Rep^C_\Db$. 
\end{proof}

We freely employ the notation of \cite[\S3]{pssdr} and \cite[\S2.3]{mffgsm} in the following 
\begin{lem}[{Generalizing \cite[Prop.\ 3.3.1]{pssdr}}]
\label{lem:pssdr_3.3.1}
Let $f: \Spf A^\circ \ra \CRep^\circ_\Db$ be formally smooth and let $C$ be one of the conditions above, so that $A^C$ is a quotient of $A$ parameterizing representations with condition $C$. Then for a closed point $x \in A^C$ corresponding to a maximal ideal $\frmm = \frmm_x \subset A^C$, the morphism $\Spf \hat A^C_\frmm \ra \fMod_{F,\phz,N}$ is formally smooth. 
\end{lem}
\begin{proof}
Let $V_{A^\circ}$ be the rank $d$ $G_K$-representation corresponding to $f$ with specialization $V_x$ at $x$ to a representation with coefficients in the $p$-adic residue field $E = A/\frmm$. By Lemma \ref{lem:jacobson}(3), the map $A^\circ \ra E$ factors through its ring of integers $O_E \subset E$, giving a choice of $G_K$-stable lattice $V^\circ_x \subset V_x$. Let $V_\bF$ denote $V^\circ_x \otimes_{O_E} \bF$, where $\bF$ is the residue field of $O_E$. Let $D_{V_\bF}$ denote its deformation groupoid as in \cite[\S3]{pssdr}. Note that $D_{V_\bF} \ra \Rep_\Db$, as a morphism of groupoids on complete Noetherian $W(\bF)$-algebras with residue field $\bF$, is schematic. Then, using the notation of \cite[\S2.3]{mffgsm}, we observe that there is an isomorphism of $\widehat{\frA\frR}_{W(\bF), (O_E)}$-groupoids $D_{V_\bF, (V_x^\circ)} \simeq \CRep_{\Db, (V_x^\circ)}$. Following the arguments of \cite[\S2.3]{mffgsm}, one may check that the complete local $\bZ_p$-algebra $A'^\circ$ given by $\Spf A'^\circ = \Spf A^\circ \times_{\Rep_\Db} D_{V_\bF}$ has a map $x': A' = A'^\circ[1/p] \ra E$ factoring $x: A \ra E$ and that $\hat A_{\frmm_x} \ra \hat A'_{\frmm_{x'}}$ is an isomorphism. This is all we need to reduce the proof to the case that $A^\circ$ is local, which then follows by \cite[Prop.\ 3.3.1]{pssdr}.
\end{proof}

\section{Potentially Semi-stable Pseudodeformation Rings}

\subsection{Potentially Semi-Stable Pseudorepresentations}
\label{subsec:psspsr}
We must clarify what it means to ask if a pseudorepresentation satisfies some property which, a priori, only applies to representations. 
\begin{defn}
\label{defn:pseudo_property}
Let $\cK$ be a full subcategory of the category of perfect fields which is closed under finite extensions, let $\cD$ be a setoid of pseudorepresentations fibered over $\cK$,\footnote{That is, given a morphism of fields $E \ra E' \in \cK$ and an $E$-valued pseudorepresentation $D \in \cD$, the base change $D \otimes_E E'$ is in $\cD$.} and let $\CRep$ be a groupoid of representations fibered over $\cK$.  Let $\cP$ be a full subcategory of $\CRep$ of representations with property $P$ such that if $V \in \cP(K)$, then its semisimplification $V^{ss}$ and any finite base change $V \otimes_K K'$ are each in $\cP$.

Then a pseudorepresentation $D \in \cD$ over $K \in \cK$ \emph{has property $P$} if, given a finite extension $K'/K$ such that there exists a semi-simple representation $V^{ss}_D \in \CRep(K')$ such that $\psi(V^{ss}_D) = D \otimes_K K'$ (which exists by Cor.\ \ref{cor:ss_degree}), $V$ has property $P$.
\end{defn}

For example, one can let $\cK$ be the category of $p$-adic fields, let $\cD$ and $\CRep$ be the continuous pseudorepresentations and representations of $G_{\bQ_p}$ over $p$-adic fields, and let the property $P$ be ``crystalline,'' or any of the conditions of \S\ref{subsec:universal_C}.

While it seems possible to emulate Definition \ref{defn:pseudo_property} over non-fields or non-perfect fields if appropriate conditions on $P$ are imposed, we do not pursue this here. 

We now return to the case of $C$ being a potentially semi-stable condition as in the previous section. Recall that $R = R_\Db[1/p]$ and that $\Rep^C_\Db$ exists unconditionally when $\Db$ is multiplicity-free.
\begin{thm}
\label{thm:psspsr}
If $\Rep^C_\Db$ exists, there exists a canonical quotient $R^C$ of $R$ with the property that for any finite field extension $E$ of $W(\bF)[1/p]$, the map $z: R \ra E$ factors through $R^C$ if and only if the semi-simple representation associated to the pseudorepresentation corresponding to $z$ satisfies the condition. This quotient $R^C$ is reduced.
\end{thm}

\begin{proof}
We have from Corollary \ref{cor:C_univ_alg} that there is a closed subscheme $\Rep^C_\Db$ of $\Rep_\Db$. Because $\psi$ is universally closed, the scheme-theoretic image $\Spec R^C$ of $\Rep^C_\Db$ under $\psi$ defines a closed subscheme of $\Spec R$, and $\psi$ restricted to $\Rep^C_\Db$ is a good moduli space over $\Spec R$ because $\Rep^C_\Db$ is realizable as a quotient stack $[\Rep^{\square, C}_\Db/ \GL_d]$ with GIT quotient ring $R^C = \Gamma(\cO_{\Rep^{\square, C}_\Db})^{\GL_d}$. 

Having constructed the quotient $R^C$ of $R$, we show that it has the desired property. Choose a closed point $\zeta: \Spec E \ra R^C$. By Theorem \ref{thm:PsR_background}(2), there exists a unique closed point $z$ in the fiber of $\psi$ in $\Rep_\Db$ over $\zeta$, with residue field some finite extension $E'/E$, corresponding to the unique semi-simple representation inducing $\zeta$. Because $\Rep^C_\Db \rinj \Rep_\Db$ is a closed immersion, we must have $z \in \Rep^C_\Db$. 

When $\Rep_\Db^C$ is reduced, then $R^C$ is also \cite[Thm.\ 4.16(viii)]{alper2013}. Then, the uniqueness of $R^C$ follows from Lemma \ref{lem:jacobson} and the fact that it is reduced. 
\end{proof}

\begin{cor}
If $\Rep^C_\Db$ exists, there exists a quotient $R_\Db^C$ of $R_\Db$ with the property that for any finite field extension $E$ of $W(\bF)[1/p]$, the map $z: R_\Db \ra E$ factors through $R_\Db^C$ if and only if the semi-simple representation associated to the pseudorepresentation corresponding to $z$ satisfies the condition. There is a unique such quotient which is reduced, namely the image of $R_\Db$ in $R^C$. 
\end{cor}
\begin{proof}
One may take $R_\Db^C$ to be any quotient of $R_\Db$ such that it realizes $R^C$ after inverting $p$. 
\end{proof}

The generic fiber $\Spec R^C$ is \emph{pseudo-rational} when $C$ is a crystalline condition.
\begin{defn}[{\cite[\S6.1]{schoutens2008}}]
\label{defn:pseudo-rational}
A Noetherian local ring $(R, \frmm)$ is called \emph{pseudo-rational} if it is analytically unramified, normal, Cohen-Macaulay, and for any projective birational map $f: Y \ra \Spec R$ with $Y$ normal, the canonical epimorphism between the top cohomology groups $\delta: H^d_\frmm(R) \ra H^d_Z(Y)$ is injective, where $Z$ is the closed fiber $f^{-1}(\frmm)$ and $d$ the dimension of $R$. A Noetherian ring $A$ is called \emph{pseudo-rational} if $A_\frp$ is pseudo-rational for every prime ideal $\frp$ in $A$.
\end{defn}

The notion of ``pseudo-rational'' is a generalization, to rings over which no resolution of singularities exists, of the notion of rational singularities for finite type algebras over a characteristic zero field. The work \cite{schoutens2008} of Schoutens is a generalization of the Hochster-Roberts theorem to this setting.

\begin{cor}
If the condition $C$ implies potentially crystalline, $R^C$ is pseudo-rational. In particular, it is reduced, normal, and Cohen-Macaulay. 
\end{cor}
\begin{proof}
Write $S$ for the coordinate ring of the regular affine scheme $\Rep^{\square,C}_\Db$, so that $R^C =S^{\GL_d} \rinj S$ is an inclusion of the invariant subring by the action of basis change. Therefore the map $R^C \rinj S$ is cyclically pure (also known as ideally closed), cf.\ \cite[Remark 4.13]{alper2013}. The main theorem \cite[Thm.\ A]{schoutens2008} states that a cyclically pure subring of a regular Noetherian equi-characteristic zero ring is pseudo-rational. Therefore $R^C$ is pseudo-rational, and hence also formally unramified, normal, and Cohen-Macaulay \cite[\S4]{schoutens2008}. Formally unramified is equivalent to reduced, since $R^C$ is finitely generated over the excellent ring $R_\Db$. 
\end{proof}

\begin{rem}
Reducedness and normality of $R^C$ are clear from the regularity of $S$ without resorting to Schoutens' result. Cf.\ also \cite[Thm.\ 4.16(viii)]{alper2013}. 
\end{rem}

There is often interest in understanding the connected components of potentially semi-stable deformation rings. It is no more complicated to study the connected components of potentially semi-stable pseudodeformation rings. Using the fact that each of the maps $\Rep^\square_\Db \ra \Rep_\Db \ra \Spec R$ and $\CRep^\square_\Db \ra \Rep_\Db$ is surjective with connected geometric fibers (cf.\ \cite[Thm.\ 4.16(vii)]{alper2013}, Thm.\ \ref{thm:PsR_background}(2)), the analysis of the geometrically connected components of $R^C$ amounts to analysis of the geometrically connected components of the affine scheme $\Rep^\square_\Db$. 
\begin{cor}
There is a natural bijective correspondence between the geometrically connected components of each of $\CRep^{\square,C}_\Db$, $\Rep^{\square,C}_\Db$, $\Rep^C_\Db$, and $\Spec R^C$.
\end{cor}

\subsection{Global Potentially Semi-stable Pseudodeformation Rings}
\label{subsec:global_psspsdr}
In this section, we will assume that all algebraizations of stacks of potentially semi-stable representations exist.  Let $F/\bQ$ be a number field, let $S$ be a finite set of places of $F$ containing those over $p$, and take $\Db: G_{F,S} \ra \bF$ to be a global Galois pseudorepresentation ramified only at places in $S$. As $G_{F,S}$ satisfies Mazur's $\Phi_p$ finiteness condition, the universal ramified-only-at-$S$ pseudodeformation ring  $R_\Db$ of global Galois representations is Noetherian (Thm.\ \ref{thm:chen_3.14}). Fix decomposition subgroups $G_v \subset G_{F,S}$ for places $v \in S$. In analogy to a common construction in the case of deformations of Galois representations, we want to find a quotient $R^C_\Db$ of $R_\Db$, $C = (C_v)_{v \in S}$, parameterizing \emph{pseudo}deformations which satisfy certain conditions $C_v$ at each $v \in S$, such as a condition $C_v$ coming from $p$-adic Hodge theory when $v \mid p$. 

In the case of deformations of a irreducible Galois representation $\br: G_{F,S} \ra \GL_d(\bF)$, one may accomplish this construction using the natural maps $R^v_\br \ra R_\br$ from a local deformation ring to a global deformation ring (usually discussed as a ``deformation condition'' to avoid unnecessary technical complications when $\br\vert_{G_v}$ is not irreducible), and the quotients $R^v_\br \ra R^{C_v}_\br$ corresponding to the condition $C_v$ on representations of $G_v$ deforming $\br\vert_{G_v}$. Then one sets 
\[
R^C_\br := R_\br \otimes_{\left(\bigotimes_{v \in S} R^v_\br\right)} \left(\otimes_{v \in S} R^{C_v}_\br\right)
\]
in order to obtain a deformation ring parameterizing representations of $G_{F,S}$ deforming $\br$ with conditions $C_v$ upon restriction to $G_v$.

In contrast, one does not want to do the same construction with pseudodeformation rings (as if $\Db$ replaced $\br$ in each place in the line above), even though the corresponding maps $R^v_\Db \ra R_\Db$ and $R^v_\Db \ra R^{C_v}_\Db$ exist. The reason is that if $\br$ is irreducible but $\br\vert_{G_v}$ is not, then a deformation $D: G_{F,S} \ra E$ of $\Db$ (where $E/\bQ_p$ is a finite extension) such that $D\vert_{G_v}$ is reducible may have information about extensions between the Jordan-H\"older factors of $D\vert_{G_v}$, while $D\vert_{G_v}$ lacks this information. If a condition $C$ is sensitive to the extension classes in a representation, then we may get too large of a quotient in this way. 

Instead, the following construction is appropriate: fixing $\Db$ as above, we have the Noetherian moduli stack $\Rep^\circ_\Db$ of representations of $G_{F,S}$ inducing residual pseudorepresentation $\Db$; it is algebraizable of finite type over $\Spec R_\Db$ via $\psi$. There are also analogous local spaces $\Rep^{\circ, v}_\Db \ra R^v_\Db$. If $C = (C_v)_{v \in S}$ is a collection of conditions on representations of $G_v$ for $v \in S$ cutting out closed immersions $\Rep^{\circ, C_v}_\Db \rinj \Rep^{\circ,v}_\Db$, we may set
\begin{equation}
\label{eq:construct_Rep_C}
\Rep^{\circ,C}_\Db := \Rep^\circ_\Db \times_{\left(\times_{v \in S} \Rep^{\circ,v}_\Db\right)} \left(\times_{v \in S} \Rep^{\circ,C_v}_\Db\right),
\end{equation}
the universal moduli stack of representations of $G_{F,S}$ satisfying condition $C$ and inducing residual pseudorepresentation $\Db$. 

Having constructed this space, we may construct the pseudodeformation ring parameterizing pseudodeformations with property $C$.
\begin{defn}
\label{defn:pseudo_C}
Let $R^C_\Db$ be the quotient of $R_\Db$ associated to the scheme-theoretic image of the universally closed composition map
\begin{equation}
\label{eq:incl_psi}
\Rep^{\circ,C}_\Db \rinj \Rep^\circ_\Db \buildrel\psi\over\lra \Spec R_\Db.
\end{equation}
\end{defn}

\begin{thm}
\label{thm:global_main}
Given a residual pseudorepresentation $\Db: G_{F,S} \ra \bF$, let $C = (C_v)_{v \in S}$ be a collection of conditions on finite $\bQ_p$-algebra valued points of $\Rep_\Db$. Assume that each $C_v$ cuts out a closed immersion $\Rep^{C_v}_\Db \ra \Rep^v_\Db$ for each $v \in S$.

Then there exists a quotient $R^C_\Db$ of $R_\Db$ such that for any finite extension $E/\bQ_p$ and closed point $z: \Spec E \ra \Spec R_\Db$, $z$ factors through $\Spec R^C_\Db \rinj \Spec R_\Db$ if and only if $D_z$ satisfies $C$, i.e.\ the corresponding semi-simple representation $V^\mathrm{ss}_z$ of $G_{F,S}$ satisfies $C$. 
\end{thm}

\begin{proof}
After possibly allowing a finite extension of the coefficient field $E'/E$, there exists a unique representation $V^\mathrm{ss}_z$ of $G_{F,S}$ over $E'$ inducing $D_z$. Recall that $V^\mathrm{ss}_z$ induces the unique closed point $\zeta$ of $\Rep_\Db$ lying over $z \in \Spec R_\Db$. Therefore, because $\psi: \Rep_\Db \ra \Spec R_\Db$ is universally closed and $\Rep^C_\Db \rinj \Rep_\Db$ is a closed immersion, we must have $\zeta \in \Rep^C_\Db$ if and only if $z \in \Spec R^C_\Db$. This means that $V^\mathrm{ss}_z$ satisfies $C$, i.e.\ its restriction to $G_v$ satisfies $C_v$ for all $v \in S$.
\end{proof}

\subsection{Example: Ordinary Pseudodeformation Rings}
\label{subsec:ordinary_psspsr}

The first local deformation conditions commonly dealt with were ``ordinary'' \cite{mazur1989} and ``flat'' \cite{ramakrishna1993}. The constructions of this paper result in a ``flat pseudodeformation ring'' since flat is equivalent to ``crystalline with Hodge-Tate weights in $[0,1]$'' \cite{breuil2000,crfc}. However, because the ordinary condition allows for arbitrary Hodge-Tate weights, it is not included in these constructions. Here, we will construct an ordinary pseudodeformation ring in the 2-dimensional case; this should be compared with the ordinary pseudodeformation ring of \cite[\S3]{CV2003}. There are several notions of ``ordinary;'' we use the following one. 

\begin{defn}
\label{defn:ordinary}
Let $K$ be a $p$-adic local field. We call a 2-dimensional representation of $G_K$ \emph{ordinary} when it is reducible and there exists an unramified $1$-dimensional quotient. 
\end{defn}

For simplicity we will address representations of $G_{\bQ,S}$ where $p \in S$, and cut out a locus of representations satisfying ordinariness with respect to a choice of decomposition group with its inertia group at $p$, $G_{\bQ,S} \supset G_p \supset I_p$. We will find a quotient of $R_\Db$ parameterizing ordinary pseudorepresentations. Obviously, the ordinary condition is sensitive to extension classes, so that a $2$-dimensional pseudorepresentation $D$ of $G_{\bQ,S}$ such that $D\vert_{G_p} \simeq \det \circ (\psi \oplus \chi)$ where $\psi$ is unramified is not necessarily ordinary. 

We let $\Db$ arise from the sum of two characters $\bar\psi, \bar\chi$ valued in $\bF^\times$, writing $\Db = \det(\bar\psi \oplus \bar\chi)$ and stipulating that $\bar\psi\vert_{I_p} = 1$ so that the set of ordinary pseudodeformations of $\Db$ is not empty. We also assume that $\bar\psi\vert_{G_p} \neq \bar\chi\vert_{G_p}$. 

\begin{rem}
One may naturally ask if there is a reasonable generalization of the ordinary condition to $n$-dimensional representations. For simplicity, assume that $\Db$ splits into a sum of characters of $G_{\bQ,S}$. In this setting, ordinary will mean totally reducible when restricted to $G_p$ (no condition on the inertia action). By observing the following construction, one can see that this will be possible if the residual characters are pairwise \emph{distinct} after restriction to $G_p$, generalizing the ``residually $p$-distinguished'' condition $\bar\psi\vert_{G_p} \neq \bar\chi\vert_{G_p}$. 
\end{rem}

\begin{lem}
\label{lem:ord_space}
With the assumptions above, there exists a closed substack $\Rep_{G_p, \Db}^{\ord}$ of $\Rep_{G_p,\Db}$ parameterizing representations of $G_p$ with induced residual pseudorepresentation $\Db$ that are ordinary. 
\end{lem}
\begin{proof}
We will use Theorem \ref{thm:R_E_rep_equiv} and consider representations of $E(G_p)_\Db$. As $\Db$ is multiplicity free, Theorem \ref{thm:CH_MF_GMA} gives us a generalized matrix algebra structure on $E(G_p)_\Db$: the data of two idempotents, $e_1$ associated to the factor $\bar\chi$ of $\br^\mathrm{ss}_\Db = \bar\psi \oplus \bar\chi$ and $e_2$ associated to $\bar\psi$. We write the generalized matrix algebra in the form
\[
E(G_p)_\Db = \begin{pmatrix} R_\Db & B \\ C & R_\Db \end{pmatrix}
\]
and consider the affine scheme $\Rep^\square_\Ad(E(G_p)_\Db, \cE)$ of adapted representations (Definition \ref{defn:adapted}), whose coordinate ring is generated by $B$ and $C$ over $R_\Db$. Those representations which are ordinary and upper triangular are cut out by the ideal generated by $(B,\rho_{22}(I_p)-1)$, where $\rho_{22}(I_p)-1$ represents non-trivial image of inertia in the lower right coordinate. Likewise, if $\bar\chi\vert_{I_p} = 1$, it is also possible to be ordinary and lower triangular, and this is cut out by the ideal $(C,\rho_{11}(I_p)-1)$. The ordinary locus is therefore cut out by the ideal $(B,\rho_{22}(I_p)-1) \cap (C, \rho_{11}(I_p)-1)$. Using Proposition \ref{prop:adapted_GMA}, the image of this closed subscheme of $\Rep^\square_\Ad(E(G_p)_\Db, \cE)$ in $\Rep_{G_p, \Db}$ is the locus of ordinary representations of $G_p$ with residual pseudorepresentation $\Db$.
\end{proof}

We may now define the moduli stack $\Rep^{\ord}_\Db$ of $G_{\bQ,S}$-representations that are ordinary at $p$ by setting
\[
\Rep^{\ord}_\Db := \Rep_\Db \times_{\Rep_{G_p, \Db}} \Rep_{G_p, \Db}^{\ord},
\]
just as in \eqref{eq:construct_Rep_C}, and then construct the \emph{global ordinary pseudodeformation ring} $R^{\ord}_\Db$  by Definition \ref{defn:pseudo_C}. Because $\Db$ is multiplicity-free, $\psi$ is a good moduli space (Thm.\ \ref{thm:PsR_background}(4)). The restriction of $\psi$ to $\Rep^{\ord}_\Db \ra \Spec R^{\ord}_\Db$ is a good moduli space as well (Thm.\ \ref{thm:alper}(7)), and $R^{\ord}_\Db$ is precisely the associated GIT quotient ring.
\begin{cor}
Let $E$ be a $p$-adic field with ring of integers $\cO$. With the data $\Db, \bar\psi, \bar\chi$ as above, choose a pseudorepresentation $D_z: G_{\bQ,S} \ra \cO \subset E$ deforming $\Db$, so that there is a corresponding morphism $z: \Spec E \ra \Spec R_\Db$. Then $z$ factors through $R^{\ord}_\Db$ if and only if $D_z$ is ordinary in the sense that the associated semi-simple representation $V^\mathrm{ss}_z$ is ordinary. 
\end{cor}
\begin{proof}
Combine Lemma \ref{lem:ord_space} and Theorem \ref{thm:global_main}.
\end{proof}

See \cite{WWE2015a} for a deformation-theoretic definition of ordinary pseudorepresentation. 

\small

\bibliographystyle{alpha}
\bibliography{CWEbib-1015}

\def\cprime{$'$} \def\Dbar{\leavevmode\lower.6ex\hbox to 0pt{\hskip-.23ex
  \accent"16\hss}D} \def\cfac#1{\ifmmode\setbox7\hbox{$\accent"5E#1$}\else
  \setbox7\hbox{\accent"5E#1}\penalty 10000\relax\fi\raise 1\ht7
  \hbox{\lower1.15ex\hbox to 1\wd7{\hss\accent"13\hss}}\penalty 10000
  \hskip-1\wd7\penalty 10000\box7}
  \def\cftil#1{\ifmmode\setbox7\hbox{$\accent"5E#1$}\else
  \setbox7\hbox{\accent"5E#1}\penalty 10000\relax\fi\raise 1\ht7
  \hbox{\lower1.15ex\hbox to 1\wd7{\hss\accent"7E\hss}}\penalty 10000
  \hskip-1\wd7\penalty 10000\box7} \def\Dbar{\leavevmode\lower.6ex\hbox to
  0pt{\hskip-.23ex \accent"16\hss}D}
  \def\cfac#1{\ifmmode\setbox7\hbox{$\accent"5E#1$}\else
  \setbox7\hbox{\accent"5E#1}\penalty 10000\relax\fi\raise 1\ht7
  \hbox{\lower1.15ex\hbox to 1\wd7{\hss\accent"13\hss}}\penalty 10000
  \hskip-1\wd7\penalty 10000\box7}
  \def\cftil#1{\ifmmode\setbox7\hbox{$\accent"5E#1$}\else
  \setbox7\hbox{\accent"5E#1}\penalty 10000\relax\fi\raise 1\ht7
  \hbox{\lower1.15ex\hbox to 1\wd7{\hss\accent"7E\hss}}\penalty 10000
  \hskip-1\wd7\penalty 10000\box7}
\begin{thebibliography}{WWE15b}

\bibitem[Alp13]{alper2013}
Jarod Alper.
\newblock Good moduli spaces for {A}rtin stacks.
\newblock {\em Ann. Inst. Fourier (Grenoble)}, 63(6):2349--2402, 2013.

\bibitem[Alp14]{alper2014}
Jarod Alper.
\newblock Adequate moduli spaces and geometrically reductive group schemes.
\newblock {\em Algebr. Geom.}, 1(4):489--531, 2014.

\bibitem[BC09a]{BC2009}
Jo{\"e}l Bella{\"{\i}}che and Ga{\"e}tan Chenevier.
\newblock Families of {G}alois representations and {S}elmer groups.
\newblock {\em Ast\'erisque}, (324):xii+314, 2009.

\bibitem[BC09b]{BCon2009}
Olivier Brinon and Brian Conrad.
\newblock {CMI} summer school notes on p-adic {H}odge theory, 2009.
\newblock \url{http://math.stanford.edu/~conrad/papers/notes.pdf}.

\bibitem[Bel16]{bellovin2014}
Rebecca Bellovin.
\newblock Generic smoothness for {$G$}-valued potentially semi-stable
  deformation rings.
\newblock {\em Ann. Inst. Fourier (Grenoble)}, 66(6):2565--2620, 2016.

\bibitem[BL95]{BL1995}
Arnaud Beauville and Yves Laszlo.
\newblock Un lemme de descente.
\newblock {\em C. R. Acad. Sci. Paris S\'er. I Math.}, 320(3):335--340, 1995.

\bibitem[Bre00]{breuil2000}
Christophe Breuil.
\newblock Groupes {$p$}-divisibles, groupes finis et modules filtr\'es.
\newblock {\em Ann. of Math. (2)}, 152(2):489--549, 2000.

\bibitem[Che13]{chen2013}
Ga\"{e}tan Chenevier.
\newblock {\em Repr\'{e}sentations {G}aloisiennes automorphes et
  cons\'{e}quences arithm\'{e}tiques des conjectures de {L}anglands et
  {A}rthur}.
\newblock Habilitation, Paris XI, 2013.
\newblock \url{http://gaetan.chenevier.perso.math.cnrs.fr/hdr/HDR.pdf}.

\bibitem[Che14]{chen2014}
Ga\"{e}tan Chenevier.
\newblock The $p$-adic analytic space of pseudocharacters of a profinite group,
  and pseudorepresentations over arbitrary rings.
\newblock In {\em Automorphic Forms and {G}alois Representations: {V}ol.\ {I}},
  volume 414 of {\em London Mathematical Society Lecture Note Series}, pages
  221--285. Cambridge Univ. Press, Cambridge, 2014.

\bibitem[CV03]{CV2003}
S.~Cho and V.~Vatsal.
\newblock Deformations of induced {G}alois representations.
\newblock {\em J. Reine Angew. Math.}, 556:79--98, 2003.

\bibitem[dJ95]{dejong1995}
A.~J. de~Jong.
\newblock Crystalline {D}ieudonn\'e module theory via formal and rigid
  geometry.
\newblock {\em Inst. Hautes \'Etudes Sci. Publ. Math.}, (82):5--96 (1996),
  1995.

\bibitem[EG15]{EG2015}
Matthew Emerton and Toby Gee.
\newblock ``{S}cheme-theoretic images'' of morphisms of stacks.
\newblock arXiv:1506.06146v2 [math.NT], 2015.

\bibitem[Fon90]{fontaine1990}
Jean-Marc Fontaine.
\newblock Repr\'esentations {$p$}-adiques des corps locaux. {I}.
\newblock In {\em The {G}rothendieck {F}estschrift, {V}ol.\ {II}}, volume~87 of
  {\em Progr. Math.}, pages 249--309. Birkh\"auser Boston, Boston, MA, 1990.

\bibitem[Fon94]{fontaine1994}
Jean-Marc Fontaine.
\newblock Repr\'esentations {$p$}-adiques semi-stables.
\newblock {\em Ast\'erisque}, (223):113--184, 1994.
\newblock With an appendix by Pierre Colmez, P{\'e}riodes $p$-adiques
  (Bures-sur-Yvette, 1988).

\bibitem[GM78]{GM1978}
Silvio Greco and Maria~Grazia Marinari.
\newblock Nagata's criterion and openness of loci for {G}orenstein and complete
  intersection.
\newblock {\em Math. Z.}, 160(3):207--216, 1978.

\bibitem[Gro60]{ega1}
A.~Grothendieck.
\newblock \'{E}l\'ements de g\'eom\'etrie alg\'ebrique. {I}. {L}e langage des
  sch\'emas.
\newblock {\em Inst. Hautes \'Etudes Sci. Publ. Math.}, (4):228, 1960.

\bibitem[Gro61a]{ega2}
A.~Grothendieck.
\newblock \'{E}l\'ements de g\'eom\'etrie alg\'ebrique. {II}. \'{E}tude globale
  \'el\'ementaire de quelques classes de morphismes.
\newblock {\em Inst. Hautes \'Etudes Sci. Publ. Math.}, (8):222, 1961.

\bibitem[Gro61b]{ega3-1}
A.~Grothendieck.
\newblock \'{E}l\'ements de g\'eom\'etrie alg\'ebrique. {III}. \'{E}tude
  cohomologique des faisceaux coh\'erents. {I}.
\newblock {\em Inst. Hautes \'Etudes Sci. Publ. Math.}, (11):167, 1961.

\bibitem[Gro64]{ega4-1}
A.~Grothendieck.
\newblock \'{E}l\'ements de g\'eom\'etrie alg\'ebrique. {IV}. \'{E}tude locale
  des sch\'emas et des morphismes de sch\'emas. {I}.
\newblock {\em Inst. Hautes \'Etudes Sci. Publ. Math.}, (20):259, 1964.

\bibitem[Gro66]{ega4-3}
A.~Grothendieck.
\newblock \'{E}l\'ements de g\'eom\'etrie alg\'ebrique. {IV}. \'{E}tude locale
  des sch\'emas et des morphismes de sch\'emas. {III}.
\newblock {\em Inst. Hautes \'Etudes Sci. Publ. Math.}, (28):255, 1966.

\bibitem[Gro67]{ega4-4}
A.~Grothendieck.
\newblock \'{E}l\'ements de g\'eom\'etrie alg\'ebrique. {IV}. \'{E}tude locale
  des sch\'emas et des morphismes de sch\'emas {IV}.
\newblock {\em Inst. Hautes \'Etudes Sci. Publ. Math.}, (32):361, 1967.

\bibitem[GS06]{GS2006}
Philippe Gille and Tam{\'a}s Szamuely.
\newblock {\em Central simple algebras and {G}alois cohomology}, volume 101 of
  {\em Cambridge Studies in Advanced Mathematics}.
\newblock Cambridge University Press, Cambridge, 2006.

\bibitem[GZB15]{GZB2015}
Anton Geraschenko and David Zureick-Brown.
\newblock Formal {GAGA} for good moduli spaces.
\newblock {\em Algebr. Geom.}, 2(2):214--230, 2015.

\bibitem[HH13]{HH2013}
Urs Hartl and Eugen Hellmann.
\newblock The universal family of semi-stable $p$-adic {G}alois
  representations.
\newblock arXiv:1312.6371v1 [math.NT], 2013.

\bibitem[Kis06]{crfc}
Mark Kisin.
\newblock Crystalline representations and {$F$}-crystals.
\newblock In {\em Algebraic geometry and number theory}, volume 253 of {\em
  Progr. Math.}, pages 459--496. Birkh\"auser Boston, Boston, MA, 2006.

\bibitem[Kis08]{pssdr}
Mark Kisin.
\newblock Potentially semi-stable deformation rings.
\newblock {\em J. Amer. Math. Soc.}, 21(2):513--546, 2008.

\bibitem[Kis09a]{fmc}
Mark Kisin.
\newblock The {F}ontaine-{M}azur conjecture for {${\rm GL}_2$}.
\newblock {\em J. Amer. Math. Soc.}, 22(3):641--690, 2009.

\bibitem[Kis09b]{mffgsm}
Mark Kisin.
\newblock Moduli of finite flat group schemes, and modularity.
\newblock {\em Ann. of Math. (2)}, 170(3):1085--1180, 2009.

\bibitem[Kra82]{kraft1982}
Hanspeter Kraft.
\newblock Geometric methods in representation theory.
\newblock In {\em Representations of algebras ({P}uebla, 1980)}, volume 944 of
  {\em Lecture Notes in Math.}, pages 180--258. Springer, Berlin, 1982.

\bibitem[LB08]{lebruyn2008}
Lieven Le~Bruyn.
\newblock {\em Noncommutative geometry and {C}ayley-smooth orders}, volume 290
  of {\em Pure and Applied Mathematics (Boca Raton)}.
\newblock Chapman \& Hall/CRC, Boca Raton, FL, 2008.

\bibitem[LB12]{lebruyn2012}
Lieven Le~Bruyn.
\newblock Representation stacks, {D}-branes and noncommutative geometry.
\newblock {\em Comm. Algebra}, 40(10):3636--3651, 2012.

\bibitem[Lev16]{levin2014}
Brandon Levin.
\newblock Local models for {W}eil-restricted groups.
\newblock {\em Compos. Math.}, 152(12):2563--2601, 2016.

\bibitem[LM85]{LM1985}
Alexander Lubotzky and Andy~R. Magid.
\newblock Varieties of representations of finitely generated groups.
\newblock {\em Mem. Amer. Math. Soc.}, 58(336):xi+117, 1985.

\bibitem[LMB00]{lmb}
G{\'e}rard Laumon and Laurent Moret-Bailly.
\newblock {\em Champs alg\'ebriques}, volume~39 of {\em Ergebnisse der
  Mathematik und ihrer Grenzgebiete. 3. Folge. A Series of Modern Surveys in
  Mathematics [Results in Mathematics and Related Areas. 3rd Series. A Series
  of Modern Surveys in Mathematics]}.
\newblock Springer-Verlag, Berlin, 2000.

\bibitem[Maz89]{mazur1989}
B.~Mazur.
\newblock Deforming {G}alois representations.
\newblock In {\em Galois groups over {${\bf Q}$} ({B}erkeley, {CA}, 1987)},
  volume~16 of {\em Math. Sci. Res. Inst. Publ.}, pages 385--437. Springer, New
  York, 1989.

\bibitem[MR01]{MR2001}
J.~C. McConnell and J.~C. Robson.
\newblock {\em Noncommutative {N}oetherian rings}, volume~30 of {\em Graduate
  Studies in Mathematics}.
\newblock American Mathematical Society, Providence, RI, revised edition, 2001.
\newblock With the cooperation of L. W. Small.

\bibitem[Nys96]{nyssen1996}
Louise Nyssen.
\newblock Pseudo-repr\'esentations.
\newblock {\em Math. Ann.}, 306(2):257--283, 1996.

\bibitem[Pro73]{procesi1973}
Claudio Procesi.
\newblock {\em Rings with polynomial identities}.
\newblock Marcel Dekker Inc., New York, 1973.
\newblock Pure and Applied Mathematics, 17.

\bibitem[Pro87]{procesi1987}
Claudio Procesi.
\newblock A formal inverse to the {C}ayley-{H}amilton theorem.
\newblock {\em J. Algebra}, 107(1):63--74, 1987.

\bibitem[PZ13]{PZ2013}
G.~Pappas and X.~Zhu.
\newblock Local models of {S}himura varieties and a conjecture of {K}ottwitz.
\newblock {\em Invent. Math.}, 194(1):147--254, 2013.

\bibitem[Ram93]{ramakrishna1993}
Ravi Ramakrishna.
\newblock On a variation of {M}azur's deformation functor.
\newblock {\em Compositio Math.}, 87(3):269--286, 1993.

\bibitem[Ric16]{richarz2015}
Timo Richarz.
\newblock Affine {G}rassmannians and geometric {S}atake equivalences.
\newblock {\em Int. Math. Res. Not. IMRN}, (12):3717--3767, 2016.

\bibitem[Rob63]{roby1}
Norbert Roby.
\newblock Lois polynomes et lois formelles en th\'eorie des modules.
\newblock {\em Ann. Sci. \'Ecole Norm. Sup. (3)}, 80:213--348, 1963.

\bibitem[Rob80]{roby2}
Norbert Roby.
\newblock Lois polyn\^omes multiplicatives universelles.
\newblock {\em C. R. Acad. Sci. Paris S\'er. A-B}, 290(19):A869--A871, 1980.

\bibitem[Rou96]{rouquier1996}
Rapha{\"e}l Rouquier.
\newblock Caract\'erisation des caract\`eres et pseudo-caract\`eres.
\newblock {\em J. Algebra}, 180(2):571--586, 1996.

\bibitem[Sam09]{samoilov2009}
L.~M. Samo{\u\i}lov.
\newblock An analogue of the {L}evitzki theorem for infinitely generated
  associative algebras.
\newblock {\em Mat. Zametki}, 86(1):151--153, 2009.

\bibitem[Sch08]{schoutens2008}
Hans Schoutens.
\newblock Pure subrings of regular rings are pseudo-rational.
\newblock {\em Trans. Amer. Math. Soc.}, 360(2):609--627 (electronic), 2008.

\bibitem[Ser58]{Serre1958}
{\em S\'eminaire {C}. {C}hevalley; 2e ann\'ee: 1958. {A}nneaux de {C}how et
  applications}.
\newblock Secr\'etariat math\'ematique, 11 rue Pierre Curie, Paris, 1958.

\bibitem[Ser79]{serre1979}
Jean-Pierre Serre.
\newblock {\em Local fields}, volume~67 of {\em Graduate Texts in Mathematics}.
\newblock Springer-Verlag, New York-Berlin, 1979.
\newblock Translated from the French by Marvin Jay Greenberg.

\bibitem[sga71]{sga1}
{\em Rev\^etements \'etales et groupe fondamental}.
\newblock Springer-Verlag, Berlin-New York, 1971.
\newblock S{\'e}minaire de G{\'e}om{\'e}trie Alg{\'e}brique du Bois Marie
  1960--1961 (SGA 1), Dirig{\'e} par Alexandre Grothendieck. Augment{\'e} de
  deux expos{\'e}s de M. Raynaud, Lecture Notes in Mathematics, Vol. 224.

\bibitem[SW99]{SW1999}
C.~M. Skinner and A.~J. Wiles.
\newblock Residually reducible representations and modular forms.
\newblock {\em Inst. Hautes \'Etudes Sci. Publ. Math.}, (89):5--126 (2000),
  1999.

\bibitem[Tay91]{taylor1991}
Richard Taylor.
\newblock Galois representations associated to {S}iegel modular forms of low
  weight.
\newblock {\em Duke Math. J.}, 63(2):281--332, 1991.

\bibitem[Vak14]{vakil2014}
Ravi Vakil.
\newblock Math 216: {F}oundations of algebraic geometry.
\newblock Version of December 30, 2014,
  \url{http://math.stanford.edu/~vakil/216blog/}, 2014.

\bibitem[Val76]{valabrega1976}
Paolo Valabrega.
\newblock A few theorems on completion of excellent rings.
\newblock {\em Nagoya Math. J.}, 61:127--133, 1976.

\bibitem[WE13]{WEthesis}
Carl Wang-Erickson.
\newblock {\em Moduli of {G}alois Representations}.
\newblock PhD thesis, Harvard University, April 2013.
\newblock Available at \url{http://dash.harvard.edu/handle/1/11108709}.

\bibitem[Wil88]{wiles1988}
A.~Wiles.
\newblock On ordinary {$\lambda$}-adic representations associated to modular
  forms.
\newblock {\em Invent. Math.}, 94(3):529--573, 1988.

\bibitem[WWE15a]{WWE2015a}
Preston Wake and Carl Wang-Erickson.
\newblock Ordinary pseudorepresentations and modular forms.
\newblock arXiv:1510.01661v4 [math.NT], To appear in \textit{Proc.\ Amer.\
  Math.\ Soc.}, 2015.

\bibitem[WWE15b]{WWE2015}
Preston Wake and Carl Wang-Erickson.
\newblock Pseudo-modularity and {I}wasawa theory.
\newblock arXiv:1505.05128v3 [math.NT], To appear in \textit{Amer.\ J.\ Math.},
  2015.

\end{thebibliography}

\end{document}